\documentclass[usenames,dvipsnames,11pt]{amsart}
\usepackage[margin=1in]{geometry}
\usepackage{xcolor}
\usepackage{amsmath}
\usepackage{amssymb}
\usepackage{amsfonts}
\usepackage{enumerate}
\usepackage{amsthm}
\usepackage{microtype}
\usepackage{mathrsfs}
\usepackage{xspace}
\usepackage{tikz}
\usepackage{tikz-cd}
\usepackage{xspace}
\usepackage{soul}
\usepackage{fancyhdr}

\usepackage[backend=biber, style=alphabetic]{biblatex}
\renewbibmacro{in:}{}
\theoremstyle{definition}

\newtheorem*{theorem*}{Theorem}
\newtheorem{theorem}{Theorem}
\newtheorem{prop}[theorem]{Proposition}
\newtheorem{ex}[theorem]{Example}
\newtheorem{lem}[theorem]{Lemma}
\newtheorem*{remark}{Remark}
\newtheorem*{conventions}{Conventions}
\newtheorem{dfn}[theorem]{Definition}
\newtheorem{cor}[theorem]{Corollary}

\numberwithin{theorem}{section}

\newcommand{\lcm}{\textrm{LCM}}
\newcommand{\N}{\mathbb{N}}
\newcommand{\Z}{\mathbb{Z}}
\newcommand{\Q}{\mathbb{Q}}

\newcommand{\B}{\mathcal{B}}
\newcommand{\D}{\mathbf{D}}
\newcommand{\mD}{\mathcal{D}}
\newcommand{\W}{\mathcal{W}}
\newcommand{\F}{\mathbb{F}}
\newcommand{\M}{\mathcal{M}}
\newcommand{\mR}{\mathcal{R}}
\newcommand{\mF}{\mathcal{F}}
\newcommand{\mL}{\mathcal{L}}
\newcommand{\G}{\mathcal{G}}
\newcommand{\mT}{\mathcal{T}}
\newcommand{\mP}{\mathcal{P}}
\newcommand{\cS}{\mathcal{S}}

\newcommand{\C}{\mathcal{C}}
\newcommand{\A}{\mathbf{A}}
\newcommand{\FF}{\mathbf{F}}
\newcommand{\DD}{\mathbf{D}}
\newcommand{\RR}{\mathbf{R}}
\newcommand{\TT}{\mathbf{T}}
\newcommand{\BB}{\mathbf{B}}
\newcommand{\ff}{\mathfrak{f}}
\newcommand{\gr}{\textrm{gr}}

\newcommand{\HS}{\textrm{HS}}
\newcommand{\inn}{\textrm{in}}
\newcommand{\Syz}{\textrm{Syz}}
\newcommand{\Coker}{\textrm{Coker}}
\newcommand{\Ker}{\textrm{Ker}}
\newcommand{\mf}[1]{\mathfrak{#1}}
\newcommand{\cb}[1]{{\color{blue}{#1}}}
\newcommand{\smprod}{\textstyle\prod}
\newcommand{\ncF}{\langle F\rangle}

\newcommand{\ord}{\text{ord}_f}

\newcommand{\ordf}{\text{ord}_f}
\newcommand{\ordff}{\text{ord}_{\ff}}
\def\vect#1#2{{#1}_1, \, \ldots, \, {#1}_{#2}}

\def\bdiag {\begin{center}\begin{tikzcd}}
\def\bmatdiag {\begin{center}\begin{tikzcd}[ampersand replacement = \&]}
\def\ediag {\end{tikzcd}\end{center}}
\def\bmat {\begin{bmatrix}}
\def\emat {\end{bmatrix}}
\def\ben {\begin{enumerate}[(i)]}
\def\een {\end{enumerate}}

\makeatletter
\def\parsept#1#2#3{%
    \def\nospace##1{\zap@space##1 \@empty}%
    \def\rawparsept(##1,##2){%
        \edef#1{\nospace{##1}}%
        \edef#2{\nospace{##2}}%
    }
    \expandafter\rawparsept#3%
}
\makeatother      
\fancypagestyle{firstpage}{%
  \fancyhead{}
  
  \fancyfoot[L]
}

\title{F-Holonomic F-Modules}
\author{Monica Lewis}\thanks{This material is based upon work supported by the National Science Foundation under Award No. DMS \#2103207}
\address{Department of Mathematics, University of Minnesota, Minneapolis, MN 55455, USA}
\email{lewi1714@umn.edu}


\date{}
\begin{document}
\thispagestyle{firstpage}
\maketitle
 \begin{abstract}
Let $R=\F_p[x_1,\ldots,x_n]$ and let $\FF$ be the ring of Frobenius operators over $R$. We introduce a notion of Bernstein dimension and multiplicity for the class of finitely generated $\FF$-modules whose structure morphism has a finite length kernel. We show that an $\FF$-module belongs to this class if and only if it admits a ``great'' filtration with respect to the Bernstein filtration on $\FF$. We describe the Hilbert series of these great filtrations, and prove that the dimension and multiplicity defined in terms of this Hilbert series are independent of the choice of filtration. We refer to the $\FF$-modules of Bernstein dimension $0$ as \textit{$F$-holonomic}. We show that $F$-holonomic $\FF$-modules are a full abelian subcategory of $\text{Mod}_\FF$, closed under taking extensions, on which multiplicity is an additive function. We show that Lyubeznik's finitely generated unit $\FF$-modules are $F$-holonomic.
 \end{abstract}

 \section{Introduction}
Fix a prime integer $p>0$ and let $R=\F_p[\vect x n]$. The \textit{ring of Frobenius operators} over $R$ is 
\[
\FF:=\frac{R\{F\}}{(Fr-r^pF\,|\,r\in R)} = \frac{\F_p\{\vect x n,F\}}{(x_ix_j-x_jx_i,\, x_i^pF-Fx_i\,|\,1\leq i,j\leq n)}
\]
 For $n\geq 1$, the ring $\FF$ is noncommutative and neither left nor right Noetherian. A \textit{word} in $\FF$ is any finite string involving the symbols $\vect x n, F$. Over $R=\F_2[x]$, the words $xxxxxFF$, $xFxxF$, and $xFFx$ all represent the same monomial in $\FF$. The \textit{complexity} of a monomial, is the length of the shortest word representing it. For example, if $p=2$, the complexity of the monomial $x^5F^2$ is $4$. 

\begin{dfn}
  The \textit{Bernstein filtration} on $\FF$ is the chain of $\F_p$-vector spaces $\B_0\subseteq\B_1\subseteq\B_2\subseteq\cdots$ in which $\B_i$ denotes the $\F_p$-linear span of all monomials of complexity at most $i$. Let $\B_i=0$ if $i<0$.
\end{dfn}

We shall be especially interested in the associated graded ring of $(\FF,\B)$, 
\[\A:=\gr_{\B}(\FF) = \bigoplus_i \B_i/\B_{i-1}\]
For $g\in \B_i$, let $[g]_i$ denote the image of $g$ in $\B_i/\B_{i-1}$. Multiplication in $\A$ is defined on homogeneous elements by $[g]_i[h]_j=[gh]_{i+j}$. For $n\geq 1$, the ring $\A$ is still noncommutative ($[x]_1[F]_1 \neq [F]_1[x]_1$) and neither left nor right Noetherian. It is also not a domain: $[x]_1[x^{p-1}F]_p=[x^pF]_{p+1}=[Fx]_{p+1}=0$. Presented as an abstract graded $\F_p$-algebra, we have
\[
\A =  \frac{\F_p\{\vect x n,f\}}{(x_ix_j-x_jx_i,\, x_i^pf\,|\,1\leq i,j\leq n)}
\]
where $\deg(x_i)=\deg(f)=1$ for all $i$, where $R$ may be regarded as a graded subalgebra of $\A$.

\begin{ex} Let $\FF=\F_2[x]\ncF$. One may directly compute the first few terms of the Hilbert series of $\A=\gr_{\B}(\FF)$, obtaining
\[
\HS_{\A}(t)= 1+2t+4t^2+7t^3+12t^4+20t^5+33t^6+54t^7+88t^8+143t^9+\cdots
\]
\begin{figure}[h]
  \begin{center}
   % \vspace{-6.0em}
    \begin{tikzpicture}[scale=0.9, every node/.style={scale=1.0}]
\makeatletter
%      \clip (-7.2,-7.2) rectangle (6.8,6);

%\node at (1,0) {\begin{tabular}{|c||c|c|c|c|c|c|c|c|c|}\hline cx & 0 & 1 & 2 & 3 & 4 & 5 & 6 & 7 & 8\\\hline
%                        mon.ct. & 1 & 2 & 4 & 7 & 12 & 20 & 33 & 54 & 88\\\hline\end{tabular}};
      \draw[step=1.0,thin,white,opacity=0.4] (-7,-7) grid (9,-1);
      \node at (-7+0.5,-7+0.5) {\footnotesize$1$};
      %\node at (-5,3) {\large$H^2_{(x,y)}(K[x,y])$};
      \foreach \p in {(0,0)}{
        \parsept{\x}{\y}{\p};
        \fill[green,opacity=0.4] (-7+\x,-7+\y)--(-6+\x,-7+\y)--(-6+\x,-6+\y)--(-7+\x,-6+\y)--(-7+\x,-7+\y);
        };
      \foreach \p in {(0,1),(1,0)}{
        \parsept{\x}{\y}{\p};
        \fill[teal,opacity=0.4] (-7+\x,-7+\y)--(-6+\x,-7+\y)--(-6+\x,-6+\y)--(-7+\x,-6+\y)--(-7+\x,-7+\y);
        };
      \foreach \p in {(0,2),(1,1),(2,1),(2,0)}{
        \parsept{\x}{\y}{\p};
        \fill[red,opacity=0.4] (-7+\x,-7+\y)--(-6+\x,-7+\y)--(-6+\x,-6+\y)--(-7+\x,-6+\y)--(-7+\x,-7+\y);
        };
      \foreach \p in {(0,3),(1,2),(2,2),(4,2),(3,1),(4,1),(3,0)}{
        \parsept{\x}{\y}{\p};
        \fill[magenta,opacity=0.4] (-7+\x,-7+\y)--(-6+\x,-7+\y)--(-6+\x,-6+\y)--(-7+\x,-6+\y)--(-7+\x,-7+\y);
        };
      \foreach \p in {(0,4),(1,3),(2,3),(4,3),(8,3),(3,2),(5,2),(6,2),(8,2),(5,1),(6,1),(4,0)}{
        \parsept{\x}{\y}{\p};
        \fill[blue,opacity=0.4] (-7+\x,-7+\y)--(-6+\x,-7+\y)--(-6+\x,-6+\y)--(-7+\x,-6+\y)--(-7+\x,-7+\y);
        };
      \foreach \p in {(0,5),(1,4),(2,4),(4,4),(8,4),(16,4),(3,3),(5,3),(6,3),(9,3),(10,3),(12,3),(16,3),(7,2),(9,2),(10,2),(12,2),(7,1),(8,1),(5,0)}{
        \parsept{\x}{\y}{\p};
        \fill[Dandelion,opacity=0.4] (-7+\x,-7+\y)--(-6+\x,-7+\y)--(-6+\x,-6+\y)--(-7+\x,-6+\y)--(-7+\x,-7+\y);
      };

            \foreach \x in {1,2,...,16}{
        \foreach \y in {1,2,...,5}{
          \node at (-7+\x+0.5,-7+\y+0.5) {\footnotesize$x^{\x}F^{\y}$};
          %\node at (0-\x+\y,6-\x-\y) {\footnotesize$\left[\tfrac{1}{x^{\x}y^{\y}}\right]$};
        };
      };

      \foreach \x in {1,2,...,16}{
        \node at (-7+\x+0.5,-7+0.5) {\footnotesize$x^{\x}$};
          %\node at (1-\x,5-\x) {\footnotesize$\left[\tfrac{1}{x^{\x}y}\right]$};
      };
      \foreach \x in {1,2,...,5}{
        \node at (-7+0.5,-7+\x+0.5) {\footnotesize$F^{\x}$};
%          \node at (-1+\x,5-\x) {\footnotesize$\left[\tfrac{1}{xy^{\x}}\right]$};
      };     
      \foreach \x in {0,1,...,6}{
        %\draw[thick] (-7,-2-2*\x)--(0+\x,5-\x);
        %\draw[ thick] (7,-2-2*\x)--(0-\x,5-\x);
        \draw[thick] (-7,-7+\x)--(10,-7+\x);
      };
      \foreach \x in {0,1,...,17}{
        \draw[thick] (-7+\x,-7)--(-7+\x,-1);
      };
    \end{tikzpicture}    
  \end{center}

  \caption{Various monomials $x^aF^b$ in $\FF=\F_2[x]\ncF$. All monomials with complexity ranging from $0$ to $5$ appear above, and are colored accordingly.}
\end{figure}

Let $[\A]_i$ denote the $i$th graded component of $\A$. We will show in the sequel that the linear recurrence relation $\dim_{\F_p}[\A]_i=2\cdot \dim_{\F_p}[\A]_{i-1}-\dim_{\F_p}[\A]_{i-3}$ holds for all $i\geq 1$. In other words, the coefficient of $t^i$ in $(1-2t+t^3)\HS_{(\FF,\B)}(t)$ vanishes for $i\geq 1$, and it follows at once that
\[
\HS_{\A}(t)=\frac{1}{t^3-2t+1}
\]
The denominator factors as $(1-t)(1-t-t^2)$.
\end{ex}

\begin{ex}\phantom{.} Below we give the Hilbert series of the associated graded ring $\A$ for various small values of $p$ and $n$.
\ben
\item The Hilbert series of $\gr_\B(\F_3[x]\ncF)$ is
  \[1 +2t+ 4t^2+ 8t^3+ 15t^4+ 28t^5+ 52t^6+ 96t^7+ 177t^8+ 326t^9+\cdots=\frac{1}{(1-t)(1-t(1+t+t^2))}
  \]
\item The Hilbert series of $\gr_\B(\F_5[x]\ncF)$ is
  \[1+ 2t+ 4t^2+ 8t^3+ 16t^4+ 32t^5+ 63t^6+ 124t^7+ \cdots=\frac{1}{(1-t)(1-t(1+t+t^2+t^3+t^4))}
  \]
\item The Hilbert series of $\gr_\B(\F_2[x,y]\ncF)$ is
  \begin{align*}
    1+ 3t+ 8t^2+ 19t^3+ 43t^4+ 95t^6+ 207t^7+ 448t^8+966t^9+\cdots=\frac{1}{(1-t)^2(1-t(1+t)^2)}
  \end{align*}

\item The Hilbert Series of $\gr_\B(\F_3[x,y]\ncF)$ is
  \[
  1+ 3t+ 8t^2+ 21t^3+ 53t^4+ 132t^3+ 327t^4+ 808t^5+ \cdots=\frac{1}{(1-t)^2(1-t(1+t+t^2)^2)}
  \]
\item The Hilbert series of $\gr_\B(\F_2[x,y,z]\ncF)$ is
  \[
  1+ 4t+ 13t^2+ 38t^3+ 105t^4+ 283t^5+ 753t^6+ 1991t^7+\cdots=\frac{1}{(1-t)^3(1-t(1+t)^3)}
  \]
\item The Hilbert series of $\gr_\B(\F_3[x,y,z]\ncF)$ is
  \[1+ 4t+ 13t^2+ 41t^3+ 126t^4+ 382t^5+1152t^6+  \cdots=\frac{1}{(1-t)^3(1-t(1+t+t^2)^3)}
  \]
  \een
    \end{ex}

That a linear recurrence relation holds on the number of Frobenius monomials of a fixed complexity was first observed by Dills and Enescu \cite{dillsenescu}. We state their result in the language of Hilbert series below.

\begin{theorem*}[Dills, Enescu; 2021]\label{deresult}
    Let $R=\F_p[\vect x n]$, let $\FF$ be the ring of Frobenius operators over $R$, let $\B$ be the Bernstein filtration on $\FF$, and let $\A=\gr_{\B}(\F)$. The Hilbert function of $\A$ has the form
    \[
    \HS_\A(t)=\frac{1}{(1-t)^ng_{p,n}(t)}
    \]
    where $g_{p,n}(t)$ is a polynomial of degree $n(p-1)+1$ depending only on $p$ and $n$. Specifically,
    \[
    g_{p,n}(t)=1-\textstyle\sum_{v\in\N_{<p}^n}t^{|v|+1} = 1-t(1+t+\cdots+t^{p-1})^n
    \]
\end{theorem*}

We will prove that all finitely presented graded $\A$-modules have a Hilbert series of the same general form as that of $\A$.

\begin{theorem*}\textbf{\ref{amodhilbert}}
  Let $M$ be a finitely presented graded left $\A$-module. The Hilbert series of $M$ is a rational function of the form
  \[
  \text{HS}_M(t) = \frac{a(t)}{(1-t)^dg_{p,n}(t)}  
  \]
  where $g_{p,n}(t)$ is the degree $n(p-1)+1$ polynomial of Dills and Enescu depending only on $p$ and $n$, and $0\leq d\leq n$.
\end{theorem*}

Our proof is based on Gr\"obner deformation in the ring $\A$. In Section 2, we study finitely generated left monomial ideals of $\A$ in sufficient detail to understand their Hilbert series, and equip the monoid of monomials with a term order. It is not hard to show that the Hilbert series of a graded submodule of a finite-rank free $\A$-module agrees with the Hilbert series of its initial submodule with respect to the term order. Much more difficult is establishing the following in absence of left Noetherianity.

\begin{theorem*}\textbf{\ref{INfg}}
  If $M$ is a finitely generated graded left submodule of a finite rank free $\A$-module, then the initial submodule $\inn(M)$ is finitely generated.
\end{theorem*}

To prove this, we need to directly investigate how the initial submodule can be constructed given a finite generating set for $M$. In Section 3, we prove that a version of Buchberger's algorithm holds for a ring like $\A$, and this algorithm is sufficient to establish the necessary finiteness claim.

Theorem \ref{amodhilbert} suggests a natural ``minimal'' class of graded $\A$-modules, namely, those with a Hilbert series of the form $a(t)/g_{p,n}(t)$. Of particular interest to us are those $\FF$-modules whose associated graded module belongs to the minimal class.

\begin{dfn}
A \textit{$\B$-compatible filtration} on a left $\FF$-module shall mean a chain $\Omega_0\subseteq \Omega_1\subseteq \Omega_2\subseteq\cdots$ of finite-dimensional $\F_p$-vector spaces $\Omega_i\subseteq M$ such that $\bigcup_i\Omega_i = M$ and $\B_i\cdot \Omega_j\subseteq \Omega_{i+j}$ for all $i$ and $j$. An \textit{$(\FF,\B)$-module} is a pair $(M,\Omega)$ consisting of an $\FF$-module $M$ equipped with a $\B$-compatible filtration $\Omega$. The \textit{associated graded $\A$-module} of $(M,\Omega)$ is
\[
\gr_{\Omega}(M)=\bigoplus_i \Omega_i/\Omega_{i-1}
\]
By the Hilbert series of $(M,\Omega)$, we shall always mean the Hilbert series of $\gr_{\Omega}(M)$, namely,
\[
\HS_{(M,\Omega)}(t)=\sum_{i=0}^\infty \dim_{\F_p}(\Omega_i/\Omega_{i-1})\,t^i
\]
\end{dfn}

Call a $\B$-compatible filtration $\Omega$ on and $\FF$-module $M$ \textit{great} if $\gr_\Omega(M)$ is a finitely presented $\A$-module. Suppose that a left $\FF$-module $M$ admits a great filtration $\Omega$. By Theorem \ref{amodhilbert}, the Hilbert series $\HS_{(M,\Omega)}(t)$ has the form $a(t)/(1-t)^dg_{p,n}(t)$. One may quite reasonably ask whether, given another great filtration $\Sigma$, the Hilbert series $\HS_{(M,\Sigma)}(t)$ has anything at all to do with $\HS_{(M,\Omega)}(t)$.

For a rational function $h(t)$, let $\delta_1(h)$ be the least nonnegative integer $d$ such that $(1-t)^dh(t)$ has no pole at $t=1$. We prove the following.
\begin{theorem*}\textbf{\ref{welldefined}}
If $\Sigma$ and $\Omega$ are two great filtrations on an $\FF$-module $M$, then $\delta_1(\HS_{(M,\Omega)})=\delta_1(\HS_{(M,\Sigma)})$. Letting $a_\Omega(t)=(1-t)^{\delta_1(\HS_{(M,\Omega)})}g_{p,n}(t)\HS_{(M,\Omega)}(t)$, with $a_\Sigma(t)$ defined analogously, we have $a_\Omega(1)=a_\Sigma(1)$.
\end{theorem*}

We call the values $\delta_1(\HS_{(M,\Omega)})$ and $a_\Omega(1)$ defined for some (equivalently, any) great filtration $\Omega$ the \textit{Bernstein dimension} and \textit{multiplicity} of $M$, respectively. The multiplicity of an $\FF$-module with positive Bernstein dimension is positive, but an $\FF$-module of Bernstein dimension $0$ can have either zero or negative multiplicity (see Example \ref{negmult}).

It is worth asking which $\FF$-modules actually admit a great filtration in the first place. We show that this class of modules actually has a rather simple description in terms of Lyubeznik's structure morphism $\theta_M$ (\cite{lyufmod}, see also \cite{emertonkisin}, \cite{blicklethesis}).

\begin{theorem*}\textbf{\ref{greatclass}}
  Let $M$ be an $\FF$-module with structure morphism $\theta_M$. The module $M$ admits a great $\B$-compatible filtration if and only if $M$ is finitely generated and $\Ker(\theta_M)$ is a finite length $R$-module. If this is the case, then $\Coker(\theta_M)$ is finitely generated over $R$, and for any great $\B$-compatible filtration $\Omega$ on $M$, we have $\delta_1(\HS_{(M,\Omega)})=\dim\Coker(\theta_M)$.
\end{theorem*}

A natural class that stands out to us are those finitely generated $\FF$-modules $M$ with a Bernstein dimension $0$. We refer to these modules as \textit{$F$-holonomic}, due to their characterization in terms of a minimality condition on the denominator of their Hilbert series. We can give a number of equivalent characterizations. All filtrations below are assumed to be $\B$-compatible.

\begin{theorem*}\textbf{\ref{hololist}}
  Let $M$ be a finitely generated $\FF$-module with structure morphism $\theta_M$. The following are equivalent.
  \begin{enumerate}[(i)]
      \item Both $\Ker(\theta_M)$ and $\Coker(\theta_M)$ are finite length $R$-modules.
  \item $M$ admits a great filtration with Bernstein dimension $0$.
  \item There exists a filtration $\Omega$ on $M$ such that both $\Ker(\gr_\Omega(\theta_M))$ and $\Coker(\gr_\Omega(\theta_M))$, suitably defined, are finite length $R$-modules.
  \item There exists a filtration $\Omega$ on $M$ such that $\gr_\Omega(M)$ has a finite Gr\"obner basis $\G=(g_1,\ldots,g_t)$ with $\inn(\Syz(\G))=I_1\oplus\cdots\oplus I_t$ finitely generated, and such that each contracted monomial ideal $I_i\cap R$ contains a power of $\mf{m}$.
  \item There exists a great filtration $\Omega$ on $M$ such that the Hilbert series of $\gr_\Omega(M)$ has the form $a(t)/g_{p,n}(t)$, for some polynomial $a(t)$.
  \end{enumerate}
  In every instance above, ``there exists a [...] filtration'' can be replaced with ``for all [...] filtrations, of which there is at least one.''
\end{theorem*}

It is evident from condition (i) that Lyubeznik's finitely generated unit $\FF$-modules are $F$-holonomic. In particular, any local cohomology module $H^i_I(R)$ is $F$-holonomic. Furthermore, $F$-holonomic $\FF$-modules are a full abelian subcategory of the category of all $\FF$-modules, closed under taking extensions.

\begin{theorem*}\textbf{\ref{extcat}}
Let $M$ and $N$ be $F$-holonomic $\FF$-modules. For any $\FF$-linear map $\varphi:M\to N$, the kernel, cokernel, and image of $\varphi$ are $F$-holonomic. An extension of two $F$-holonomic $\FF$-modules is $F$-holonomic. The multiplicity $M\mapsto e(M)$ is an additive function on the category of $F$-holonomic $\FF$-modules.
\end{theorem*}

\begin{remark}
We refer to our modules as $F$-holonomic in order to draw an analogy to the classical setting of holonomic $\D$-modules, which are defined in terms of a minimality condition on the Hilbert series of one (equivalently, all) great filtrations (which, over $\DD$, are the same as good filtrations) compatible with the Bernstein filtration on $\D$. A very different notion of \textit{holonomic} Frobenius modules, however, does exists in the literature, due to Bhatt and Lurie \cite{bhattlurie}. As their holonomic modules are perfect, and our $F$-holonomic modules are never perfect outside of trivial cases (e.g., $\F_p$ itself), we hope that there is little danger of confusion.
\end{remark}
\begin{remark}
  We sacrifice some amount of generality for the sake of cleaner exposition by restricting to the case of polynomial rings over $\F_p$. This allows us to talk about linear maps between graded components of $\A$-modules instead of $F$-sesquilinear maps. Compared to the case of polynomial rings over a general perfect field $K$, in which case the ring $K\ncF$ is still left Noetherian, this is primarily a matter of convenience than a serious technical obstruction. If $K$ is an imperfect field, especially an imperfect field that is not $F$-finite, the distinction is less trivial
\end{remark}
\begin{conventions}
  All rings are unital, and all ring homomorphisms preserve unity. By default, all ideals are left ideals, all modules are left modules, and Noetherianity refers to the ascending chain condition on left submodules. We will use boldface (e.g., $\A$, $\BB$, $\RR$) to denote rings that are not necessarily commutative, with standard Roman font (e.g., $A$, $B$, $R$) reserved for commutative rings. If $\TT$ is a graded ring and $M$ is a graded $\TT$-module, we write $[M]_i$ for the $i$th graded component of $M$. The module $M(j)$ denotes the Serre twist of $M$, in which $[M(j)]_i=M_{i+j}$. If $\RR$ is a ring, we let $\RR\{X\}$ denote the noncommutative polynomial ring in which a symbol $X$ is adjoined as freely as possible, satisfying no requirements beyond than those strictly necessary to make $\RR\{X\}$ into a ring.
\end{conventions}

\section{Frobenius monomials}
A product written as $(\smprod_{i=1}^n r_is)$ will always mean $r_1sr_2s\cdots r_ns$. We would instead write $(\smprod_{i=1}^n r_i)s$ if we wished to describe $r_1r_2\cdots r_n s$.

Let $R=\F_p[\vect x n]$. Consider the $R$-algebra obtained by adjoining a symbol $\ff$ as freely as possible.
\[
\BB = R\{\ff\} = \frac{\F_p\{x_1,\ldots,x_n,\ff\}}{(x_ix_j-x_jx_i\,|\,1\leq i,j\leq n)}
\]
We regard $\BB$ as an $\N$-graded ring with $\deg(x_i)=\deg(\ff)=1$ for all $i$. An element of $\BB$ having the form $m=(\smprod_{i=0}^{e-1} x^{v_i}\ff) x^{v_e}$ for $v_0,\ldots,v_e\in \N^n$ is a \textit{monomial} of $\BB$. The \textit{$\ff$-order} of $m$ is $\ordff(m):=e$.

Let $\M_n$ (or just $\M$, if $n$ is clear) denote the multiplicative monoid of monomials in $\BB$. This monoid is both left and right cancellative. It also admits several straightforward term orders. We will default to using the following.
\begin{dfn}\label{termorder}
 Let $m_1=(\smprod_{i=0}^{e-1} x^{v_i}\ff) x^{v_e}$ and $m_2=(\smprod_{i=0}^{d-1} x^{u_i}\ff) x^{u_d}$ be monomials in $\M_n$. The \textit{graded reverse lexicographic} term order declares $m_1<m_2$ if and only if one of the following conditions holds:
\begin{itemize}
\item $\deg(m_1)<\deg(m_2)$.
\item $\deg(m_1)=\deg(m_2)$ and $\ordff(m_1)<\ordff(m_2)$.
\item $\deg(m_1)=\deg(m_2)$, $\ordff(m_1)=\ordff(m_2)$, and $v_i=u_i$ for all $i<j$, with $v_j<u_j$ in the standard reverse lexicographic order on $\N^n$.
\end{itemize}
\end{dfn}
It is not hard to see that (i) $<$ is a well-order, (ii) $1$ is the least element of $\M_n$, (iii) for any monomials $m_1,m_2,m^\prime \in \M_n$, we have $m_1<m_2$ if and only if $m^\prime m_1<m^\prime m_2$ if and only if $m_1m^\prime < m_1m^\prime$, and (iv) for any $m_1$ and $m_2$, there are at most finitely many monomials $m^\prime$ such that $m_1<m^\prime<m_2$.

For $i\in \N$, let $\mP_i$ denote the $\F_p$-span of all monomials of degree at most $i$ in $\BB$, and for $m\in \M_n$, let $\mT_m$ denote the $\F_p$-span of all monomials $m^\prime$ that satisfy $m^\prime \leq m$. Both are exhaustive filtrations of $\BB$ by finite-dimensional vector spaces. It is clear that $\mP_i \mP_j\subseteq \mP_{i+j}$ and $\mT_m\mT_{m^\prime}\subseteq \mT_{mm^\prime}$ for all $i,j\in \N$ and all $m,m^\prime\in \M_n$. If $\mT_{<m}$ denotes the union $\bigcup_{m^\prime<m} \mT_{m^\prime}$, then $\mT_m/\mT_{<m}$ can be identified with the $\F_p$ span of the monomial $m$.

We will use $\BB$ primarily as a vehicle for studying the $R$-algebras
\[
\FF := \frac{R\{F\}}{(x_i^pF-Fx_i\,|\,1\leq i\leq n)}\hspace{1.0em}\text{and}\hspace{1.0em}
\A := \frac{R\{f\}}{(x_i^pf\,|\,1\leq i\leq n)}
\]
There are obvious $R$-algebra surjections $\varphi:\BB\twoheadrightarrow \FF$ via $\ff\mapsto F$ and $\psi:\BB\twoheadrightarrow \A$ via $\ff\mapsto f$. As the kernel of $\psi$ is homogeneous, the ring $\A$ naturally inherits a grading from $\BB$, while the ring $\FF$ does not. The $F$-order (resp. $f$-order) on $\FF$ (resp. $\A$), at least, is clearly well-defined. The \textit{Bernstein filtration} on $\FF$ is the image of the $\N$-indexed filtration $\mP$ under $\varphi$. We will postpone our study of the filtrations $\varphi(\mT)$ and $\psi(\mT)$ indexed by the monoid $\M_n$ until the next section.

A \textit{monomial} of $\FF$ or $\A$ is the image of a monomial of $\BB$ under $\varphi$ or $\psi$, respectively. While a monomial in $\FF$ may have many monomial preimages under $\varphi$, a monomial in $\A$ always has exactly one monomial preimage under $\psi$, and this allows us to place a kind of term order on $\A$.
\begin{dfn}\label{aatermorder}
  Let $a$ and $b$ be monomials in $\A$. Let $m_a$ and $m_b$ denote their unique monomial preimages under $\psi:\BB\twoheadrightarrow \A$. We define $a<b$ if and only if $m_a<m_b$ in the graded reverse lexicographic order on $\M_n$.
\end{dfn}
If $a<b$ in $\A$, and $c\in \A$ is any monomial, it is clear that whenever $ca\neq 0$ and $cb\neq 0$, we have $a<b$ if and only if $ca<cb$ if and only if $ac<bc$.

\begin{ex}
If $R=\F_2[x,y]$, the graded reverse lexicographic term order on $\A$ begins with
%3<1<2<0
%(3,3)<(3,1)<(3,2)<(3,0)<(1,3)
%n1<0
\[
1<x<y<f<x^2<xy<y^2<xf<yf<fx<fy<f^2
%[1]_0<[x]_1<[y]_1<[F]_1<[x^2]_2<[xy]_2<[y^2]_2<[xF]_2<[yF]_2<[Fx]_2<[Fy]_2<[F^2]_2%<x^3<x^2y<xy^2<y^3
\]
\end{ex}

\begin{ex}
  In the case $R=\F_2[x]$, the ordering on the set of monomials of degree $6$ and $f$-order $3$ in $\A$ is:
  \[
  xfxfxf<xfxffx<xffxfx<xfffxx<fxfxfx<fxffxx<ffxfxx<fffxxx
  \]
  These monomials in $\A$ are the images of monomials in $\FF$ under $\B_6\twoheadrightarrow \B_6/\B_5$. Letting $[m]_6$ denote the image of $m\in \B_6$ under this map, we have 
 \[
\hspace{-0.7em}    [x^{\cb{111}}F^3]_6<[x^{1\cb{011}}F^3]_6<[x^{1\cb{101}}F^3]_6<[x^{10\cb{001}}F^3]_6<[x^{1\cb{110}}F^3]_6<[x^{10\cb{010}}F^3]_6<[x^{10\cb{100}}F^3]_6<[x^{11\cb{000}}F^3]_6
%\hspace{3.0em}x^{0011}F^4<x^{0101}F^4<x^{1001}F^4<x^{0001}F^4x<x^{0110}F^4<x^{1010}F^4<x^{0010}F^4x<x^{1100}F^4<x^{0100}F^4x<x^{1000}F^4x<x^{0000}F^4x^{10}
\]
where all exponents are expressed in base $p=2$. The order is based on the three least significant bits of the exponent of $x$. In decimal, the ordering would read
\[
    [x^7F^3]_6<[x^{11}F^3]_6<[x^{13}F^3]_6<[x^{17}F^3]_6<[x^{14}F^3]_6<[x^{18}F^3]_6<[x^{20}F^3]_6<[x^{48}F^3]_6
%x^3F^4<x^5F^4<x^9F^4<x^{17}F^4<x^6F^4<x^{10}F^4<x^{18}F^4<x^{12}F^4<x^{20}F^4<x^{24}F^4<x^{32}F^4
    \]
In general, for fixed $e$, the images under $\B_i\twoheadrightarrow \B_i/\B_{i-1}$ of the monomials $x^aF^e$ of degree $i$ in $\F_p[x]\ncF$ will be ordered reverse lexicographically according to the $e$ least significant digits in the base $p$ expansion of $a$.
\end{ex}

A key ingredient in Gr\"obner degeneration is classifying monomial syzygies. A first observation to make is that it is not uncommon for a pair of monomials in $\BB$ to have no left syzygies whatsoever. We need not look any further than $\F_p[x]\{\ff\}$, where the monomials $x$ and $\ff$ have no common left multiple. The existence of a common left multiple for monomials of unequal $\ff$-order is actually quite restrictive.
\begin{lem}\label{monomialsyz}
Let $a$ and $b$ be two monomials in $\BB$. Suppose that $(\M\cdot a)\cap (\M\cdot b)\neq \varnothing$, say with $ca-db=0$ for $c,d\in \M$.
\ben
\item If $\ordff(a)\neq\ordff(b)$, then either $a =mb$ or $b=ma$ for some $m\in \M$. The vector $(c,-d)$ belongs to $\M\cdot (1,-m)$ or $\M\cdot (m,-1)$, respectively.
  \item If $\ordff(a)=\ordff(b)$, then there is a monomial $m$ such that $a=x^\alpha m$ and $b=x^\beta m$. In this case, the vector $(c,-d)$ belongs to $\M\cdot (x^{\ell-\alpha},-x^{\ell-\beta})$, where $x^\ell$ is the LCM of $x^{\alpha}$ and $x^\beta$.
\een
\end{lem}
\begin{proof}
  Write out
  \[
  ca = db = (\smprod_{i=0}^{e-1} x^{v_i}\ff)x^{v_e}
  \]
  The products $m_1=(\smprod_{i=e-\ordff(a)}^{e-1} x^{v_i}\ff)x^{v_e}$ and $m_2=(\smprod_{i=e-\ordff(b)}^{e-1} x^{v_i}\ff)x^{v_e}$ must contain $a$ and $b$, respectively, as right factors via $m_1=x^ra$ and $m_2=x^sb$ for some $r,s\in \N^n$. It follows that every right factor of $m_1$ (resp. $m_2$) of strictly lower $\ff$-order is a right factor of $a$ (resp. $b$). (i) If $\ordff(a)\neq \ordff(b)$, it is easily seen that one of $m_1$ or $m_2$ is a right factor of the other, and thus, one of $a$ or $b$ is a right factor of the other. The claim about the vector $(c,-d)$ follows at once. (ii) If $\ordff(a)= \ordff(b)=0$, the claim is obvious. Otherwise, let $m_0=(\smprod_{i=e-\ordff(a)+1}^{e-1} x^{v_i}\ff ) x^{v_e}$, and let $v\in\N^n$ be such that $m_1=m_2=x^v\ff m_0$. We have $x^ra=x^v\ff m_0=x^sb$, and hence, $a=x^{v-r}\ff m_0$ and $b=x^{v-s}\ff m_0$ where it is clear that $v_i\geq r_i,s_i$ for all $i$. The claim about $(c,-d)$ follows.
\end{proof}

Let $a$ and $b$ be monomials in $\A$. If they admit a nonzero left common multiple $ca=db$, then by taking the unique monomial lifts of $a,b,c,d$ to $\BB$, we can readily classify $(c,-d)$ using the above lemma. In addition to syzygies of this form, it is possible in $\A$ to have either $ca=0$ or $db=0$, giving $(c,0)$ or $(0,d)$ as a syzygy, respectively. If $ca=0$, then $a$ must have $f$-order at least $1$, say $a = x_1^{\alpha_1}\cdots x_n^{\alpha_n}f m$ for some monomial $m$. In this situation, syzygies of the form $(c,0)$ are multiples of $(x_i^{p-\alpha_i},0)$ for $1\leq i\leq n$, with an analogous statement for $(0,d)$.

We want to understand monomial Hilbert functions. In order to do so, it is helpful to first study the left-multiplication maps by monomials of the form $x^vf$. We write $\N_{<p}^n$ to denote all vectors $(v_1,\ldots,v_n)\in\N^n$ with $0\leq v_i<p$ for all $i$. A \textit{monomial ideal} of $\A$ will always mean a left ideal generated by monomials.

\begin{theorem}\label{prefixcapture}
  Let $I$ be a finitely generated monomial left ideal of $\A$. Fix an exponent vector $v\in \N_{<p}^n$. There exists a constant $k$ depending only on $I$ and $v$ such that, for any monomial $m$ of degree $\geq k$, having $x^{v}f m\in I$ implies that $m\in I$.
\end{theorem}
\begin{proof}
  First, it is clear that $x^{v}fm\neq 0$ for any nonzero monomial $m\in \A$. Let $a_1,\ldots,a_t$ be a set of monomial generators for $I$. If $x^{v}fm\in I$, then there is an $a_i$ such that $x^v fm= ca_i$ for some monomial $c$. From here, there are two meaningful cases, depending on $\ordf(a_i)$. The case $\ordf(a_i)=0$ is easy. If $a_i=x^r$ for some $r\in\N^n$, then $x^vfm=cx^r$ implies that $c=x^vfc_0$ for some $c_0$, giving $m=c_0x^r$, and thus, $m\in I$.

  Suppose now that $\ordf(a_i)\geq 1$. For a nonzero multiple $ca_i$ to have $\ordf(ca_i)\leq\ordf(a_i)+1$, it is necessary for $c$ to take the form $x^\alpha fx^\beta$ with $0\leq \alpha_i, \beta_i<p$ for all $i$. In particular, if $\deg(c)>2(p-1)n+1$, then we must have $\ordf(ca_i)>\ordf(a_i)+1$. Suppose $m$ has degree at least $2(p-1)n+1+\deg(a_i)$. Then from $x^v f m=xa_i$, we get
  \[
  \deg(c)=(|v|+1+\deg(m))-\deg(a_i)\geq 2(p-1)n+2
  \]
  Given $\ordf(c a_i)>\ordf(a_i)+1$ and $\ordf(m)=\ordf(ca_i)-1$, we have $\ordf(m)>\ordf(a_i)$. By Lemma \ref{monomialsyz}, if $\ordf(m)>\ordf(a_i)$ and $m$ and $a_i$ have a nonzero left common multiple, then $m$ is a left multiple of $a_i$, and thus, $m\in I$.
\end{proof}

If $I$ is a monomial ideal, then each graded component of the module $\A/I$ has an obvious choice of monomial basis over $\F_p$. Left multiplication by $x^vf$ is a graded map of degree $|v|+1$ over $\A$, allowing us to reinterpret the above theorem as follows.

\begin{cor}\label{bigtheta}
  Let $I$ be a finitely generated monomial ideal of $\A$. Let $\Theta_i$ be the map
  \[
  [\A/I]_i \xleftarrow{\Theta_i} \bigoplus_{v\in \N_{<p}^n} [\A/I]_{i-|v|-1}
  \]
  defined as the sum over all $v\in\N_{<p}^n$ of the component maps
  \[
    [\A/I]_i \xleftarrow{x^{v}f\cdot} [\A/I]_{i-|v|-1}.
    \]
    Each $\Theta_i$ is injective for $i\gg 0$. The image of $\Theta_i$ is the $\F_p$-span of (the images mod $I$ of) all degree $i$ monomials of $\A$ not belonging to $I$ that have $f$-order $\geq 1$. If we let $J=I\cap R$, then the cokernel of $\Theta_i$ coincides precisely with $[R/J]_i$. In other words, for $i\gg 0$, there is a short exact sequence
      \[
0\leftarrow [R/J]_i \leftarrow [\A/I]_i \xleftarrow{\Theta_i} \bigoplus_{v_0\in \N_{<p}^n} [\A/I]_{i-|v_0|-1}\leftarrow 0
  \]
\end{cor}
\begin{proof}
  Injectivity is immediate from Theorem \ref{prefixcapture}. The claim about the image of $\Theta_i$ is just the statement that if $m$ has $f$-order $\geq 1$, then $m=x^{v}fm_0$ for some $v\in\N_{<p}^n$ and some monomial $m_0$, where $m\not\in I$ obviously implies $m_0\not\in I$. That is to say, $m$ is the image of $m_0\in [\A/I]_{i-|v|-1}$ under the map
    \[
    [\A/I]_i \xleftarrow{x^{v}f\cdot} [\A/I]_{i-|v|-1}.
    \]
The definition of $\Theta_i$ makes it clear that no monomials of $f$-order $0$ can possibly lie in $\text{Im}(\Theta_i)$. A basis for the cokernel of $\Theta_i$ is therefore given by the set of degree $i$ monomials of $R=\F_p[x_1,\ldots,x_n]$ that do not belong to $J=I\cap R$.
\end{proof}

  \begin{theorem}\label{monomialhilbertseries}
    Let $I$ be a finitely generated monomial ideal of $\A$. The Hilbert series of $\A/I$ is a rational function of the form
    \[
    \HS_{\A/I}(t) =\frac{a(t)}{(1-t)^dg_{p,n}(t)}
    \]
    where $d$ is the Krull dimension of the monomial ideal $J=I\cap R$ of $R$ (in particular, $0\leq d\leq n$), and $g_{p,n}(t)=1-\sum_{v\in\N_{<p}^n}t^{|v|+1}$. If $d>0$, then $a(1)>0$.
    \end{theorem}
  \begin{proof}
For each $i$, we have an exact sequence
        \[
0\leftarrow [R/J]_i \leftarrow [\A/I]_i \xleftarrow{\Theta_i} \bigoplus_{v\in \N_{<p}^n} [\A/I]_{i-|v|-1}\leftarrow [K]_i\leftarrow 0
\]
in which $[K]_i=\Ker(\Theta_i)=0$ for $i\gg 0$. We obtain
\[
\HS_{\A/I}(t) = -b(t)+\HS_{R/J}(t)+\sum_{v\in\N_{<p}^n}t^{|v|+1}\HS_{\A/I}(t) 
\]
where $b(t)=\HS_{K}(t)$ is a polynomial whose coefficients are nonnegative integers. Of course, $\HS_{R/J}(t)$ has the form $c(t)/(1-t)^d$ where $d=\dim(R/J)$ and $c(t)$ is a polynomial such that $c(1)>0$. Let
\[
g_{p,n}(t) = 1-\sum_{v\in\N_{<p}^n} t^{|v|+1}
\]
and we see that
\[
\HS_{A/I}(t)=\frac{-b(t)(1-t)^d+c(t)}{(1-t)^dg_{p,n}(t)}
\]
Let $a(t)=-b(t)(1-t)^d+c(t)$. If $d\geq 1$, then $a(1)=c(1)>0$. If $d=0$, then $a(1)=-b(1)+c(1)$, and we have no such guarantee.
  \end{proof}
  \begin{dfn}\label{gpndef}
    Fix a prime $p$ and a natural number $n\geq 1$. Let $g_{p,n}(t)$ be the polynomial
    \[
    g_{p,n}(t) = 1-\textstyle\sum_{v\in\N_{<p}^n}t^{|v|+1}=1-t(1+t+\cdots+t^{p-1})^n
    \]
\end{dfn}  
  We can use Theorem \ref{monomialhilbertseries} to prove the following result of Dills-Enescu \cite{dillsenescu}.
    \begin{cor}
      Let $R=\F_p[x_1,\ldots,x_n]$ and $\A=R\{f\}/(x_i^pf\,|\,1\leq i\leq n)$. Then
      \[
      \HS_\A(t) = \frac{1}{(1-t)^ng_{p,n}(t)}
        \]
    \end{cor}
    \begin{proof}
      Left multiplication by $x^v f$ on $\A$ is injective in all degrees, and the contraction of the $0$ ideal from $\A$ to $R$ is $0$. The result follows from the homogeneous short exact sequence
      \[
      0\leftarrow [R]_i \leftarrow [\A]_i \leftarrow \bigoplus_{v\in \N_{<p}^n} [\A]_{i-|v|-1}\leftarrow 0
      \]
    \end{proof}

As the following examples show, when the Hilbert series of $I$ has the form $a(t)/g_{p,n}(t)$, it is possible to have $a(1)=0$ or even $a(1)<0$.

    \begin{ex}\label{negmult}
      Let $R=\F_2[x]$ and $\A=R\{f\}/(x^2f)$. No pair among the monomials $x,xf,xf^2,\ldots,xf^k$ in $\A$ admits any common left multiples, so the ideal $I=\A(x,xf,xf^2,\ldots,xf^k)$ decomposes as a direct sum $\A(x)\oplus \A(xf)\oplus \A(xf^2)\oplus\cdots\oplus\A(xf^k)$. From this decomposition, we have
      \[
      \HS_{\A/I}(t) =\HS_\A(t)-\sum_{i=0}^k \HS_{\A(xf^i)}(t)
      \]
     where we already know that $\HS_\A(t)=1/(1-t)(1-t-t^2)$. Let us compute the Hilbert series of $\A(xf^i)$. In the case $i=0$, right multiplication $g\mapsto gx$ is injective, and the homogeneous exact sequence
      \[
      0\to \A(-1) \xrightarrow{\cdot x} I_0\to 0
      \]
      gives $\HS_{I_0}(t) = t/(1-t)(1-t-t^2)$. For $i\geq 1$, the map $g\mapsto gxf^i$ has kernel precisely $\A(x)$, so
      \[
      0\to I_0(-k-1)\to \A(-k-1) \xrightarrow{\cdot xf^i} I_i\to 0
      \]
      is exact, giving $\HS_{I_i}(t) = t^{k+1}(\HS_\A(t)-\HS_{I_0}(t)) = t^{k+1}/(1-t-t^2)$. In total, we have
            \[
      \HS_{\A/I}(t) =\frac{1}{(1-t)(1-t-t^2)} -\frac{t}{(1-t)(1-t-t^2)}-\sum_{i=1}^k\frac{t^{k+1}}{(1-t-t^2)} = \frac{1-\sum_{i=1}^k t^{k+1}}{(1-t-t^2)}.
      \]
Letting $a_k(t)=1-\sum_{i=1}^k t^{k+1}$, we get $a_k(1)=1-k$. Choosing $k$ appropriately, we have a Hilbert series $a_k(t)/g_{2,1}(t)$ in which $a_k(1)$ can be either $0$ or arbitrarily negative.
    \end{ex}

    \section{Filtrations and Buchberger's algorithm}
We will use multiplicative notation when working with potentially noncommutative monoids.
\begin{dfn}\label{nicemonoid}
Let $Q$ be a monoid with identity element $1$. We will call a pair $(Q,<)$ consisting of $Q$ and a total order $<$ a \textit{nice ordered monoid} if the following criteria are satisfied.
\ben
\item $Q$ is cancellative.
\item $<$ is a well order.
\item $1\leq a$ for all $a\in Q$.
\item For all $a,b,c \in Q$, $a<b$ if and only if $ca<cb$ if and only if $ac<bc$.
\een
\end{dfn}

The following are the main examples of nice ordered monoids for our purposes.
\begin{ex} $\N$ with the standard order. This is the only monoid in which we will use additive notation.
  \end{ex}
\begin{ex} The monoid $\mathcal{C}_n$ of monomials in $\F_p[x_1,\ldots,x_n]$ with, for example, the graded reverse lexicographic term order.
  \end{ex}
\begin{ex} The monoid $\mathcal{M}_n$ of monomials in $\mathbf{B}=\F_p[x_1,\ldots,x_n]\{\ff\}$ with the ordering given in Definition \ref{termorder}.
  \end{ex}

Let $K$ be a field, and let $\RR$ be a $K$-algebra, and let $U$ and $V$ be two $K$-subspaces of $\RR$. We denote by $UV$ the subspace spanned by all products $uv$ for $u\in U$ and $v\in V$.

\begin{dfn}
  Let $K$ be a field, let $Q$ be a nice ordered monoid, and let $\RR$ be a possibly noncommutative $K$-algebra. Let $\W$ be a collection of $K$-subspaces of $\RR$ indexed by the elements of $Q$, $\W=\{\W_a\}_{a\in Q}$. We call $\W$ a $Q$-\textit{filtration} of $\RR$ if (i) $\W_a\subseteq \W_b$ whenever $a<b$, (ii) $\bigcup_{a\in Q}\W_a=\RR$, and (iii) $\W_a\W_b\subseteq \W_{ab}$ for all $a,b\in Q$. The pair $(\RR,\W)$ is a $Q$-\textit{filtered $K$-algebra} (or just a \textit{filtered ring} if $Q$ and $K$ are clear from context). The \textit{associated graded ring} of a $Q$-filtered $K$-algebra $(\RR,\W)$ is the $Q$-graded $K$-algebra
  \[
  \gr_\W(\RR)=\bigoplus_{a\in Q} \W_a/(\textstyle\bigcup_{b<a} \W_b)
  \]
We will use $\W_{<a}$ as shorthand for $\bigcup_{b<a}\W_b$.
\end{dfn}

The following are the main examples of filtered rings in which we are interested.
\begin{ex}\label{dfilt} The ring $R=\F_p[x_1,\ldots,x_n]$ with the standard $\N$-filtration by degree. For $i\in \N$, we let $\mathcal{D}_i$ be the $\F_p$-span of monomials of degree $\leq i$. Here, $R$ is naturally identified with its associated graded ring.
  \end{ex}
\begin{ex}\label{lfilt} The ring $R=\F_p[x_1,\ldots,x_n]$ with the following $\mathcal{C}_n$-filtration. For $a\in \C_n$, we let $\mL_a$ be the $\F_p$-span of all monomials $b\leq a$ in the term order of $\C_n$. Again, $R$ is naturally identified with its associated graded ring.
  \end{ex}
\begin{ex}\label{bfilt} The ring $\FF = \F_p[x_1,\ldots,x_n]\{F\}/(x_i^pF-Fx_i\,|\,1\leq i\leq n)$ with the Bernstein $\N$-filtration. That is, $\B_i$ is the $\F_p$-span of the images of all degree $\leq i$ monomials of $\BB$ under the $R$-algebra map $\varphi:\BB\twoheadrightarrow \FF$ that sends $\ff\to F$. The associated graded ring $\gr_\B(\FF)$ is $\A = \F_p[x_1,\ldots,x_n]\{f\}/(x_i^pf\,|\,1\leq i\leq n)$.
  \end{ex}
\begin{ex}\label{rfilt} The ring $\A$ as above with the following $\mathcal{M}_n$-filtration. For $a\in \M_n$, let $\mR_m$ denote the $\F_p$-span of the images of all monomials $b\leq a$ in the term order of $\M_n$ under the $R$-algebra map $\psi:\BB\twoheadrightarrow \A$ that sends $\ff\to f$. The associated graded ring $\gr_\mR(\A)$ is isomorphic to $\A$ as an $R$-algebra, with an $\M_n$-grading in which $[\A]_a=\F_p\cdot\psi(a)$ for all $a$. In particular, the graded components of $\gr_\mR(\A)$ are either $0$ or $1$ dimensional vector spaces.
  \end{ex}

\begin{dfn}
Let $Q$ be a nice ordered monoid. A \textit{nice ordered $Q$-module} is a pair $(P,<)$ consisting of a set $P$ equipped with an action\footnote{By which we mean $1r=r$ and $a(br)=(ab)r$ for all $a,b\in Q$ and all $r\in P$.} $Q\times P\to P$ along with a total order $<$ such that
\ben
\item $P$ is cancellative in the sense that $ar=as$ implies $r=s$ and $ar=b r$ implies $a=b$ for all $a,b\in Q$ and $r,s\in P$.
\item $<$ is a well order.
\item For all $r,s\in P$ and $a\in Q$, it holds that $r<s$ if and only if $ar<as$.
\item For all $r\in P$ and $a,b\in Q$, it holds that $a<b$ if and only if $ar<br$.
\een
\end{dfn}

For our purposes, the only $Q$-modules we care about are coproducts $\coprod_{i\in S} Q\varepsilon_i$ indexed by a well-ordered set $S$, equipped with the total order in which $a\varepsilon_i < b\varepsilon_j$ if and only if either $a<b$ in $Q$, or $a=b$ and $i<j$ in $S$.

\begin{dfn}
Let $K$ be a field, let $Q$ be a nice ordered monoid, and let $(\RR,\W)$ be a $Q$-filtered $K$-algebra. For a left $\RR$-module $M$ and a nice ordered $Q$-module $P$, we define a \textit{$\W$-compatible $P$-filtration} (or just \textit{filtration} if $\W$ and $P$ are clear) on $M$ to be a collection $\{\Omega_r\}_{r\in P}$ of $K$-vector subspaces $\Omega_r$ of $M$ such that (i) $\Omega_r\subseteq\Omega_{s}$ whenever $r<s$, (ii) $\bigcup_{r\in P}\Omega_r=M$, and (iii) $\W_a\Omega_r\subseteq \Omega_{ar}$ for all $a\in Q$ and $r\in P$. We will call $(M,\Omega)$ a \textit{$P$-filtered $(\RR,\W)$-module} (or just an \textit{$(\RR,\W)$-module} if $P$ is understood). The \textit{associated graded module of $(M,\Omega)$}, as a $P$-graded $K$-vector space, is given by
  \[
  \gr_\Omega(M):=\bigoplus_{r\in P}\Omega_r/(\textstyle\bigcup_{s<r}\Omega_s).
  \]
We will write $\Omega_{<r}$ as shorthand for $\bigcup_{s<r}\Omega_s$. For $m\in \Omega_r$, let $[m]^\Omega_r$ denote the image of $m$ in $\Omega_r/\Omega_{<r}$. If $r$ is the least element of $P$ such that $m\in \Omega_r$, then we call $[m]^\Omega_r$ the \textit{symbol} of $m$, denoted $\sigma^\Omega(m)$, and we will call $r$ the \textit{$\Omega$-degree} of $m$, denoted $\deg_\Omega(m)$. We make $\gr_\Omega(M)$ into a $P$-graded left $\gr_\W(\RR)$-module by defining $[r]^\W_a\cdot [m]^\Omega_r := [rm]^\Omega_{ar}$ for all $a\in Q$, $r\in P$, $r\in \W_a$, and $m\in \Omega_r$.
\end{dfn}

Let $(\RR,\W)$ be a filtered ring, and $(M,\Omega)$ be an $(\RR,\W)$-module. Suppose we are given a generating set $(g_i)_{i\in\Upsilon}$ for $M$ over $\RR$. We should have no expectation that the symbols $(\sigma^\Omega(g_i))_{i\in\Upsilon}$ will generate $\gr_\Omega(M)$ over $\gr_\W(\RR)$. We are interested in the generating sets that do have this property.

\begin{dfn}
  Let $(\RR,\W)$ be a filtered ring and $(M,\Omega)$ be an $(\RR,\W)$-module. A sequence of elements $\G=(g_i)_{i\in\Upsilon}$ of $M$ whose symbols $\sigma^\Omega(\G)=(\sigma_\Omega(g_i))_{i\in\Upsilon}$ generate the associated graded module $\gr_\Omega(M)$ over the ring $\gr_\W(\RR)$ is called an $\Omega$-\textit{Gr\"obner basis} for $M$, or just a \textit{Gr\"obner basis} if $\Omega$ is clear.
\end{dfn}

\begin{ex} Take $R=\F_p[x_1,\ldots,x_n]$ with the $\N$-filtration $\mathcal{D}$ of Example \ref{dfilt}. The associated graded ring $\gr_\mathcal{D}(R)$ is identified with $R$ itself. The symbol of a polynomial $f\in R$ is the leading form of $f$, i.e., the homogeneous polynomial whose terms are exactly those terms of $f$ with degree $\deg(f)$. If $I$ is an ideal of $R$, the associated graded ideal $\gr_\mathcal{D}(I)$ is the ideal generated by the leading forms of all elements of $I$. A $\mD$-Gr\"obner basis is usually called a \textit{Macaulay} basis. Clearly not all generating sets are Macaulay bases, but the Noetherianity of $\gr_{\mathcal{D}}(R)$ makes it obvious that finite Macaulay bases exist.
  \end{ex}
\begin{ex} Take $R=\F_p[x_1,\ldots,x_n]$ with the $\mathcal{C}_n$-filtration $\mL$ of Example \ref{lfilt}. Once again, the associated graded ring $\gr_\mL(R)$ is naturally identified with $R$. The symbol of a polynomial $f\in R$ with respect to $\mL$ is the initial term of $f$ under the monomial order we chose for $\mathcal{C}_n$. If $I$ is an ideal of $R$, the associated graded $\gr_\mL(I)$ is the ideal generated by the initial terms of all elements of $I$. Clearly not all generating sets are Gr\"obner bases, though the Noetherianity of $\gr_{\mL}(R)$ makes it obvious that finite $\mL$-Gr\"obner bases exist.
  \end{ex}
\begin{ex} Take the ring
  \[
  \D=\Q[x_1,\ldots,x_n]\langle\partial_1,\ldots,\partial_n\rangle=\frac{\Q\{x_1,\ldots,x_n,\partial_1,\ldots,\partial_n\}}{(x_ix_j-x_jx_i,\,\partial_i\partial_j-\partial_j\partial_i,\,\partial_jx_i-x_i\partial_j-\delta_{ij}\,|\,1\leq i\leq j)}
  \]
  (where $\delta_{ij}$ is the Kronecker symbol) with the Bernstein $\N$-filtration $\B$ by total degree in both the $x$ and $\partial$ variables. The associated graded ring $\gr_\B(\D)$ can be identified with the polynomial ring $\Q[x_1,\ldots,x_n,\xi_1,\ldots,\xi_n]$. The symbol of the operator $x_1\partial_2^2+x_2\partial_1^2+\partial_1\partial_2+1$, for example, is $x_1\xi_2^2+x_2\xi_1^2$. If $I$ is any left ideal of $\D$, the Noetherianity of $\gr_\B(\D)$ makes it clear that $\gr_\D(I)$ is generated by the symbols of finitely many elements of $I$.
  \end{ex}
\begin{ex} Let $\FF$ be equipped with the Bernstein $\N$-filtration $\B$. The associated graded ring of $(\FF,\B)$ is $\A$, and the symbol of a Frobenius operator $g\in \FF$ is the polynomial in $\A$ whose terms correspond only to the highest-complexity terms of $g$. The ring $\A$ is not left-Noetherian and it is entirely possible for a finitely generated left ideal $I$ of $\FF$ to have a non-finitely generated ideal of symbols $\gr_\B(I)$.
  \end{ex}
\begin{ex} Let $\A$ be equipped with the $\M_n$-filtration $\mR=\psi(\mT)$ inherited from the presentation $\psi:\BB\twoheadrightarrow \A$. The associated graded ring $\gr_\mR(\A)$ can be identified with $\A$ itself. The symbol of a polynomial $g\in \A$ is its initial term under the term order of Definition \ref{aatermorder}. Because $\A$ is not left Noetherian, it is not at all obvious whether a finitely generated ideal of $\A$ needs to have a finite $\mR$-Gr\"obner basis. This exact issue is the subject of Theorem \ref{INfg}.
  \end{ex}

We do not have the option of appealing to the Noetherianity of the associated graded ring, so if we wish to understand whether a finitely generated modules admits a finite Gr\"obner bases, we will need to understand the means by which a Gr\"obner basis can be constructed from an arbitrary generating set. This is precisely the subject of Buchberger's criterion. Our treatment of the subject here will fairly closely align with the perspective in Kreuzer and Robbiano \cite[Section 2.3]{kreuzrob}, working in somewhat greater generality to encompass the types of filtrations encountered in the context of Frobenius modules.

Let $\RR$ be a ring, and $M$ be a left $\RR$-module. Let $\G=(g_i)_{i\in\Upsilon}$ be a sequence of elements of $M$. By a \textit{syzygy} of $\G$, we will always mean \textit{left syzygy}, i.e., a vector $\sum_{i\in\Upsilon}\tau_i \varepsilon_i$ in $\bigoplus_{i\in\Upsilon}\RR\varepsilon_i$ such that $\sum_i\tau_ig_i=0$. We denote by $\Syz(\G)$ the left submodule of $\bigoplus_{i\in\Upsilon}\RR\varepsilon_i$ consisting of the syzygies on $\G$.

Given a sequence of elements $\G=(g_i)_{i\in\Upsilon}$ of $M$, by a harmless abuse of notation, we will also let $\G$ denote the left $\RR$-linear map $\G:\bigoplus_{i\in\Upsilon} \RR\varepsilon_i\to M$ that sends $\sum_i \tau_i\varepsilon_i$ to $\sum_i\tau_ig_i$. We have an exact sequence of left $\RR$-modules
\[
0\to\Syz(\G)\to\bigoplus_{i\in\Upsilon}\RR\varepsilon_i\xrightarrow{\G} M
\]
We would like to compare the submodule of $M$ generated by $\G$ to the submodule of $\gr_\Omega(M)$ generated by the sequence of symbols $\sigma^\Omega(\G):=(\sigma^\Omega(g_i))_{i\in\Upsilon}$. It will help to place a filtration on $\bigoplus_{i\in\Upsilon}\RR\varepsilon_i$ that is compatible with both the map $\G$ and the filtration $\Omega$ on $M$. Let $r=(\deg_{\Omega}(g_i))_{i\in\Upsilon}$.

\begin{dfn}
  Let $K$ be a field. Let $Q$ be a nice ordered monoid, let $(\RR,\W)$ be a $Q$-filtered $K$-algebra, and let $P$ be a nice ordered $Q$-module. Let $r=(r_i)_{i\in\Upsilon}$ be a sequence of elements of $P$. The \textit{Schreyer pullback filtration on $\bigoplus_{i\in\Upsilon}\RR\varepsilon_i$ with respect to $r$} is the $P$-filtration $\Lambda$ on $\bigoplus_{i\in \Upsilon} \RR\varepsilon_i$ defined by $\deg_\Lambda( \tau \varepsilon_i) = \deg_\W(\tau)r_i$. That is, for all $s\in P$,
\[
\Lambda_{s} := \text{Span}_{K}\{\tau \varepsilon_i\,|\, i\in\Upsilon,\,\,\deg_\W(\tau) r_i=s\}
\]
\end{dfn}
\noindent The Schreyer filtration with respect to $r$ may be expressed as 
\[
\Lambda_s = \bigoplus_{i\in\Upsilon} \W(\div r_i)_s\,\varepsilon_i
\]
where $\W(\div r_i)$ is the filtration on $\RR$ defined by
\[
\W(\div r_i)_s:=\bigcup_{a\in Q\colon ar_i\leq s}\W_a
\]
\noindent The quotient
\[
\W(\div r_i)_{s}/\textstyle\bigcup_{t<s}\W(\div r_i)_t
\]
can only be nonzero if $s=ar_i$ for some $a\in Q$. If $a=1$ is the identity of $Q$, this quotient is just $\W_1$. In general, the union $\bigcup_{t<s}\W(\div r_i)_t$ is unaltered if we were to drop all indices $t$ that are not left multiples of $r_i$, and the elements of $P$ less than $ar_i$ that happen to be left multiples of $r_i$ correspond bijectively with the elements of $Q$ less than $a$. As a result, for $s=ar_i$,
\[
\W(\div r_i)_{s}/\textstyle\bigcup_{t<s}\W(\div r_i)_t = \W(\div r_i)_{r a_i}/\textstyle\bigcup_{b<a}\W(\div r_i)_{b r_i} = \W_{a}/\textstyle\bigcup_{b<a}\W_{b}.
\]
To summarize,
\[
  [\gr_{\W(\div r_i)}(\RR)]_s=\begin{cases}
  [\gr_{\W}(\RR)]_{a} & \text{ if } s=ar_i\\
  0 & \text{otherwise}\\
  \end{cases}
  \]
The action of $\gr_{\W}(\RR)$ on $\gr_{\W(\div r_i)}(\RR)$ is simply
  \[
  [\tau]^{\W}_{a}[\rho]^{\W(\div r_i)}_{b r_i}=[\tau\rho]^{\W(\div r_i)}_{abr_i}
  \]
and $\gr_{\W(\div r_i)}(\RR)$ is generated over $\gr_{\W}(\RR)$ by the element $[1]_{r_i}^{W(\div r_i)}$. If $\TT =\gr_\W(\RR)$, we will write $\TT(\div r_i)$ as shorthand for $\gr_{\W(\div r_i)}(\RR)$, which we may think of as a Serre twist of the $P$-graded ring $\TT$ backwards by $r_i$. Regarding the Schreyer filtration $\Lambda$ on $\bigoplus_{i\in\Upsilon} \RR \varepsilon_i$ with respect to $r$, we have
  \[
  \gr_\Lambda\left(\bigoplus_{i\in\Upsilon} \RR \varepsilon_i\right) = \bigoplus_{i\in\Upsilon} \TT(\div r_i)\varepsilon_i
  \]
  
\begin{lem}\label{syzmap}
  Let $(\RR,\W)$ be a $Q$-filtered ring, let $\TT=\gr_\W(\RR)$, and let $(M,\Omega)$ be a $P$-filtered $(\RR,\W)$-module. Let $\G=(g_i)_{i\in\Upsilon}$ be a sequence of elements in $M$ with degrees $r=(r_i)_{i\in\Upsilon}$, and let $\Lambda$ be the Schreyer pullback filtration with respect to $r$. The symbol map
  \[
  \sigma^{\Lambda}:\bigoplus_{i\in\Upsilon}\RR\varepsilon_i\to \bigoplus_{i\in\Upsilon} \TT(\div r_i)\varepsilon_i
  \]
  sends syzygies on $\G$ to syzygies on $\sigma^\Omega(\G)$.
\end{lem}
\begin{proof}
  Let $\tau=(\tau_i)_{i\in\Upsilon}$ be a syzygy on $\G=(g_i)_{i\in\Upsilon}$, and let us say that $\deg_{\Lambda}(\tau)=s$, i.e., for all $i$, $\deg_\W(\tau_i)=t_i$ such that $t_ir_i\leq s$ with equality for at least one $i$. Then for each $i$, we have $\tau_i g_i\in \Omega_s$, and
  \[
  0 = \left[\textstyle\sum_i \tau_i g_i\right]^\Omega_s = \sum_i [\tau_i]^\W_{t_i}[g_i]^\Omega_{r_i}.
  \]
  In other words, the vector
  \[
  \sigma^{\Lambda}(\tau)=[\tau]^{\Lambda}_s=([\tau_i]^{\W(\div r_i)}_s)_{i\in\Upsilon}=([\tau_i]^\W_{t_i})_{i\in\Upsilon}
  \]
  is a syzygy on $\sigma^\Omega(\G)$.
\end{proof}
From the above, we obtain what Kreuzer and Robbiano call the ``fundamental diagram'' \cite[Sec. 2.3.6]{kreuzrob} for $\G$ and $\sigma^\Omega(\G)$.
  \begin{center}
\begin{tikzcd}
0\ar[r] & \Syz(\G)\ar[d,"\sigma^{\Lambda}"] \ar[r] & \bigoplus_{i\in\Upsilon}\RR\varepsilon_i\ar[d,"\sigma^{\Lambda}"] \ar[r,"{\G}"] & M\ar[r]\ar[d,"\sigma^{\Omega}"] & 0\\%M/\\G\ar[r] & 0\\
0\ar[r] & \Syz(\sigma^\Omega(\G)) \ar[r] & \bigoplus_{i\in\Upsilon}\TT(\div r_i)\varepsilon_i \ar[r,"{\sigma^\Omega(\G)}"] & \gr_{\Omega}(M)%\gr_{\Omega}(M)/\A\sigma(\G)\ar[r] & 0
\end{tikzcd}
  \end{center}
  Importantly, the right square of the diagram fails to commute. The following lemmas will be of some importance in understanding the lifting properties of this diagram. First, recall the notion of a strict morphism of filtered modules (see, e.g., McConnell and Robson \cite[Sec. 7.6.12/13]{mcconnellrobson}).
\begin{dfn}
  Let $(\RR,\W)$ be a $Q$-filtered ring, and let $(M,\Omega)$ and $(N,\Sigma)$ be $P$-filtered $(\RR,\W)$-modules. A $\RR$-linear map $\varphi:M\to N$ such that $\varphi(\Omega_r)=\Sigma_r\cap \varphi(M)$ for all $r\in P$ is \textit{strict}.
\end{dfn}
\begin{lem}\label{strictmaps}
  %Let $Q$ be a nice ordered monoid, and let $(\RR,\W)$ be a $Q$-filtered ring. Let $P$ be a nice ordered $Q$-module, and let $(M,\Omega)$ and $(N,\Sigma)$ be filtered $(\RR,\W)$-modules.
  Let $\varphi:(M,\Omega)\to(N,\Sigma)$ be a strict map. Let $K=\Ker(\varphi)$ and $C=\Coker(\varphi)$. Then the induced sequence of homogeneous maps between the graded left $\gr_\W(\RR)$-modules below is exact:
  \[
  0\to \gr_{\Omega\cap K}(K)\to\gr_\Omega(M)\to\gr_\Sigma(N)\to \gr_{\overline{\Sigma}}(C)\to 0
  \]
  where $(\Omega\cap K)_r:=\Omega_r\cap K$ and $\overline{\Sigma}_r := \Sigma_r/(\Sigma_r\cap\text{Im}(\varphi))$.
\end{lem}
\begin{proof}
For each $r\in P$, we have natural short exact sequences
  \[
  0\to\frac{\Omega_r\cap K}{\Omega_{<r}\cap K}\to \frac{\Omega_r}{\Omega_{<r}}\to \frac{\Omega_r}{\Omega_{<r}+\Omega_r\cap K}\to 0
  \]
  and
  \[
    0\to\frac{\Sigma_r\cap \text{Im}(\varphi)}{\Sigma_{<r}\cap \text{Im}(\varphi)}\to \frac{\Sigma_r}{\Sigma_{<r}}\to \frac{\Sigma_r}{\Sigma_{<r}+\Sigma_r\cap \text{Im}(\varphi)}\to 0
    \]
    Strictness, i.e., the equality $\Sigma_r\cap\text{Im}(\varphi)=\varphi(\Omega_r)$, gives the result at once upon identifying $\varphi(\Omega_r)$ with $\Omega_r/(\Omega_r\cap\Ker(\varphi))$, so that
    \[
    \frac{\Sigma_r\cap\text{Im}(\varphi)}{\Sigma_{<r}\cap\text{Im}(\varphi)}
    =\frac{\varphi(\Omega_r)}{\varphi(\Omega_{<r})} = \frac{\Omega_r}{\Omega_{<r}+\Omega_r\cap\Ker(\varphi)}.
    \]
\end{proof}

\begin{lem}\label{strictsymbol}
  Let $(\RR,\W)$ be a $Q$-filtered ring, and let $(M,\Omega)$ be a $P$-filtered $(\RR,\W)$-module. Let $\G=(g_i)_{i\in\Upsilon}$ be a Gr\"obner basis for $M$, with degrees $r=(r_i)_{i\in\Upsilon}$. For all $s\in P$, the image of the $s$th term of the Schreyer pullback filtration $\Lambda_s$ under $\G$ is exactly $\Omega_s$. That is, the map of filtered modules
  \[
  \G:\left(\textstyle\bigoplus_{i\in\Upsilon}\RR\varepsilon_i,\,\Lambda\right)\to(M,\Omega)
  \]
  is strict. In particular, the sequence
  \[
  0\to \gr_{\Lambda\cap\Syz(\G)}(\Syz(\G))\to \textstyle\bigoplus_{i\in\Upsilon} \TT(\div r_i)\varepsilon_i \xrightarrow{\sigma^\Omega(\G)} \gr_\Omega(M)\to 0
  \]
  is exact, and $\gr_{\Lambda\cap\Syz(\G)}(\Syz(\G))=\Syz(\sigma^\Omega(\G))$.
\end{lem}
\begin{proof}
  Once we have proven strictness, the claimed exact sequence is immediate from Lemma \ref{strictmaps}. So, let $m$ be any element of $M$ with $\deg_\Omega(m)=s$. Since $\sigma^\Omega(\G)$ generates $\gr_\Omega(M)$, there exist $\tau_i\in \W_{t_i}$ with $t_i$ satisfying $t_ir_i\leq s$ for all $i$ such that $[m]^\Omega_s = \sum_{i\in\Upsilon}[\tau_i]^\W_{t_i}[g_i]^\Omega_{r_i}$. Lifting this expression from $[\gr_\Omega(M)]_s$ to $\Omega_r$, we get $m=m_0+\sum_i\tau_ig_i$ with $m_0\in\Omega_{<s}$. We claim that an element of this form belongs to the image of $\Lambda_s$ under $\G$. If not, since $P$ is well-ordered, we may assume that $m=m_0+\sum_i\tau_ig_i$ is a counterexample for which $t=\deg_\Omega(m_0)$ is minimal. Write $m_0=m_1 + \sum_i\rho_i g_i$ for $\rho_i\in\W(\div r_i)_t\subseteq \W(\div r_i)_s$ and $\deg_\Omega(m_1)<t$. By minimality, $m_0$ cannot be a counterexample, so it lies in the image of $\Lambda_s$ under $\G$. But now the same can be said of $m=m_0+\sum_i\tau_ig_i$, which is a contradiction.
\end{proof}
\begin{lem}\label{grobgen}
  Let $(\RR,\W)$ be a $Q$-filtered ring, and let $(M,\Omega)$ be a $P$-filtered $(\RR,\W)$-module. A Gr\"obner basis for $M$ is a generating set for $M$ over $\RR$. In particular, for any homogeneous generating set $(a_i)_{i\in\Upsilon}$ for $\gr_\Omega(M)$ over $\gr_\W(\RR)$, there exists a lift $(g_i)_{i\in\Upsilon}$ (in the sense that $\sigma^\Omega(g_i)=a_i$ for all $i$) that generates $M$ over $\RR$.
\end{lem}
\begin{proof}
Let $\G$ be a Gr\"obner basis for $M$, and take any $m\in M$, with $r=\deg_\Omega(m)$. By Lemma \ref{strictsymbol}, we can write $m=\sum_i \tau_ig_i$ with $\tau_i\in\W(\div r_i)_s$. In particular, $m$ is in the $\RR$-span of $\G$.
\end{proof}

\begin{theorem}\label{buchcrit}
  (Buchberger's Criterion) Let $(\RR,\W)$ be a $Q$-filtered ring, let $\TT=\gr_\W(\RR)$, and let $(M,\Omega)$ be a $P$-filtered $(\RR,\W)$-module. Let $\G=(g_i)_{i\in\Upsilon}$ be a (not necessarily finite) generating set for $M$, with $\sigma^\Omega(\G)=(\sigma^\Omega(g_i))_{i\in\Upsilon}$ the corresponding sequence of symbols. Let $\Lambda$ denote the Schreyer pullback filtration with respect to $r=(\deg_\Omega(g_i))_{i\in\Upsilon}$. The following are equivalent.
  \begin{enumerate}[(i)]
  \item The sequence $\G$ is an $\Omega$-Gr\"obner basis for $M$. That is, $\sigma^\Omega(\G)$ generates $\gr_\Omega(M)$ over $\TT$.
  \item $\gr_{\Lambda\cap\Syz(\G)}(\Syz(\G))=\Syz(\sigma^\Omega(\G))$.
  \item The lifts of any homogeneous generating set for $\Syz(\sigma^\Omega(\G))$ form a $\Lambda$-Gr\"obner basis for $\Syz(\G)$.
  \item Every homogeneous syzygy on $\sigma^\Omega(\G)$ lifts to a syzygy on $\G$.
    \end{enumerate}
\end{theorem}
\begin{proof}
 (i)$\implies$(ii) is a consequence of Lemma \ref{strictsymbol}.

  (ii)$\implies$(iii) Take any homogeneous generating set $(\zeta_i)_{i\in \Xi}$ for $\Syz(\sigma^\Omega(\G))=\gr_{\Lambda\cap\Syz(\G)}(\Syz(\G))$. A sequence of homogeneous elements in $\gr_{\Lambda\cap\Syz(\G)}(\Syz(\G))$ lifts to a sequence $\mathcal{Z}=(z_i)_{i\in\Xi}$ in $\Syz(\G)$ with $\sigma^\Lambda(z_i)=\zeta_i$ for all $i$. Since $\sigma^\Lambda(\mathcal{Z})=(\zeta_i)_{i\in\Xi}$ generates $\gr_{\Lambda\cap\Syz(\G)}(\Syz(\G))$, our sequence $\mathcal{Z}$ is by definition a Gr\"obner basis for $\Syz(\G)$.
  
  (iii)$\implies$(iv) Given any homogeneous syzygy $\zeta$ on $\sigma^\Omega(\G)$, we can trivially extend a homogeneous generating set for $\Syz(\sigma^\Omega(\G))$ to include $\zeta$, and therefore $\zeta$ admits a lift to $\Syz(\G)$.
  
  (iv)$\implies$(i) Take any homogeneous element $\xi\in \gr_\Omega(M)$. Find some $m\in M$ such that $\sigma^\Omega(m)=\xi$, and then write $m$ as $\tau\G$ for some vector $\tau\in\bigoplus_{i\in\Upsilon}\RR\varepsilon_i$. Let $s=\deg_\Lambda(\tau)$, with $t_i=\deg_\W(\tau_i)$ satisfying $t_ir_i\leq s$ for all $i$. Suppose $\xi$ does not belong to the $\TT$-span of $\sigma^\Omega(\G)$, and assume that $s=\deg_\Lambda(\tau)$ is chosen minimally among all elements of $\gr_\Omega(M)$ not lying in $\TT\cdot \sigma^\Omega(M)$. We have $[m]^\Omega_s=[\tau\G]^\Omega_s=\sum_i [\tau_i]^\W_{t_i} [g_i]^\Omega_{r_i}=\sigma^\Lambda(\tau)\sigma^\Omega(\G)$. Since $[m]^\Omega_s$ transparently belongs to the $\TT\cdot \sigma^\Omega(\G)$, it cannot be nonzero, as this would give $[m]^\Omega_s=\sigma^\Omega(m)=\xi$. So, $\sigma^\Lambda(\tau)$ is a syzygy on $\sigma^\Omega(\G)$. By hypothesis, we can lift $\sigma^\Lambda(\tau)$ to a (necessarily $\Lambda$-degree $s$) syzygy $\zeta$ on $\G$ itself. Since $\tau$ and $\zeta$ have the same $\Lambda$-symbol, we have $\deg_\Lambda(\tau-\zeta)<s$, and thus, by minimality, $\sigma^\Omega((\tau-\zeta)\G)$ belongs to $\TT\cdot \sigma^\Omega(\G)$. But $(\tau-\zeta)\G=\tau\G=m$, so this is a contradiction.
  \end{proof}

\begin{cor}\label{buchalg}
  (Buchberger's Algorithm) Let $(\RR,\W)$ be a $Q$-filtered ring, and let $(M,\Omega)$ be a $P$-filtered $(\RR,\W)$-module. Let $\G$ be a (not necessarily finite) sequence of elements generating $M$. Let $\G_0=\G$, and set $i=0$. Consider the following procedure\footnote{Which, as we've described it here, will run forever.}
  \begin{enumerate}[(1)]
  \item Compute the sequence of symbols $\sigma^\Omega(\G_i)=(\sigma^\Omega(g)\,|\,g\in \G_i)$.
  \item Compute a homogeneous generating set $\mathcal{Z}_i$ for $\Syz(\sigma^\Omega(\G_i))$.
  \item For each syzygy $\zeta=\sum_{g\in\G_i}\zeta_g \varepsilon_g\in \mathcal{Z}_i$ choose any lift $\sum_{g\in\G_i}z_g\varepsilon_g$, where $\sigma^\Omega(z_g)=\zeta_g$ for all $g$. Compute $h_\zeta = \sum_{g\in\G_i} z_g g$.
  \item Let $\G_{i+1} = \G_i\cup(h_\zeta\,|\,\zeta\in\mathcal{Z}_i)$.
  \item Increment $i$ by $1$, and return to step (1).
  \end{enumerate}    
Let $\G_\infty = \bigcup_i \G_i$ and $\mathcal{Z}_\infty=\bigcup_i\mathcal{Z}_i$. The sequence $\G_\infty$ is an $\Omega$-Gr\"obner basis for $M$. Let $\Upsilon=|\mathcal{G}_\infty|$. If $\Lambda$ denotes the Schreyer pullback filtration on $\bigoplus_{i\in\Upsilon}\RR \varepsilon_i$ induced by $\mathcal{G}_\infty$, then any lift of $\mathcal{Z}_\infty$ to $\bigoplus_{i\in\Upsilon}\RR \varepsilon_i$ will form a $\Lambda$-Gr\"obner basis for $\Syz(\G_\infty)$.
\end{cor}
\begin{proof}
  A homogeneous syzygy $\zeta$ on the symbols of $\G_\infty$ can only involve finitely many nonzero terms, so all the elements implicated in that syzygy must appear in $\G_i$ for some finite value of $i$. Given a homogeneous syzygy $\zeta$ of $\Lambda$-degree $s$ on the symbols of $\G_i$, choose a lift of $\zeta$ to a vector $z$, and form the element $h_\zeta=\sum_{g\in\G_i} z_g g$ to be placed into $\G_{i+1}$. The component $z_g$ appears in some $\W_{a_g}$ such that $a_g\deg_\Omega(g)\leq s$, and since $\zeta$ is a syzygy, we must have $t_z:=\deg_\Omega(h_z)<s$. The expression $0=(-1)h_z + \sum_{g\in\G_i} Z_g g$ gives a syzygy on $\G_{i+1}$, and hence, on $\G_\infty$. When we take the symbol of this syzygy under the Schreyer pullback filtration, we exactly recover the original $\zeta$. This is because the only new component we introduced, $-1$ on $h_z$, has degree $1\in Q$, with $1\cdot \deg_\Omega(h_z)=t_z<s$, and thus, maps to $0$ in $\TT(\div t_z)_s=\W(\div t_z)_s/\bigcup_{s^\prime<s}\W(\div t_z)_{s^\prime}$. So, every syzygy $\zeta$ on $\sigma(\G_\infty)$ has a lift to a syzygy on $\G_\infty$, i.e., $\G_\infty$ is an $\Omega$-Gr\"obner basis.

  A syzygy on $\sigma^\Omega(\G_\infty)$ must belong to $\sigma^\Omega(\G_i)$ for some finite value of $i$, and thus, is in the $\TT$-span of the generators $\mathcal{Z}_i$. It follows that $\mathcal{Z}_\infty=\bigcup_i\mathcal{Z}_i$ generates $\Syz(\sigma^\Omega(\G_\infty))=\gr_{\Lambda\cap\Syz(\G_\infty)}(\Syz(\G_\infty))$, and any lift of $\mathcal{Z}_\infty$ to $\bigoplus_{i\in\Upsilon}\RR \varepsilon_i$ is therefore a Gr\"obner basis for $\Syz(\G_\infty)$.
\end{proof}

We end this section with some remarks about ``good'' and ``great'' filtrations.

\begin{dfn}
  Let $(\RR,\W)$ be a $Q$-filtered $K$-algebra. Let $M$ be an $\RR$-module, and let $\Omega$ be a $\W$-compatible $P$-filtration on $M$. If every $\Omega_r$ is finite-dimensional over $K$ and the module $\gr_\Omega(M)$ is finitely generated over $\gr_\W(\RR)$, then we refer to $\Omega$ as a \textit{good filtration}. If every $\Omega_r$ is finite dimensional and $\gr_\Omega(M)$ is finitely presented over $\gr_\W(\RR)$, then we refer to $\Omega$ as a \textit{great filtration}.
\end{dfn}
The following is an easy consequence of Lemma \ref{grobgen}.
\begin{lem}\label{goodfg}
    Let $(\RR,\W)$ be a filtered $K$-algebra, and let $M$ be a $\RR$-module. If $M$ has a good filtration $\Omega$, then $M$ admits a finite $\Omega$-Gr\"obner basis, and in particular, $M$ is finitely generated over $\RR$.
\end{lem}
\begin{proof}
  If $\Omega$ is a good filtration, take a finite generating set for $\gr_\Omega(M)$, and lift it to $M$. The lifts are an $\Omega$-Gr\"obner basis, which generate $M$ by Lemma \ref{grobgen}.
  \end{proof}
Thanks to Buchberger's criterion, we can also prove the analogue of this statement for great filtrations.
\begin{lem}\label{greatfp}
    Let $(\RR,\W)$ be a filtered $K$-algebra, and let $M$ be a left $\RR$-module. If $M$ admits a great filtration $\Omega$, then every finite $\Omega$-Gr\"obner basis for $M$ has a finitely generated syzygy module, and hence, $M$ is finitely presented over $\RR$.
\end{lem}
\begin{proof}
  Let $\Omega$ be a great filtration on $M$, and let $\G=(g_i)_{i=1}^t$ be a finite $\Omega$-Gr\"obner basis with degrees $r=(r_i)_{i=1}^t$. Surject a finite rank free module $\RR^{t}=\bigoplus_{i=1}^t\RR\varepsilon_i$ onto $M$ by $\varepsilon_i\mapsto g_i$ for all $i$, and let $\Lambda$ be the Schreyer pullback filtration on $\RR^t$ with respect to $r$. Let $\TT=\gr_\W(\RR)$. We have a short exact sequence of graded $\TT$-modules
  \[
  0\to \Syz(\sigma^\Omega(\G))\to \textstyle\bigoplus_{i=1}^t\TT(\div r_i)\varepsilon_i \xrightarrow{\sigma^\Omega(\G)}\gr_\Omega(M)\to 0
    \]
    Since both $\gr_\Omega(M)$ and $\bigoplus_{i=0}^t\TT(\div r_i)\varepsilon_i$ are finitely presented over $\TT$, the module $\text{Syz}(\sigma^\Omega(\G))$ is finitely generated. Buchberger's criterion implies that $\Syz(\sigma^\Omega(\G))=\gr_{\Lambda\cap\Syz(\G)}(\Syz(\G))$, so we see that $\Lambda\cap\Syz(\G)$ is a good filtration, and hence, $\Syz(\G)$ is finitely generated.
\end{proof}

  %%%%%%%%%%%%%%%%%%%%%%%%%%%%%%%%%%%%%%%%%%%%%%%%%%%%%%%%%%%%%%%%%%%%%%%%%%%%%%%%%%%%%%%%%%%%%%%%%55
  %%%%%%%%%%%%%%%%%%%%%%%%%%%%%%%%%%%%%%%%%%%%%%%%%%%%%%%%%%%%%%%%%%%%%%%%%%%%%%%%%%%%%%%%%%%%%%%%%%%
  %%%%%%%%%%%%%%%%%%%%%%%%%%%%%%%%%%%%%%%%%%%%%%%%%%%%%%%%%%%%%%%%%%%%%%%%%%%%%%%%%%%%%%%%%%%%%%%%%%%

  %%%%%%%%%%%%%%%%%%%     %%%%%     %%%
         %%%%%            %%% %%    %%%
         %%%%%            %%%  %    %%%
         %%%%%            %%%   %   %%%
         %%%%%            %%%   %   %%%
         %%%%%            %%%    %  %%%
         %%%%%            %%%     %%%%%
  %%%%%%%%%%%%%%%%%%%     %%%       %%%

  %%%%%%%%%%%%%%%%%%%%%%%%%%%%%%%%%%%%%%%%%%%%%%%%%%%%%%%%%%%%%%%%%%%%%%%%%%%%%%%%%%%%%%%%%%%%%%%%%55
  %%%%%%%%%%%%%%%%%%%%%%%%%%%%%%%%%%%%%%%%%%%%%%%%%%%%%%%%%%%%%%%%%%%%%%%%%%%%%%%%%%%%%%%%%%%%%%%%%%%
  %%%%%%%%%%%%%%%%%%%%%%%%%%%%%%%%%%%%%%%%%%%%%%%%%%%%%%%%%%%%%%%%%%%%%%%%%%%%%%%%%%%%%%%%%%%%%%%%%%%

\section{Finitely presented graded $\A$-modules}

Let $R=\F_p[x_1,\ldots,x_n]$, let $\BB=R\{\ff\}$, and let $\A=R\{f\}(x_i^pf\,|\,1\leq i\leq n)$. Under the $R$-algebra map $\psi:\BB\twoheadrightarrow \A$ that sends $\ff\to f$, the ring $\A$ inherits a grading from $\BB$ in which $\deg(x_i)=\deg(f)=1$ for all $i$. It also inherits a monomial filtration. Let $\M_n$ denote the monoid of monomials in $\BB$, with the term order of Definition \ref{termorder}. In the terminology of Definition \ref{nicemonoid}, $\M_n$ is a nice ordered monoid. There is an obvious $\M_n$-filtration $\mT$ on $\BB$ taking $\mT_a$ to be the $\F_p$ span of all monomials $b\leq a$. The image of $\mT$ under $\psi$ gives an $\M_n$-filtration on $\A$, which we'll denote by $\mR$. The associated graded ring $\gr_\RR(\A)$ is naturally identified with $\A$ itself as an $R$-algebra, with the grading $[\A]_a=\F_p\cdot \psi(a)$ for all $a\in\M_n$. 

We can straightforwardly extend this monomial filtration to free $\A$-modules. For any well ordered set ${S}$, let $\M_n^{S}=\coprod_{i\in{S}}\M_n\varepsilon_i$ be ordered by $a\varepsilon_i < b\varepsilon_j$ if and only if either $a<b$ or $a=b$ and $i<j$. The free module $\BB^{S}=\sum_{i\in{S}}\BB\varepsilon_i$ has an obvious $\M_n^{S}$-filtration $\mT^{S}$, where $\mT^{S}_{a\varepsilon_i}$ is the $\F_p$-span of all monomials $b\varepsilon_j<a\varepsilon_i$. Mapping $\BB^{S}\twoheadrightarrow \A^{S}$ induces an $\M_n^{S}$-filtration $\mR^{S}$ on $\A^{S}$. Again, $\gr_{\mR^{S}}(\A^{S})$ can be identified with a copy of $\A^{S}$ graded by $[\A^{S}]_{a\varepsilon_i} =\F_p\cdot \psi(a\varepsilon_i)$.

In a free module $\A^S$, we do not necessarily assume that all standard basis elements have degree $1$. If we wish to emphasize the degrees of standard basis elements, we will write $\A^S=\sum_{i\in S}\A(d_i)\varepsilon_i$ for $d_i\in \Z$, in which case $\deg(1\cdot \varepsilon_i)=d_i$.

The identification $\gr_{\mR^{S}}(\A^{S})$ with $\A^{S}$ allows us to talk about the symbol $\sigma(g)$ of an element $g\in\A^{S}$ as an element of $\A^{S}$ in its own right.
\begin{dfn}
Let $\A^{S}$ be a free $\A$-module. The \textit{initial term} of an element $g\in \A^{S}$, denoted $\inn(g)$, is the image of $\sigma^{\mR^{S}}(g)$ under the identification $\gr_{\mR^{S}}(\A^{S})=\A^{S}$.
\end{dfn}
An element $g\in\A^{S}$ can be written as an $\F_p$-linear combination $\sum_{a,i} c_{a,i} a\varepsilon_i$ of finitely many monomials $a\varepsilon_i$ for $a\in\M_n$, $i\in{S}$. Those monomials $a\varepsilon_i$ appearing with $c_{a,i}\neq 0$ are the \textit{monomial support} of $g$. The initial term $\inn(g)$ is precisely the term $c_{a,i}a\varepsilon_i$ where $a\varepsilon_i$ is maximum element the monomial support under the ordering of Definition \ref{aatermorder} (extended to $\A^{S}$ in the obvious way).
\begin{dfn}
Let $\A^{S}$ be a free $\A$-module, and let $N\subseteq \A^{S}$ be a submodule. The \textit{initial submodule} of $N$, denoted $\inn(N)$, is the image of $\gr_{\mR^{S}\cap N}(N)$ under the identification $\gr_{\mR^{S}}(\A^{S})=\A^{S}$.
\end{dfn}
The initial submodule $\inn(N)$ is the $\A$-submodule of $\A^{S}$ generated by the initial terms of all elements $g\in N$. As in the previous section, a collection of elements $\G=(g_i)$ in $N$ whose initial terms generate $\inn(N)$ is a \textit{Gr\"obner basis} for $N$.

\begin{ex}\label{noto}
A single element may not necessarily constitute a Gr\"obner basis for the ideal it generates. For example, over $R=\F_2[x]$, let $g=xfxf+f+x$. The initial ideal of $\A\cdot g$ is $\A(xfxf,\,xf,\,x^3)$, with Gr\"obner basis $\G=(xfxf+f+x,\, xf+x^2,\,x^3)$, i.e., $\G=(g,xg,x^2g)$.
\end{ex}

For the purpose of understanding Hilbert functions, the following straightforward lemma will be of some assistance.

\begin{lem}\label{innkeeper}
 Let $N$ be a graded submodule of a free module $\A^S$. Every monomial belonging to $\inn(N)$ is the initial term of a homogeneous element of $N$.
\end{lem}
\begin{proof}
  Since the monomial filtration on $\A^S$ refines the $\N$-filtration by degree, the initial term of $g$ is the same as the inital term of the homogeneous leading form of $g$. Since $N$ is graded, the leading form of $g\in N$ already belongs to $N$.
\end{proof}

The following argument is completely standard in the theory of term orders. We include its proof merely to reassure the reader that none of the pathologies of our current setting interfere with the standard argument.
\begin{theorem}\label{initprop}\label{MBThilbert}
 (Macaulay's Basis Theorem) Let $N$ be a graded submodule of a free module $\A^{S}$.
  \ben
\item The images mod $N$ of the degree $d$ monomials of $\A^{S}$ not belonging to $\inn(N)$ form an $\F_p$-vector space basis for $[\A^S/N]_{d}$.
  \item\label{mbt} The Hilbert functions of $\A^{S}/N$ and $\A^{S}/\text{in}(N)$ are identical.
  \een
\end{theorem}
\begin{proof}
  (i) The images of these monomials are clearly linearly independent over $\F_p$, since a nontrivial combination $\sum_m \alpha_m m$ equal to $0$ mod $N$ would need to have its initial term in $\inn(N)$. So, the claim could only fail if the monomials in $[\A^S]_d$ not in $[\inn(N)]_d$ fail to span $[\A^S/N]_d$. Call $g\in [\A^S]_d$ a counterexample if its image mod $N$ is not in the span of those monomials. Every counterexample contains at least one monomial in $[\inn(N)]_d$, and for a fixed counterexample, we may consider the largest monomial of $[\inn(N)]_d$ appearing in its support. Let $m$ be minimal among all monomials in $[\inn(N)]_d$ that appear as the maximal $[\inn(N)]_d$ monomial in the support of a counterexample. Let $g$ be a counterexample containing $m$ as its maximal $[\inn(N)]_d$ monomial, and let $h\in N$ be a homogeneous element with $\inn(h)=m$. For suitably chosen $\gamma\in\F_p$, if the element $g-\gamma h\in [\A^S]_d$ contains any terms in $[\inn(N)]_d$ at all, they must be strictly less than $m$. So, by minimality, $g-\gamma h$ is not a counterexample, i.e., the image mod $N$ of $g-\gamma h$ is in the span of monomials not in $[\inn(N)]_d$. But $h\in N$, so this image coincides with the image of $g$, a contradiction

  (ii) Immediate from (i), since $\inn(\inn(N))=\inn(N)$.
\end{proof}

Of primary importance to us are finitely generated submodules of finite rank free modules.

\begin{lem}\label{trumonsyz}
  Let $a$ and $b$ be monomials in $\A$. There exists a finite generating set for the module of left syzygies on $(a,b)$ whose elements can be described as follows, all of which are homogeneous in the Schreyer pullback grading.
  \begin{enumerate}[(i)]
  \item\label{syz0} If $\ord(a)=\ord(b)=0$, then let $x^\ell=\text{LCM}(a,b)$, and the syzygy $(x^\ell/a,-x^\ell,b)$ generates the module $\Syz(a,b)$.
  \item Otherwise, the following syzygies form a generating set.
    \begin{enumerate}
    \item\label{syz1} If $\ord(a)>0$, say $a=x^v f m_a$, include $(x_i^{p-v_i},0)$ for all $i$.
    \item\label{syz2} If $\ord(b)>0$, say $b=x^u f m_b$, include $(0,x_i^{p-u_i})$ for all $i$.   
    \item\label{syz3} If $\ord(a)=\ord(b)>0$ and there happens to exist a monomial $m$ such that $a=x^vf m$ and $b=x^v fm$, then let $x^\ell=\text{LCM}(x^u,x^v)$ and include the syzygy $(x^{\ell-u}, -x^{\ell-v})$.
    \item\label{syz4} If $\ord(a)\neq\ord(b)$ and we happen to have either $a=cb$ or $b=ca$ for some monomial $c$ of positive $f$-order, then include the syzygy $(1,-m)$ or $(m,-1)$, respectively.
    \end{enumerate}
  \end{enumerate}
\end{lem}
\begin{proof}
  Since $a$ and $b$ are their own initial terms, Theorem \ref{buchcrit} gives an identification $\Syz(a,b)=\gr_{\Lambda\cap \Syz(a,b)}(\Syz(a,b))$ where $\Lambda$ is the Schreyer $\M_n$-filtration. In particular, $\Syz(a,b)$ is homogeneous in the Schreyer grading on $\A(\div a)\varepsilon_a\oplus\A(\div b)\varepsilon_b$, and we need only worry about having enough elements to generate all Schreyer-homogeneous syzygies on $(a,b)$. For $s\in \M_n$, let us ask which homogeneous elements of degree $s$ in $\A(\div a)\varepsilon_a\oplus\A(\div b)$ can be syzygies. An element of the form $c\varepsilon_a$ with $ca=s$ in $\BB$ is a syzygy if and only if $ca=0$ in $\A$, in which case $a=x^v f m_a$ for some $m_a$, and $c$ is a multiple of some $x_i^{p-v_i}$. The $x_i^{p-v_i}\varepsilon_a$ syzygies are covered by (\ref{syz1}). A similar situation holds for elements of the form $d\varepsilon_b$ and the syzygies in (\ref{syz2}).

  The potential syzygies that remain have the form $\alpha c\varepsilon_a+\beta d\varepsilon_b$ for $\alpha,\beta\in\F_p$, and $c,d$ monomials such that $ca=db=s$ when regarded as elements of $\BB$, where $s\neq 0$ in $\A$. In this case, the situation $ca-db=0$ is preserved when we lift each monomial from $\A$ to a monomial in $\BB$, and here we may use Lemma \ref{monomialsyz}, which classifies $(c,-d)$ as belonging to one of the classes (\ref{syz0}), (\ref{syz3}), or (\ref{syz4}) above.
 \end{proof}  
\begin{lem}\label{fgmonsyz}
  Let $a_1,\ldots,a_t$ be a sequence of monomials in $\A$. The module of syzygies $\Syz(a_1,\ldots,a_t)$ is generated by the following elements of $\bigoplus_{i=1}^t\A\varepsilon_i$, which are homogeneous with respect to the Schreyer pullback grading.
  \begin{enumerate}[(i)]
  \item\label{syzlong0} $x_i^{p-v_{ji}}\varepsilon_j$ for all $1\leq j\leq t$ such that $a_j=x^{v_j}fm_j$ for some $m_j$, and all $1\leq i\leq n$
  \item $x^{u_i}\varepsilon_i-x^{u_j}\varepsilon_j$ whenever either $a_i=x^{v_i}$ and $a_j=x^{v_j}$, or $a_i=x^{v_i}fm$ and $a_j=x^{v_j}fm$ for some $m$. In either case, $x^{u_i}=\lcm(x^{v_i},x^{v_j})/x^{v_i}$ and $x^{u_j}=\lcm(x^{v_i},x^{v_j})/x^{v_j}$.
  \item $c\varepsilon_i-\varepsilon_j$ with $\ordf(c)>0$ if we happen to have some $i,j$ with $ca_i=a_j$ and $\ordf(a_i)<\ordf(a_j)$. Otherwise, no syzygies containing a term of positive $f$-order are necessary.
  \end{enumerate}
\end{lem}
\begin{proof}
  As before, $\Syz(a_1,\ldots,a_t)$ is generated by Schreyer homogeneous elements. One the syzygies (\ref{syzlong0}) have been accounted for, the Schreyer homogeneous syzygies of degree $s$ that remain have the form $\sum_i \alpha_i c_i a_i=0$ where $c_ia_i=s$ for all $i$ and $s\neq 0$ in $\A$. But then $(\sum_i\alpha_i)s=0$. Since $(\alpha_1,\ldots,\alpha_t)$ sums to $0$, our syzygy $\sum_i \alpha_i c_i \varepsilon_i$ decomposes into syzygies of the form $c_i\varepsilon_i-c_j\varepsilon_j=0$. The result is now immediate from Lemma \ref{trumonsyz}.
\end{proof}

\def\rb#1{\text{rb}(#1)}
\begin{dfn}
  Let $R=\F_p[x_1,\ldots,x_n]$ be regarded as a graded subalgebra of $\A=R\{f\}/(x_i^pf\,|\,1\leq i\leq n)$. The \textit{robustness} of an element $g\in\A$, denoted $\rb{g}$, is the smallest $d\in \N$ such that $ag\in R$ for every monomial $a$ in $x_1,\ldots,x_n$ of degree $\geq d$. If $g=(g_i)_{i\in S}$ is a vector in a free module $\A^S$, define $\rb{g}$ to be $\sup_{i\in S}\{\rb{g_i}\}$.
\end{dfn}

For the reader's convenience, we list some basic properties of robustness below.

\begin{lem}
  Let $R=\F_p[x_1,\ldots,x_n]$ and let $\A=R\{f\}/(x_i^pf\,|\,1\leq i\leq n)$. Let $g$ and $h$ be elements of the ring $\A$.
  \begin{enumerate}[(i)]
  \item $\rb{g}=0$ if and only if $g\in R$.
  \item $\rb{g}\leq p(n-1)+1$.
  \item The robustness of $g$ is the maximum robustness of any monomial in the support of $g$.
  \item $\rb{g+h}\leq \max\{\rb{g},\rb{h}\}$.
  \item For any $v\in \N_{<p}^n$, $\rb{x^vfg}=n(p-1)+1-|v|$.  
  \item If $c$ is a monomial with $\ordf(c)>0$, then $\rb{cg}=\rb{c}$.
  \item If $c$ is a monomial with $\ordf(c)=0$, then $\rb{cg}=\max\{0,\rb{g}-|c|\}$.
    \end{enumerate}
\end{lem}
\begin{proof}
Straightforward from the definition.
\end{proof}

\begin{theorem}\label{INfg}
  If $N$ is a finitely generated graded left submodule of a finite rank free module $\A^t$, then $\inn(N)$ is finitely generated.
\end{theorem}
\begin{proof} 
  Take a finite sequence of homogeneous generators $\G_0=(g_1,\ldots,g_r)$ for $N$, and begin carrying out Buchberger's algorithm, producing progressively larger sequences $\G_0\subseteq \G_1\subseteq \G_2\subseteq\G_3\subseteq\cdots$. 

  Given $\G_k$, consider the set of points $\cS_k=\{(\rb{g},\deg(g))\,|\,g\in\G_k\}\subseteq\N^2$. As we construct $\G_{k+1}$, we will be interested in how the points of $\cS_{k+1}$ compare to those in $\cS_k$. So, let us take a finite homogeneous generating set $\mathcal{Z}_k$ for $\Syz(\inn(\G_k))$ as in Lemma \ref{fgmonsyz}, and form the elements $h_z=\sum_{g\in\G_k} z_g g$ for each generating syzygy $z=(z_g)_{g\in\G_k}\in\mathcal{Z}_k$. There are three basic types of syzygies on $\inn(\G_{k})$ to consider.
  \begin{enumerate}[(i)]
  \item $z=x_i^e\varepsilon_j$. Clearly $\deg(h_z)\leq \deg(g_j)+e$, and either $\rb{h_z}=0$ or $\rb{h_z}=\rb{g_j}-e$. 
  \item $z=x^u\varepsilon_i-x^w\varepsilon_j$. Here, either $\inn(g_i)=x^{v_i}$ and $\inn(g_j)=x^{v_j}$ or $\inn(g_i)=x^{v_i}fm$ and $\inn(g_j)=x^{v_j}fm$ for some monomial $m$. In either case, let $x^\ell=\lcm(x^{v_i},x^{v_j})$, and we have $u=\ell-v_i$ and $w=\ell-v_j$. Depending on which case we're in, the initial term of $h_z=x^ug_i-x^wg_j$ is strictly less than $x^\ell$, resp. $x^\ell fm$, and thus $\deg(h_z)\leq \ell$, resp. $\deg(h_z)\leq \ell+1+\deg(m)$. Either way, $\deg(h_z)\leq \deg(g_i)+|u|=\deg(g_j)+|w|$. Regarding robustness, either $x^ug_i\in R$ or $\rb{x^ug_i}=\rb{g_i}-|u|$. Likewise, either $x^wg_i\in R$ or $\rb{x^wg_j}=\rb{g_j}-|w|$. Without loss of generality, say $\rb{x^wg_j}\geq \rb{x^ug_i}$. If $\rb{x^wg_j}=0$, then both $x^wg_j$ and $x^ug_i$ are in $R$, giving $\rb{h_z}=0$. Otherwise, if $\rb{x^wg_j}>0$, then $\rb{x^wg_j}=\rb{g_j}-|w|$, and we get $\rb{h_z}\leq \max\{\rb{x^ug_i},\rb{x^wg_j}\}=\rb{g_j}-|w|$. In total, 
    \[
    \rb{h_z}\leq \rb{g_j}-|w|\text{ and }\deg(h_z)\leq \deg(g_j)+|w|.
    \]
    If we had assumed $\rb{x^ug_i}>\rb{x^wg_j}$, we would have instead obtained either $\rb{h_z}=0$ or
    \[
    \rb{h_z}\leq \rb{g_i}-|u|\text{ and }\deg(h_z)\leq \deg(g_i)+|u|.
    \]
%    Either way, there is a point $(r,d)\in\cS_k$ such that $\rb{h_z}+\deg(h_z)\leq r+d$.
  \item $z=c\varepsilon_i-\varepsilon_j$ with $\ordf(c)\geq 1$. For any monomial $m$ appearing with nonzero coefficient in $g_i$, either $cm=0$ or $\rb{cm}=\rb{c}=\rb{\inn(g_j)}$ (since $\inn(g_j)=c\inn(g_i)$ and $\ord(c)>0$). We therefore have $\rb{h_z}\leq \rb{g_j}$. Furthermore, it is clear that $\inn(h_z)<\inn(g_j)$, and in particular, $\deg(h_z)\leq \deg(g_j)$.% In total, $\rb{h_z}+\deg(h_z)\leq \rb{g_j}+\deg(g_j)$.
\end{enumerate}

In each case, we find that for every point $(\rb{h_z},\deg(h_z))\in\cS_{k+1}\setminus\cS_k$ with $\rb{h_z}\neq 0$, there exists some $(r,d)\in\cS_k$ such that $\rb{h_z}+\deg(h_z)\leq r+d$. Taking $C=\max\{\rb{g}+\deg(g)\,|\,g\in\G_0\text{ with }\rb{g}\neq 0\}$, some finite constant, we have by induction on $k$ that for all $g\in\G_\infty$, either $\rb{g}=0$ or $\rb{g}+\deg(g)\leq C$. Only finitely many monomials $m\in\A^t$ satisfy the bounds $\rb{m}>0$ and $\rb{m}+\deg(m)\leq C$. The sequence $\inn(\G_\infty)$, which, by Corollary \ref{buchalg}, generates $\inn(N)$, therefore involves only finitely of monomials of positive $f$-order, along with some possibly infinite set of monomials of $f$-order $0$. Dickson's lemma for $\F_p[x_1,\ldots,x_n]$ then allows us to generate all the $f$-order $0$ monomials belonging to $\inn(\G_\infty)$ from some finite subset.
\end{proof}

We observe a few immediate consequences of Theorem \ref{INfg}, in light of Buchberger's criterion.

\begin{cor}\label{grfgfp}
  If $N$ is a finitely generated graded left submodule of a finite rank free module $\A^t$, then $N$ is finitely presented.
\end{cor}
\begin{proof}
  By Theorem \ref{INfg}, the monomial submodule $\inn(N)$ of $\A^t$ is finitely generated. Lemma \ref{fgmonsyz} shows that any finite sequence $(m_1,\ldots,m_t)$ of monomial generators for $\inn(N)$ has a finitely generated syzygy module. Let $\mathcal{G}=(g_1,\ldots,g_t)$ be a sequence of elements of $N$ with $\inn(g_i)=m_i$ for all $i$. By Theorem \ref{buchcrit}(iii), any lifts (with respect to the Schreyer filtration) of a finite homogeneous generating set for $\Syz(m_1,\ldots,m_t)$ will constitute finite a Gr\"obner basis for $\Syz(\G)$.
\end{proof}
\begin{cor}\label{subfpisfp}
  Let $M$ be a finitely presented graded $\A$-module, and let $N$ be a graded left submodule of $M$. If $N$ is finitely generated over $\A$, then it is finitely presented.
\end{cor}
\begin{proof}
  Straightforward from Corollary \ref{grfgfp}.
\end{proof}

With Theorems \ref{INfg}, \ref{MBThilbert}, and \ref{monomialhilbertseries}, we obtain a complete description of the possible Hilbert series of a finitely presented graded $\A$-module. However, as in the monomial ideal case, we can understand where the Hilbert series comes from somewhat more directly using a certain structure map. This approach will still require us to use the fact that if $\A^t/N$ is a presentation of a graded $\A$-module $M$, with $N$ finitely generated, then $\inn(N)$ is also finitely generated.

\begin{dfn}\label{grstruct}
  Let $N$ be a graded left $\A$-module. Define a map
  \[
  [N]_i \xleftarrow{\Theta_{N,i}} \textstyle\bigoplus_{v\in \N_{<p}^n} [N]_{i-|v|-1}
  \]
  as the sum over all $v\in\N_{<p}^n$ of the component maps
  \[
    [N]_i \xleftarrow{x^{v}f\cdot} [N]_{i-|v|-1}.
    \]
    The \textit{structure morphism} of $N$, as a homogeneous map of graded vector spaces, is $\Theta_N=\bigoplus_i \Theta_{N,i}$,
    \[
    N\xleftarrow{\Theta_N}\textstyle\bigoplus_{v\in\N_{<p}^n} N(-|v|-1).
    \]
\end{dfn}

A graded $\A$-module $N$ can naturally be regarded as a graded $R$-module by restriction of scalars along the homogeneous inclusion $R\hookrightarrow \A$. However, to work with the map $\Theta_N$, we will want to place a nonstandard $R$-module structure on $\bigoplus_{v\in \N_{<p}^n} N(-|v|-1)$.
\begin{dfn}\label{mfaction}
  Let $N$ be a graded $\A$-module. As a graded vector $\F_p$-vector space, we define $\mathcal{F}_\A(N)$ to be the direct sum of Serre twists $\bigoplus_{v\in \N_{<p}^n} N(-|v|-1)\varepsilon_v$, where $\varepsilon_v$ denotes a formal basis vector. We make $\mathcal{F}_\A(N)$ into an $R$-module via the action
  \[
  x_i\cdot (g\varepsilon_v) = \begin{cases}
    g\varepsilon_{v+e_i} & \text{ if $v_i<p-1$}\\
    0 & \text{ if $v_i=p-1$}\\
    \end{cases}
  \]
  where $e_i$ denotes the $i$th standard basis vector of $\N^n$. If $\varphi:M\to N$ is a homogeneous $\A$-linear map, we define $\mF_\A(\varphi):\mF_\A(M)\to\mF_\A(N)$ by
  \begin{center}
  \begin{tikzcd}
        \textstyle\bigoplus_{v\in\N_{<p}^n} M(-|v|-1)\ar[rr,"\bigoplus_{v\in\N_{<p}^n} \varphi"] && \textstyle\bigoplus_{v\in\N_{<p}^n} N(-|v|-1)
  \end{tikzcd}
  \end{center}
  which is readily seen to be $R$-linear. We write $\mF_\A(-)$ for the resulting functor from graded $\A$-modules to graded $R$-modules.
\end{dfn}
For any $\A$-linear map $\varphi:M\to N$, the map $\mF_\A(\varphi)$ is, at the level of graded vector spaces, just a direct sum of Serre twists of $\varphi$. It is therefore obvious that $\mF_\A(-)$ is exact.

\begin{lem}\label{thetafacts}
  Let $N$ be a graded left $\A$-module, regarded as a graded $R$-module via restriction of scalars along $R\hookrightarrow\A$.
  \begin{enumerate}[(i)]
      \item The structure map $\Theta_N:\mathcal{F}_\A(N)\to N$ is $R$-linear.
  \item If $M$ is another graded left $\A$-module, and $\varphi:M\to N$ is a homogeneous $\A$-linear map, then the following diagram of graded $R$-modules commutes.
    \begin{center}
      \begin{tikzcd}
  \mF_\A(M)\ar[r,"\mF_\A(\varphi)"]\ar[d,"\Theta_M"'] & \mF_\A(N)\ar[d,"\Theta_N"]\\
        M\ar[r,"\varphi"] & N
        \end{tikzcd}
    \end{center}
    That is to say, $\Theta_{(-)}$ defines a natural transformation from $\mF_\A(-)$ to the forgetful functor, as functors from graded $\A$-modules to graded $R$-modules.
  \end{enumerate}
\end{lem}
\begin{proof} 
  (i) $\Theta_N:\mF_\A(N)\to N$ is the map that sends $g\varepsilon_v$ to $x^vfg$ for all $v\in \N_{<p}^n$. For any $v\in\N_{<p}^n$, we have $x_i\cdot\Theta_N\left(g \varepsilon_v\right)=x^{e_i+v}fg$. If $v_i=p-1$, this vanishes, in agreement with $\Theta_N(x_i\cdot g\varepsilon_v)$. If $v_i<p-1$, then $x^{e_i+v}fg=\Theta_N(g\varepsilon_{v+e_i})$. Regarding claim (ii), it is enough to show that for each $v\in\N^n_{<p}$, the diagram
    \begin{center}
      \begin{tikzcd}
        M(-|v|-1)\ar[r,"\varphi"]\ar[d,"x^vf\cdot"'] & N(-|v|-1)\ar[d,"x^vf\cdot"']\\
        M\ar[r,"\varphi"] & N
        \end{tikzcd}
      \end{center}
commutes, but this is obvious from the $\A$-linearity of $\varphi$.
\end{proof}
\begin{dfn}
  Let $N$ be a graded $\A$-module. We define the \textit{structural kernel (resp. cokernel)} of $N$ to be the graded $R$-module $\Ker(\Theta_N)$ (resp. $\Coker(\Theta_N)$), which, for the sake of brevity, we will denote by $K_N$ (resp. $C_N$).
\end{dfn}
By Lemma \ref{thetafacts}, the assignments $N\mapsto C_N$ and $N\mapsto K_N$ are functorial.

\begin{ex}
  Considering $\A$ itself as a graded left $\A$-module, left multiplication by $x^vf$ for any $v\in\N_{<p}^n$ is injective, so $K_\A=0$. The image of $\Theta_\A$ is precisely the span of those monomials with positive $f$-order. The cokernel $C_\A$ is naturally identified with $R=\F_p[x_1,\ldots,x_n]$, giving a exact sequence of graded $R$-modules:
  \[
  0\leftarrow R \leftarrow \A\xleftarrow{\Theta_\A}\mF_\A(\A)\leftarrow 0
  \]
\end{ex}

For any $\A$-module $N$, the action of $x_i$ on $\mF_\A(N)$ is defined in such a way that $\mF_\A(N)$ is automatically an $\mf{m}$-torsion module, where $\mf{m}$ denotes the homogeneous maximal ideal $(x_1,\ldots,x_n)$ of $R$. In particular, an $R$-submodule of $\mF_\A(N)$ is finitely generated if and only if it has finite length.

\begin{theorem}\label{thetaaction}
  Let $N$ be a left $\A$-module.
  \begin{enumerate}[(i)]
  \item If $N$ is finitely generated over $\A$, then the cokernel of $\Theta_N$ is finitely generated over $R$.
  \item If $N$ is finitely presented over $\A$, then the kernel of $\Theta_N$ has finite length over $R$
    \end{enumerate}
\end{theorem}
\begin{proof}
  (i) Say $N$ is generated over $\A$ by $g_1,\ldots,g_t$. The module $C_N$ is spanned over $\F_p$ by the images mod $\text{Im}(\Theta_N)$ of all elements of the form $cg_1,\ldots,cg_t$ with $c$ a monomial in $\A$. However, if $c$ has positive $f$-order, then $cg_i$ belongs to the image of $\Theta_N$, so the only elements $cg_i$ with a nonzero image in $C_N$ are those in the $R$-span of $g_1,\ldots,g_t$. 
  
 (ii) Take a presentation $N=\bigoplus_{i=0}^t\A(-d_i)\varepsilon_i/H$ where $H$ is finitely generated and graded. Every monomial of $\bigoplus_{i=0}^t\A(-d_i)\varepsilon_i$ that does not belong to $\inn(H)$ must reduce modulo $H$ to a nonzero element of $N$. By Theorem \ref{prefixcapture} (note we are using the fact that $\inn(H)$ is finitely generated, by Theorem \ref{INfg}), it holds for $i\gg 0$ that if $m$ is a degree $i$ monomial not belonging to $\inn(H)$, then for any $v\in\N_{<p}^n$, the degree $i+1+|v|$ monomial $x^{v}fm$ still does not belong to $\inn(H)$, and thus, $x^{v}fm$ reduces mod $H$ to a nonzero element of $N$. That is, for $i\gg 0$, the maps
  \[
      [N]_i \xleftarrow{x^{v}f\cdot} [N]_{i-|v|-1}
  \]
  are all injective, and so too is their sum $\Theta_{N,i}$. As $[K_N]_i=\Ker(\Theta_{N,i})=0$ for $i\gg 0$, the finite length of $K_N$ is evident.
\end{proof}
There is a converse to Theorem \ref{thetaaction}, in the following sense.
\begin{theorem}\label{thetafg}
  Let $N$ be a graded left $\A$-module whose components $[N]_i$ are finite-dimensional $\F_p$-vector spaces for all $i$, with $[N]_i=0$ for $i\ll 0$.
  \begin{enumerate}[(i)]
  \item If the cokernel of $\Theta_N$ is finitely generated over $R$, then $N$ is finitely generated over $\A$.
  \item If both the cokernel and kernel of $\Theta_N$ are finitely generated over $R$, then $N$ is finitely presented over $\A$.
  \end{enumerate}
\end{theorem}
\begin{proof}
 (i) Let $g_1,\ldots,g_t\in N$ be homogeneous lifts to $N$ of a minimal homogeneous $R$-module generating set for $C_N$. We claim that these elements generate $N$ over $\A$, and will prove this by induction on degree. Let $d\in \Z$ be the degree such that $[N]_d\neq 0$ and $[N]_i=0$ for all $i<d$. Then clearly $[\mF_\A(N)]_d=0$, giving an identification of $[N]_d$ with $[C_N]_d$, in which the $g_i$'s are sufficient to generate. Fix $e\geq d$, and suppose we've shown that all elements of $N$ with degree $<e$ are generated over $\A$ by $g_1,\ldots,g_t$. Let $h$ be some homogeneous element of degree $e$. Consider the exact sequence
  \[
      0\leftarrow [C_N]_e \leftarrow [N]_e\xleftarrow{\Theta_{N,e}} [\mF_\A(N)]_e\leftarrow [K_N]_e\leftarrow 0
      \]
       We can write $h=r_1g_1+\cdots+r_tg_t+\sum_{v\in\N_{<p}^n} x^vf h_v$ where $r_1,\ldots,r_t\in R$ are homogeneous elements such that $\deg(r_ig_i)=e$, and each $h_v$ is a homogeneous element of $N$ with degree $e-|v|-1$. By induction, we can write $h_v=a_{v,1}g_1+\cdots+a_{v,t}g_t$ for various $a_{v,i}\in \A$ of appropriate degrees, giving
      \[
      h=\textstyle\sum_{i=1}^t\left(r_i+\textstyle\sum_{v\in\N_{<p}^n} x^vfa_{v,i}\right) g_i
      \]
      so that $h$ is in the $\A$-span of $g_1,\ldots,g_t$.

      (ii) By part (i), $N$ is finitely generated over $\A$. This gives us a homogeneous short exact sequence of $\A$-modules
      \[
0\to H\to \textstyle\bigoplus_{i=1}^t \A(-d_i) \to N \to 0
\]
By Lemma \ref{thetafacts}, we have the following commutative diagram of homogeneous maps between graded $R$-modules, with exact rows:
  \begin{center}{
    \begin{tikzcd}
      0\ar[r] & \mF_\A(H)\ar[r]\ar[d,"\Theta_G"]& \bigoplus_{i=1}^t\mF_\A(\A)(-d_i)\ar[r]\ar[d,"\bigoplus_{i=1}^t\Theta_\A(-d_i)"]& \mF_\A(N)\ar[r]\ar[d,"\Theta_N"]& 0\\
      0\ar[r] & G\ar[r]& \bigoplus_{i=1}^t\A(-d_i)\ar[r]& N\ar[r]& 0
      \end{tikzcd}}
  \end{center}
and an accompanying exact sequence:
  \[
K_N\to C_H\to \textstyle\bigoplus_{i=1}^t R(-d_i)\to C_N\to 0
  \]
  The finite generation over $R$ of $K_N$ implies the finite generation of $C_H$ over $R$, which by the result of part (i), implies the finite generation of $H$ itself over $\A$. It follows at once that $N$ is finitely presented over $\A$.
\end{proof}
\begin{theorem}\label{amodhilbert}
  Let $N$ be a finitely presented graded $\A$-module. The Hilbert series of $N$ is a rational function of the form
  \[
  \text{HS}_N(t) = \frac{a(t)}{(1-t)^dg_{p,n}(t)}  
  \]
  where $g_{p,n}(t)$ is the polynomial of Definition \ref{gpndef} and $d$ is the Krull dimension of $C_N$ (in particular, $0\leq d\leq n$). If $d>0$, then $a(1)>0$.
\end{theorem}
\begin{proof}
  For each $i$, we have an exact sequence
      \[
    0\leftarrow [C_N]_i \leftarrow [N]_i\xleftarrow{\Theta_{N,i}} \textstyle\bigoplus_{v\in\N_{<p}^n} [N]_{i-|v|-1}\leftarrow [K_N]_i\leftarrow 0
    \]
    and thus, 
    \[
    \HS_N(t) -\textstyle\sum_{v\in\N_{<p}^n} t^{|v|+1}\HS_N(t) = \HS_{C_N}(t)-\HS_{K_N}(t).
    \]
    Since $K_N$ has finite length, $\HS_{K_N}(t)$ is a polynomial $b(t)$ with nonnegative coefficients, while $\HS_{C_N}(t)$ takes the form $c(t)/(1-t)^d$ where $d=\dim(C_N)$ and $c(1)>0$. Let $g_{p,n}(t)=1-\sum_{v\in\N_{<p}^n}t^{|v|+1}$, and we get
    \[
    \HS_N(t)=\frac{c(t)-(1-t)^d b(t)}{(1-t)^dg_{p,n}(t)}
    \]
    Let $a(t)=c(t)-(1-t)^db(t)$. If $d>0$, then $a(1)=c(1)>0$. Otherwise $a(1)=c(1)-b(1)$.
\end{proof}
It is not hard to see that the Krull dimension of $C_N$ can be recovered purely from the Hilbert series of $N$.
\begin{dfn}
  Let $h\in\Q(t)$ be a rational function. Let $\delta_1(h)$ denote the smallest nonnegative integer $d$ such that $(1-t)^dh$ has no pole at $t=1$.
\end{dfn}
\begin{cor}\label{dimfromhs}
  If $N$ is a finitely presented graded $\A$-module, then $\delta_1(\HS_N)=\dim(C_N)$.
\end{cor}
\begin{proof}
For any prime $p$ and any $n\geq 1$, we have $g_{p,n}(1)=1-p^n\neq 0$, so $g_{p,n}(t)$ has no effect on the local behavior of $\HS_N(t)$ at $t=1$. Write $\HS_N(t) = a(t)/(1-t)^{\dim(C_N)}g_{p,n}(t)$. If $\dim(C_n)>0$, then $a(1)> 0$ and it is clear that $\delta_1(\HS_N)=\dim(C_N)$. If $\dim(C_N)=0$, it may be possible to have $a(1)=0$, but regardless of whether this is the case, we clearly have $\delta_1(\HS_N)=0$.
\end{proof}

\begin{dfn}\label{eddef}
  Let $N$ be a finitely presented graded left $\A$-module. We define the \textit{Bernstein dimension} of $N$ to be $\delta(N):=\delta_1(\HS_N)$. Letting $a(t) = (1-t)^{\delta(N)}g_{p,n}(t)\HS_N(t)$, we define the \textit{multiplicity} of $N$ to be $e(N):=a(1)/\delta(N)!$.
\end{dfn}
\begin{remark}
  If $\delta(N)>0$, Theorem \ref{amodhilbert} shows that the multiplicity of $N$ coincides exactly with the multiplicity of $C_N$ in the usual sense. This is not the case when $\delta(N)=0$. Example \ref{negmult} gives a module $\A/I$ with $\delta(\A/I)=0$ and $e(\A/I)=1-k$ for any choice of $k\geq 1$.
\end{remark}

It is clear that the dimension and multiplicity are both invariant under Serre twist, i.e., $e(N(t))=e(N)$ and $\delta(N(t))=\delta(N)$ for any $t\in\Z$.

\begin{theorem}\label{edcalc}
  Let
  \[
  0\to N\to M\to Q\to 0
  \]
  be an $\A$-linear homogeneous short exact sequence of finitely presented graded $\A$-modules. Then $\delta(M)=\max\{\delta(N),\delta(Q)\}$, and
  \[
  e(M)=  \begin{cases}
    e(N) & \delta(N)>\delta(Q)\\
    e(N)+e(Q) & \delta(N)=\delta(Q)\\
    e(Q) & \delta(N)<\delta(Q)
    \end{cases}
  \]
  In particular, the dimension and multiplicity of $M$ are determined by the dimensions and multiplicities of $N$ and $Q$. 
\end{theorem}
\begin{proof}
By Lemma \ref{thetafacts}, we have a commutative diagram of graded $R$-modules with exact rows
  \begin{center}
    \begin{tikzcd}
      0\ar[r] & \mF_\A(N)\ar[r]\ar[d,"\Theta_N"]&\mF_\A(M)\ar[r]\ar[d,"\Theta_M"]& \mF_\A(Q)\ar[r]\ar[d,"\Theta_Q"]& 0\\
      0\ar[r] & N\ar[r]& M\ar[r]& Q\ar[r]& 0
      \end{tikzcd}
  \end{center}
  and an homogeneous exact sequence of $R$-linear maps
  \[
  K_Q\to  C_N\to  C_M\to C_Q\to 0.
  \]
  Since $K_Q$ has finite length, the claim about the dimension follows at once:
  \[
  \dim(C_M)=\max(\dim(C_N),\dim(C_Q))
  \]  
  Let $\HS_N(t)=a(t)/(1-t)^{\delta(N)}g_{p,n}(t)$ and $\HS_Q(t)=c(t)/(1-t)^{\delta(Q)}g_{p,n}(t)$, with $\delta(M)=\max\{\delta(N),\delta(Q)\}$. Then
  \[
  \HS_M(t)=\frac{a(t)(1-t)^{\delta(M)-\delta(N)}+c(t)(1-t)^{\delta(M)-\delta(Q)}}{(1-t)^{\delta(M)}g_{p,n}(t)}
  \]
  Let $b(t)=a(t)(1-t)^{\delta(M)-\delta(N)}+c(t)(1-t)^{\delta(M)-\delta(C)}$. Then either $b(1)=a(1)$ (if $\delta(M)=\delta(N)$ and $\delta(M)>\delta(Q)$), $b(1)=c(1)$ (if $\delta(M)>\delta(N)$) and $\delta(M)=\delta(Q)$), or $b(1)=a(1)+c(1)$ (if $\delta(M)=\delta(N)=\delta(Q)$).
\end{proof}

  %%%%%%%%%%%%%%%%%%%%%%%%%%%%%%%%%%%%%%%%%%%%%%%%%%%%%%%%%%%%%%%%%%%%%%%%%%%%%%%%%%%%%%%%%%%%%%%%%55
  %%%%%%%%%%%%%%%%%%%%%%%%%%%%%%%%%%%%%%%%%%%%%%%%%%%%%%%%%%%%%%%%%%%%%%%%%%%%%%%%%%%%%%%%%%%%%%%%%%%
  %%%%%%%%%%%%%%%%%%%%%%%%%%%%%%%%%%%%%%%%%%%%%%%%%%%%%%%%%%%%%%%%%%%%%%%%%%%%%%%%%%%%%%%%%%%%%%%%%%%

  %%% WELCOME TO

               %%%   %%   %%%%%  %%     %%
               %%%   %%   %%     %%     %%
               %%%%%%%%   %%%%%  %%     %%
               %%%   %%   %%     %%     %%
               %%%   %%   %%%%%  %%%%%  %%%%%%

  %%%%%%%%%%%%%%%%%%%%%%%%%%%%%%%%%%%%%%%%%%%%%%%%%%%%%%%%%%%%%%%%%%%%%%%%%%%%%%%%%%%%%%%%%%%%%%%%%55
  %%%%%%%%%%%%%%%%%%%%%%%%%%%%%%%%%%%%%%%%%%%%%%%%%%%%%%%%%%%%%%%%%%%%%%%%%%%%%%%%%%%%%%%%%%%%%%%%%%%
  %%%%%%%%%%%%%%%%%%%%%%%%%%%%%%%%%%%%%%%%%%%%%%%%%%%%%%%%%%%%%%%%%%%%%%%%%%%%%%%%%%%%%%%%%%%%%%%%%%%

\section{Comparing filtrations}
\begin{remark}
  Throughout this section, we will assume that all filtrations are indexed by $\N$. The crucial property we require is that for fixed $w$, the number of elements between $i$ and $i+w$ is constant for all $i$, and those elements are nothing more than shifts by $i$ of the elements between $0$ and $w$. The analogous statement for monoids of monomials is generally false: consider that there are no elements strictly between $y$ and $z$ in the graded reverse lexicographic order on $K[x,y,z]$ (taking $x<y<z$, say), but we have $xy<y^2<xz$.
  \end{remark}

Let $R=\F_p[x_1,\ldots,x_n]$, let $\FF$ be the ring of Frobenius operators over $R$, and let $\B$ be the Bernstein $\N$-filtration on $\FF$. If $M$ is an $\FF$-module that admits a great filtration $\Omega$, we have defined a Bernstein dimension and multiplicity for the finitely presented graded $\A$-module $\gr_\Omega(M)$. It is quite reasonable to ask whether these values depend on the choice of filtration on $M$. Our goal in this section is to prove that they do not.

The following is a standard result in the theory of filtered modules. See, e.g., Bj\"ork\footnote{The cautious reader will observe that most of Bj\"ork's results on filtered rings are stated with the hypothesis that $\gr_\W(\RR)$ is commutative and Noetherian. However, not all arguments actually require this hypothesis, and the Lemma cited here is one of them.} \cite[Chapter 1, Lemma 3.4]{bjork}.
\begin{lem}\label{goodfiltsep}
  Let $(\RR,\W)$ be an $\N$-filtered ring, and let $M$ be a left $\RR$-module. If $A$ and $B$ are two good $\W$-compatible $\N$-filtrations on $M$, then there exist integers $u$ and $v$ such that, for all $i\in \N$,
  \[
  A_{i-v}\subseteq B_i\subseteq A_{i+u}.
  \]
\end{lem}
If we reindex via $A^\prime_i=A_{i-v}$ (taking the convention that $A_i=0$ for $i<0$), then setting $w=u+v$, we have ${A}^\prime_i\subseteq{B}_i\subseteq{A}^\prime_{i+w}$ for all $i$. Henceforth, we will assume that such reindexing has been done from the start.

\begin{dfn}
Suppose that $A$ and $B$ are two $\N$-filtrations on an $\RR$-module $M$ such that $A_i\subseteq B_i$ for all $i$. The \textit{linear gap} between $A$ and $B$ is the infimum of the set of nonnegative integers $w$ such that $A_i\subseteq B_i\subseteq A_{i+w}$ for all $i\in \N$.
\end{dfn}
Lemma \ref{goodfiltsep} is the statement that, up to reindexing in order to assume $A_i\subseteq B_i$ in the first place, two good filtrations on the same module are always separated by a finite linear gap.

If the functions $h_A:i\mapsto \dim(A_i)$ an $h_B:i\mapsto \dim(B_i)$ grow (eventually) polynomially in $i$, then Lemma \ref{goodfiltsep} is be sufficient to prove that the (eventual) polynomials describing $h_A$ and $h_B$ have the same degree and the same leading term. For our applications, however the growth of $h_A$ and $h_B$ will satisfy a linear recurrence relation with non-polynomial growth, and a linear comparison $h_A(i-v)\leq h_B(i)\leq h_A(i+u)$ is not sufficient to imply that the recurrence relations are the same.
\begin{ex}
  Consider the sequences of integers
  \[
  \begin{matrix}
    a: & 1 & 2 & 3 & 5 & 8 & 13 & 21 & 34 & \cdots\\
    b: & 1 & 2 & 4 & 7 & 12 & 20 & 33 & 54 & \cdots\\
    a(1): & 2 & 3 & 5 & 8 & 13 & 21 & 34 & 55&\cdots
    \end{matrix}
  \]
  with $a_i\leq b_i\leq a(1)_i=a_{i+1}$ for all $i$. The sequences $a$ and $a(1)$ satisfy the eventual linear recurrence relation $1-t-t^2$, but the sequence $b$ does not.    
\end{ex}

Let $(\RR,\W)$ be an $\N$-filtered ring, and let $\TT=\gr_\W(R)$. Let $M$ be an $\RR$-module equipped with two good filtrations $A$ and $B$ such that $A_i\subseteq B_i\subseteq A_{i+w}$ for all $i$. We are interested in the $\N^2$-indexed family of subspaces ${A}_i\cap {B}_j$. There is a well-known comparison between the formation of associated graded objects along rows or columns of the array of subspaces, which we will review below. Our main motivation in reviewing the construction is to establish some notation for (and some properties of) certain $\TT$-modules that come up when working with the double filtration of $A$ and $B$.

\begin{center}
  {\footnotesize
  \begin{tikzcd}[row sep = 1.0em, column sep = 0.8em]
    {A}_0\ar[r,hook]\ar[d,hook] & {A}_0\ar[r,hook]\ar[d,hook] & \cdots\ar[r,hook] & {A}_0\ar[d,hook]\ar[r,hook] &{A}_0\ar[r,hook]\ar[d,hook]&\cdots\ar[r,hook]&{A}_0\ar[d,hook]\\
    {A}_1\cap{B}_0 \ar[r,hook]\ar[d,hook] & {A}_1 \ar[r,hook]\ar[d,hook] & \cdots\ar[r,hook] & {A}_1\ar[d,hook]\ar[r,hook] &{A}_1\ar[r,hook]\ar[d,hook]&\cdots\ar[r,hook]&{A}_1\ar[d,hook]\\
    {A}_2\cap{B}_0 \ar[r,hook]\ar[d,hook] & {A}_2\cap{B}_1\ar[d,hook] \ar[r,hook] & \cdots\ar[r,hook] & {A}_2\ar[d,hook]\ar[r,hook] &{A}_2\ar[r,hook]\ar[d,hook]&\cdots\ar[r,hook]&{A}_2\ar[d,hook]\\
    \vdots\ar[d,hook] & \vdots\ar[d,hook] &\ddots &\vdots\ar[d,hook] &\vdots\ar[d,hook]&\ddots&\vdots\ar[d,hook]\\
    {A}_{w-1}\cap {B}_0\ar[r,hook]\ar[d,hook] & {A}_{w-1}\cap{B}_1\ar[d,hook]\ar[r,hook] &\cdots\ar[r,hook] &{A}_{w-1}\ar[d,hook]\ar[r,hook] &{A}_{w-1}\ar[r,hook]\ar[d,hook]&\cdots\ar[r,hook]&{A}_{w-1}\ar[d,hook]\\
    {B}_0 \ar[r,hook]\ar[d,hook] & {A}_w\cap {B}_1 \ar[r,hook]\ar[d,hook] &\cdots\ar[r,hook] & {A}_w\ar[d,hook]\ar[r,hook] &{A}_w\ar[r,hook]\ar[d,hook]&\cdots\ar[r,hook]&{A}_w\ar[d,hook]\\
    {B}_0 \ar[r,hook]\ar[d,hook] & {B}_1\ar[d,hook] \ar[r,hook]\ar[d,hook] &\cdots\ar[r,hook] & {A}_{w+1}\cap{B}_w\ar[d,hook]\ar[r,hook]\ar[d,hook] &{A}_{w+1}\ar[r,hook]\ar[d,hook]&\cdots\ar[r,hook]&{A}_{w+1}\ar[d,hook]\\
    {B}_0 \ar[r,hook]\ar[d,hook] & {B}_1 \ar[r,hook]\ar[d,hook] &\cdots\ar[r,hook] & {A}_{w+2}\cap{B}_w\ar[r,hook]\ar[d,hook] &{A}_{w+2}\cap{B}_{w+1}\ar[r,hook]\ar[d,hook]&\cdots\ar[r,hook]&{A}_{w+2}\ar[d,hook]\\
        \vdots\ar[d,hook] & \vdots\ar[d,hook] &\ddots &\vdots\ar[d,hook] &\vdots\ar[d,hook]&\ddots&\vdots\ar[d,hook]\\
    {B}_0\ar[r,hook] & {B}_1\ar[r,hook] &\cdots\ar[r,hook] &{B}_w\ar[r,hook] &{B}_{w+1}\ar[r,hook]&\cdots\ar[r,hook]& M\\
  \end{tikzcd}
  }
\end{center}

If we form the associated graded for the filtration along each row, we obtain an array of columns, the $j$th of which is the filtration of ${B}_j/{B}_{j-1}$ by $\{{B}_j\cap{A}_i/{B}_{j-1}\cap{A}_i\}_i$.
\begin{center}
  {\footnotesize
  \begin{tikzcd}[row sep = 1.0em, column sep = 0.8em]
    {A}_0\ar[d,hook] & 0\ar[d,hook] & \cdots & 0\ar[d,hook] & 0\ar[d,hook]&\cdots&\gr_{B}({A}_0)\ar[d,hook]\\
    {A}_1\cap{B}_0 \ar[d,hook] & \frac{{A}_1}{{A}_1\cap{B}_0} \ar[d,hook] & \cdots & 0\ar[d,hook] &0 \ar[d,hook]&\cdots&\gr_{B}({A}_1)\ar[d,hook]\\
    {A}_2\cap{B}_0 \ar[d,hook] & \frac{{A}_2\cap{B}_1}{{A}_2\cap{B}_0}\ar[d,hook]  & \cdots & 0\ar[d,hook] &0\ar[d,hook]&\cdots&\gr_{B}({A}_2)\ar[d,hook]\\
    \vdots\ar[d,hook] & \vdots\ar[d,hook] &\ddots &\vdots\ar[d,hook] &\vdots\ar[d,hook]&\ddots&\vdots\ar[d,hook]\\
    {A}_{w-1}\cap {B}_0\ar[d,hook] & \frac{{A}_{w-1}\cap{B}_1}{{A}_{w-1}\cap{B}_0}\ar[d,hook] &\cdots &0\ar[d,hook] &0\ar[d,hook]&\cdots&\gr_B(A_{w-1})\ar[d,hook]\\
    {B}_0 \ar[d,hook] & \frac{{A}_w\cap {B}_1}{{B}_0} \ar[d,hook] &\cdots & \frac{{A}_w}{{A}_w\cap{B}_{w-1}}\ar[d,hook] &0\ar[d,hook]&\cdots&\gr_B({A}_w)\ar[d,hook]\\
    {B}_0 \ar[d,hook] & {B}_1/{B}_1\ar[d,hook] \ar[d,hook] &\cdots & \frac{{A}_{w+1}\cap{B}_w}{{A}_{w+1}\cap{B}_{w-1}}\ar[d,hook]\ar[d,hook] &\frac{{A}_{w+1}}{{A}_{w+1}\cap{B}_w}\ar[d,hook]&\cdots&\gr_{B}({A}_{w+1})\ar[d,hook]\\
    {B}_0 \ar[d,hook] & {B}_1/{B}_0 \ar[d,hook] &\cdots & \frac{{A}_{w+2}\cap{B}_w}{{A}_{w+2}\cap{B}_{w-1}}\ar[d,hook] &\frac{{A}_{w+2}\cap{B}_{w+1}}{{A}_{w+2}\cap{B}_{w}}\ar[d,hook]&\cdots&\gr_{B}({A}_{w+2})\ar[d,hook]\\
        \vdots\ar[d,hook] & \vdots\ar[d,hook] &\ddots &\vdots\ar[d,hook] &\vdots\ar[d,hook]&\ddots&\vdots\ar[d,hook]\\
    {B}_0 & {B}_1/{B}_0 &\cdots &{B}_w/{B}_{w-1} &{B}_{w+1}/{B}_w&\cdots& \gr_{B}(M)\\
  \end{tikzcd}
  }
\end{center}

Taking direct sums along diagonals in the above diagram, we find a filtration of $\gr_B(M)$ by a certain family of graded $\TT$-modules.
\begin{dfn}
Let $(\RR,\W)$ be an $\N$-filtered ring, and let $\TT=\gr_\W(R)$. Let $M$ be an $\RR$-module equipped with two good filtrations $A$ and $B$, and let $w\in \N$ be such that ${A}_i\subseteq{B}_i\subseteq{A}_{i+w}$ for all $i$. For $k=0,\ldots,w$, define a graded $\TT$-module
\[
V^k_{A,B}:=\bigoplus_{i=0}^\infty\frac{A_{i+k}\cap B_{i}}{A_{i+k}\cap B_{i-1}}
\]
\end{dfn}
\noindent There is a natural chain of homogeneous inclusions
\[
V^0_{A,B}\subseteq V^1_{A,B}\subseteq\cdots\subseteq V^w_{A,B}=\gr_B(M).
\]
The associated graded module with respect to the above filtration is the $\N^2$-graded $\TT$-module
\[
\gr_A\gr_B(M) = \bigoplus_{k=0}^w V^k_{A,B}/V^{k-1}_{A,B}
\]
whose $(k,i)$th graded component is
\[
[\gr_A\gr_B(M)]_{(k,i)} = [V^k_{A,B}/V^{k-1}_{A,B}]_i = \frac{A_{i+k}\cap B_i}{A_{i+k}\cap B_{i-1}+A_{i+k-1}\cap B_{i}}
\]

If we instead formed the associated graded of each column in the original diagram, then we obtain an array of row filtrations, with the $i$th row filtering ${A}_i/{A}_{i-1}$ by $\{{A}_i\cap{B}_j/{A}_{i-1}\cap{B}_j\}_j$.
\begin{center}
  {\footnotesize
  \begin{tikzcd}[row sep = 1.0em, column sep = 0.8em]
    {A}_0\ar[r,hook] & {A}_0\ar[r,hook] & \cdots\ar[r,hook] & {A}_0\ar[r,hook] &{A}_0\ar[r,hook]&\cdots\ar[r,hook]&{A}_0\\
    \frac{{A}_1\cap{B}_0}{A_0} \ar[r,hook] & A_1/A_0 \ar[r,hook] & \cdots\ar[r,hook] & {A}_1/A_0\ar[r,hook] &{A}_1/A_0\ar[r,hook]&\cdots\ar[r,hook]&{A}_1/A_0\\
    \frac{{A}_2\cap{B}_0}{A_1\cap B_0} \ar[r,hook] & \frac{{A}_2\cap{B}_1}{A_1} \ar[r,hook] & \cdots\ar[r,hook] & {A}_2/A_1\ar[r,hook] &{A}_2/A_1\ar[r,hook]&\cdots\ar[r,hook]&{A}_2/A_1\\
    \vdots & \vdots &\ddots &\vdots &\vdots&\ddots&\vdots\\
    \frac{{A}_{w-1}\cap {B}_0}{A_{w-2}\cap B_0}\ar[r,hook] & \frac{{A}_{w-1}\cap{B}_1}{A_{w-2}\cap B_1}\ar[r,hook] &\cdots\ar[r,hook] &{A}_{w-1}/A_{w-2}\ar[r,hook] &{A}_{w-1}/A_{w-2}\ar[r,hook]&\cdots\ar[r,hook]&{A}_{w-1}/A_{w-2}\\
    \frac{{B}_0}{A_{w-1}\cap B_0} \ar[r,hook] & \frac{{A}_w\cap {B}_1}{A_{w-1}\cap B_1} \ar[r,hook] &\cdots\ar[r,hook] & {A}_w/A_{w-1}\ar[r,hook] &{A}_w/A_{w-1}\ar[r,hook]&\cdots\ar[r,hook]&{A}_w/A_{w-1}\\
    0 \ar[r,hook] & \frac{{B}_1}{A_w\cap B_1} \ar[r,hook] &\cdots\ar[r,hook] & \frac{{A}_{w+1}\cap{B}_w}{A_w}\ar[r,hook] &{A}_{w+1}/A_w\ar[r,hook]&\cdots\ar[r,hook]&{A}_{w+1}/A_w\\
    0 \ar[r,hook] & 0 \ar[r,hook] &\cdots\ar[r,hook] & \frac{{A}_{w+2}\cap{B}_w}{A_{w+1\cap B_w}}\ar[r,hook] &\frac{{A}_{w+2}\cap{B}_{w+1}}{A_{w+1}}\ar[r,hook]&\cdots\ar[r,hook]&{A}_{w+2}/A_{w+1}\\
        \vdots & \vdots &\ddots &\vdots &\vdots&\ddots&\vdots\\
    \gr_A({B}_0)\ar[r,hook] & \gr_A({B}_1)\ar[r,hook] &\cdots\ar[r,hook] &\gr_A({B}_w)\ar[r,hook] &\gr_A({B}_{w+1})\ar[r,hook]&\cdots\ar[r,hook]& \gr_A(M)\\
  \end{tikzcd}
  }
\end{center}
Again taking direct sums along the diagonals of the above diagram, we obtain a filtration of $\gr_A(M)$, this time by the following family of graded $\TT$-modules, analogous to the $V^k_{A,B}$-modules.
\begin{dfn}
Let $(\RR,\W)$ be an $\N$-filtered ring, and let $\TT=\gr_\W(R)$. Let $M$ be an $\RR$-module equipped with two good filtrations $A$ and $B$, and let $w\in \N$ be such that ${A}_i\subseteq{B}_i\subseteq{A}_{i+w}$ for all $i$. For $k=0,\ldots,w$, define a graded $\TT$-module
\[
U^k_{A,B}:=\bigoplus_{i=0}^\infty\frac{A_{i}\cap B_{i+k-w}}{A_{i-1}\cap B_{i+k-w}}
\]
\end{dfn}
\noindent There is a natural chain of homogeneous inclusions
\[
U^0_{A,B}\subseteq U^{1}_{A,B}\subseteq\cdots\subseteq U^w_{A,B}=\gr_A(M)
\]
The associated graded module with respect to this filtration is the $\N^2$-graded $\TT$-module
\[
\gr_B\gr_A(M) = \bigoplus_{k=0}^w U^{k}_{A,B}/U^{k-1}_{A,B}
\]
whose $(k,i)$th graded component is
\[
[\gr_B\gr_A(M)]_{(k,i)} = [U^k_{A,B}/U^{k-1}_{A,B}]_i = \frac{A_{i}\cap B_{i+k-w}}{A_{i-1}\cap B_{i+k-w}+A_{i}\cap B_{i+k-w-1}}
\]
For all $(i,k)$, we have $[\gr_B\gr_A(M)]_{(k,i)} = [\gr_A\gr_B(M)]_{(w-k,i+k-w)}$, so that $\gr_B\gr_A(M)$ and $\gr_A\gr_B(M)$ differ only by a change of grading. 

Whether we formed the associated graded along columns first and then rows, or along rows first and then columns, we obtain the same $\N^2$-indexed family of components in the end.

\begin{center}
  {\footnotesize
  \begin{tikzcd}[row sep = 0.5em, column sep = 0.5em]
    {A}_0 & 0 & \cdots & 0 & 0&\cdots&\gr_{B}({A}_0)\\
    \frac{{A}_1\cap{B}_0}{{A}_0}  & \frac{{A}_1}{{A}_1\cap{B}_0}  & \cdots & 0 & 0&\cdots&\gr_{B}({A}_1/{A}_0)\\
    \frac{{A}_2\cap{B}_0}{{A}_1\cap{B}_0}  & \frac{{A}_2\cap{B}_1}{{A}_1+{A}_2\cap{B}_0}  & \cdots & 0 &0&\cdots&\gr_{B}({A}_2/{A}_1)\\
    \vdots & \vdots &\ddots &\vdots &\vdots&\ddots&\vdots\\
    \frac{{A}_{w-1}\cap{B}_0}{{A}_{w-2}\cap{B}_0} & \frac{{A}_{w-1}\cap{B}_1}{{A}_{w-1}\cap{B}_0+{A}_{w-2}\cap{B}_1} &\cdots &0 &0&\cdots&\gr_{B}({A}_{w-1}/{A}_{w-2})\\
    \frac{{B}_0}{{A}_{w-1}\cap{B}_0}  & \frac{{A}_w\cap {B}_1}{{B}_0+{A}_{w-1}\cap{B}_1}  &\cdots & \frac{{A}_w}{{A}_w\cap{B}_{w-1}} &0&\cdots&\gr_{B}({A}_w/{A}_{w-1})\\
    0  & \frac{{B}_1}{{A}_w\cap{B}_1}  &\cdots & \frac{{A}_{w+1}\cap{B}_w}{{A}_w+{A}_{w+1}\cap{B}_w} &\frac{{A}_{w+1}}{{A}_{w+1}\cap{B}_w}&\cdots&\gr_{B}({A}_{w+1}/{A}_w)\\
    0  & 0  &\cdots & \frac{{A}_{w+2}\cap{B}_w}{{A}_{w+1}\cap{B}_w+{A}_{w+2}\cap{B}_{w-1}} &\frac{{A}_{w+2}\cap{B}_{w+1}}{{A}_{w+1}+{A}_{w+2}\cap{B}_{w+1}}&\cdots&\gr_{B}({A}_{w+2}/{A}_{w+1})\\
        \vdots & \vdots &\ddots &\vdots &\vdots&\ddots&\vdots\\
    \gr_{A}({B}_0) & \gr_{A}({B}_1/{B}_0) &\cdots &\gr_{A}({B}_w/{B}_{w-1}) &\gr_{A}({B}_{w+1}/{B}_w)&\cdots& \gr_{{A},{B}}(M)\\
  \end{tikzcd}
  }
\end{center}
%The $(i,j)$th component of $\gr_{{A},{B}}(M)$ is $({A}_i\cap{B}_j)/({A}_{i-1}\cap{B}_j+{A}_i\cap{B}_{j-1})$.
\begin{dfn}
  Let $(\RR,\W)$ be an $\N$-filtered ring, and let $\TT=\gr_\W(R)$. Let $M$ be an $\RR$-module equipped with two good filtrations $A$ and $B$ satisfying $A_i\subseteq B_i\subseteq A_{i+w}$ for all $i$. As an $\N^2$-graded vector space, we define
\[
\gr_{A,B}(M) :=\bigoplus_{(i,j)\in\N^2} \frac{A_i\cap B_j}{A_{i-1}\cap B_j+A_i\cap B_{j-1}}
\]
For $m\in A_i\cap B_j$, let $[m]^{A,B}_{(i,j)}$ denote the image of $m$ in $[\gr_{A,B}(M)]_{(i,j)}$. To make $\gr_{A,B}(M)$ into a $\TT$-module, we defined by $[t]^\W_k\cdot [m]^{{A},{B}}_{(i,j)}=[tm]^{{A},{B}}_{(i+k,j+k)}$. 
\end{dfn}
At the level of graded vector spaces, $\gr_{A,B}(M)$ agrees with $\gr_A\gr_B(M)$ and $\gr_B\gr_A(M)$ up to a change of grading, with $[\gr_{A,B}(M)]_{(i,j)} = [\gr_A\gr_B(M)]_{(i-j,j)}=[\gr_B\gr_A(M)]_{(j-i+w,i)}$. For fixed $j$, we have $\bigoplus_i[\gr_{{A},{B}}(M)]_{(i,j)}=\gr_{A}({B}_j/{B}_{j-1})$, and for fixed $i$, $\bigoplus_j [\gr_{{A},{B}}(M)]_{(i,j)}=\gr_{B}({A}_i/{A}_{i-1})$. As a $\TT$-module, $\gr_{{A},{B}}(M)$ naturally decomposes as a direct sum of $w+1$ components.

\begin{dfn}
  Let $(\RR,\W)$ be an $\N$-filtered ring, and let $\TT=\gr_\W(R)$. Let $M$ be an $\RR$-module equipped with two good filtrations $A$ and $B$, and let $w\in \N$ be such that ${A}_i\subseteq{B}_i\subseteq{A}_{i+w}$ for all $i$. For $k=0,\ldots,w$, define a graded $\TT$-module
\[
\gr^k_{{A},{B}}(M) :=\bigoplus_i [\gr_{{A},{B}}(M)]_{(i+k,i)}
\]
\end{dfn}
\begin{lem}\label{uvfilt}
  Let $(\RR,\W)$ be an $\N$-filtered ring, and let $\TT=\gr_\W(R)$. Let $M$ be an $\RR$-module equipped with two good filtrations $A$ and $B$, and let $w\in \N$ be such that ${A}_i\subseteq{B}_i\subseteq{A}_{i+w}$ for all $i$. For $k=0,\ldots,w$, there are homogeneous isomorphisms of graded $\TT$-modules
\[
\gr^k_{A,B}(M) = \frac{V^k_{A,B}}{V^{k-1}_{A,B}} = \frac{U^{w-k}_{A,B}}{U^{w-k-1}_{A,B}}(+k).
\]
In particular, we have identifications
\[
\gr^0_{A,B}(M) = V^0_{A,B}
\hspace{1.0em}
\text{and}
\hspace{1.0em}
\gr^w_{A,B}(M) = U^0_{A,B}(+w)
\]
and the following homogeneous short exact sequences of graded $\TT$-modules,
\[
0\to V^{w-1}_{A,B}\to \gr_B(M)\to \gr^w_{A,B}(M)\to 0 \hspace{1.0em}\text{and}\hspace{1.0em}
0\to U^{w-1}_{A,B}\to \gr_A(M)\to \gr^0_{A,B}(M)\to 0
\]
\end{lem}
\begin{proof}
  We have $[\gr^k_{A,B}(M)]_i=[\gr_{A,B}(M)]_{(i+k,i)}=[\gr_A\gr_B(M)]_{(k,i)}=[V^k_{A,B}/V^{k-1}_{A,B}]_i$ for all $i$ and $k$, and likewise, $[\gr^k_{A,B}(M)]_i=[\gr_{A,B}(M)]_{(i+k,i)}=[\gr_B\gr_A(M)]_{(w-k,i+k)}=[U^{w-k}_{A,B}/U^{w-k-1}_{A,B}]_{i+k}$. These identifications of graded components straightforwardly preserve the $\TT$-action.
  \end{proof}

In the special case $w=1$, we can realize both $\gr_A(M)$ and $\gr_B(M)$ as extensions involving the same graded $\TT$-modules $\gr^0_{A,B}(M)$ and $\gr^1_{A,B}(M)$, up to Serre twists.
\[
0\to \gr^0_{A,B}(M)\to \gr_B(M)\to \gr^1_{A,B}(M)\to 0
\]\[
0\to \gr^1_{A,B}(M)(-1)\to \gr_A(M)\to \gr^0_{A,B}(M)\to 0
\]
If $w>1$, there is still a nice comparison between the two filtrations available, involving a chain of intermediate filtrations, each separated by only a linear gap of $1$. These filtrations make an appearance in the proof of \cite[Chapter 2, Lemma 6.2]{bjork}.
\begin{dfn}
  Let $A$ and $B$ be two $\N$-filtrations on an $\RR$-module $M$, and suppose $w\in\N$ is such that $A_i\subseteq B_i\subseteq A_{i+w}$ for all $i$. The \textit{interpolating filtrations} of $A$ and $B$ are the family of filtrations $T^k$ indexed by $0\leq k\leq w$, defined by $T^k_i={A}_{i+k}\cap{B}_i$.
\end{dfn}
\noindent Clearly
\[
{A}_i=T^0_i\subseteq T^1_i\subseteq\cdots\subseteq T^{w-1}_i\subseteq T^w_i={B}_i
\]
and, for any $i$ and $k$, $T^{k+1}_{i}= T^{k}_{i+1}\cap{B}_{i}$. In fact,
\[
T^a_i\cap T^b_j = (A_{a+i}\cap B_i)\cap (A_{b+j}\cap B_j) = A_{\text{min}(a+i,b+j)}\cap B_{\min(i,j)}
\]
\begin{lem}\label{filtransformations}
  Let $A$ and $B$ be two $\N$-filtrations on an $\RR$-module $M$, and suppose $w\in\N$ is such that $A_i\subseteq B_i\subseteq A_{i+w}$ for all $i$. For $0\leq k\leq w$, let $T^k_i=A_{i+k}\cap B_i$ be the $k$th interpolating filtration of $A$ and $B$. Fix two such filtrations $T^a$ and $T^b$, with $a\leq b$.
  \begin{enumerate}[(i)]
  \item For any $0\leq k\leq b-a$ and any $0\leq c\leq k$
    \[
    V^k_{T^a,T^b}=V^{k-c}_{a+c,b}
    \]
  \item For any $0\leq k < b-a$ and any $0\leq c\leq w-b$,
    \[
    V^k_{T^a,T^b} = V^k_{T^a,T^{b+c}}
    \]
  \item For any $0\leq k\leq b-a$ and any $0\leq c\leq w-b$,
    \[
    U^k_{T^a,T^b} = U^{k+c}_{T^a,T^{b+c}}
    \]
  \item For any $0\leq k < b-a$ and any $-a\leq c < b-a-k$,
    \[
    U^k_{T^a,T^b}=U^k_{T^{a+c},T^b}(-c)
    \]
  \item For any $0\leq k< b-a$,
        \[
    \gr^k_{T^a,T^b}(M)=\gr^k_{T^a,T^{b+c}}
    \]
  \item For any $-a\leq c < k\leq b-a$,
        \[
    \gr^k_{T^a,T^b}(M)=\gr^{k-c}_{T^{a+c},T^{b}}
    \]    
  \end{enumerate}
\end{lem}
\begin{proof}
  For $0\leq k\leq b-a$ and $0\leq c\leq k$, we have
\begin{align*}
  V^k_{T^a,T^b} &=\bigoplus_{i=0}^\infty \frac{T^a_{i+k}\cap T^b_i}{T^a_{i+k}\cap T^b_{i-1}}\\
  &=\bigoplus_{i=0}^\infty \frac{T^{a+c}_{i+k-c}\cap T^b_i}{T^{a+c}_{i+k-c}\cap T^b_{i-1}}  \\
  &=V^{k-c}_{T^{a+c},T^b}
\end{align*}
and for $k<b-a$, $0\leq c\leq w-b$,
\begin{align*}
  V^k_{T^a,T^b} &=\bigoplus_{i=0}^\infty \frac{T^a_{i+k}\cap T^b_i}{T^a_{i+k}\cap T^b_{i-1}}\\
  &=\bigoplus_{i=0}^\infty \frac{T^a_{i+k}\cap T^{b+c}_i}{T^a_{i+k}\cap T^{b+c}_{i-1}}\\
  &=V^k_{T^a,T^{b+c}}
\end{align*}
Likewise, for $0\leq k\leq (b-a)$ and $0\leq c\leq w-b$,
\begin{align*}
  U^k_{T^a,T^b} &=\bigoplus_{i=0}^\infty \frac{T^a_{i}\cap T^b_{i+k-(b-a)}}{T^a_{i-1}\cap T^b_{i+k-(b-a)}}\\
  &=\bigoplus_{i=0}^\infty \frac{T^a_{i}\cap T^{b+c}_{i+k-(b-a)}}{T^a_{i-1}\cap T^{b+c}_{i+k-(b-a)}}\\
  &=\bigoplus_{i=0}^\infty \frac{T^a_{i}\cap T^{b+c}_{i+k+c-((b+c)-a)}}{T^a_{i-1}\cap T^{b+c}_{i+k+c-((b+c)-a)}}\\
  &=U^{k+c}_{T^a,T^{b+c}}
\end{align*}
and if $k+c< b-a$
\begin{align*}
  U^k_{T^a,T^b} &=\bigoplus_{i=0}^\infty \frac{T^a_{i}\cap T^b_{i+k-(b-a)}}{T^a_{i-1}\cap T^b_{i+k-(b-a)}}\\
 &=\bigoplus_{i=0}^\infty \frac{T^{a+c}_{i-c}\cap T^b_{i+k-(b-a)}}{T^{a+c}_{i-c-1}\cap T^b_{i+k-(b-a)}}\\
  &=\bigoplus_{i=0}^\infty \frac{T^{a+c}_{i-c}\cap T^b_{i-c+k-(b-a-c)}}{T^{a+c}_{i-c-1}\cap T^b_{i-c+k-(b-a-c)}}\\
  &=U^{k}_{T^{a+c},T^b}(-c)
\end{align*}
As a consequence of the identifies on the $V^k_{T^a,T^b}$ modules (a completely analogous argument exists using the transformations of the $U^k_{T^a,T^b}$ modules instead), we get for $c<k$ that
\[
\gr^k_{T^a,T^b}(M) = \frac{V^k_{T^a,T^b}}{V^{k-1}_{T^a,T^b}}=\frac{V^{k-c}_{T^{a+c},T^b}}{V^{k-c}_{T^{a+c},T^b}}=\gr^{k-c}_{T^{a+c},T^b}(M)
  \]
and for $k<b-a$
  \[
\gr^k_{T^a,T^b}(M) = \frac{V^k_{T^a,T^b}}{V^{k-1}_{T^a,T^b}}=\frac{V^{k}_{T^{a},T^{b+c}}}{V^{k-1}_{T^{a},T^{b+c}}}=\gr^{k}_{T^{a},T^{b+c}}(M)
  \]
\end{proof}

We are now ready to prove that the Bernstein dimension and multiplicity of an $\FF$-module are filtration invariant. For the remainder of this section, $R=\F_p[x_1,\ldots,x_n]$, $\FF$ is the ring of Frobenius operators over $R$, $\B$ is the Bernstein filtration on $\FF$, and $\A=\gr_\B(\FF)$.

\begin{theorem}\label{intergreat}
Let $M$ be an $\FF$-module, with $\Sigma$ and $\Omega$ two great $\B$-compatible filtrations on $M$. We assume, possibly after reindexing, that there is some $w\in \N$ such that $\Sigma_i\subseteq\Omega_i\subseteq\Sigma_{i+w}$ for all $i$. Let $T^a_i=\Sigma_{i+a}\cap\Omega_i$ be the family of interpolating filtrations between $\Sigma$ and $\Omega$. Each $T^a$ is a great filtration, and the modules $\gr^k_{T^a,T^b}(M)$, $U^k_{T^a,T^b}$, and $V^k_{T^a,T^b}$ are finitely presented over $\A$.
\end{theorem}
\begin{proof}
    We will prove these claims by induction on the linear gap $w$, beginning with the case $w=1$. As the only interpolating filtrations in this case are $T^0=\Sigma$ and $T^1=\Omega$ themselves, the claim about the greatness of interpolating filtrations is automatic. We have $U^0_{T^0,T^1}=\gr^1_{T^0,T^1}(M)(-1)$, $U^1_{T^0,T^1}=\gr_{T^0}(M)$ (which is finitely presented by hypothesis), $V^0_{T^0,T^1}=\gr^0_{T^0,T^1}(M)$, and $V^1_{T^0,T^1}=\gr_{T^1}(M)$ (also assumed to be finitely presented). So, we only need to show that $\gr^0_{T^0,T^1}(M)$ and $\gr^1_{T^0,T^1}(M)$ are finitely presented. We have homogeneous short exact sequences of graded $\A$-modules
\[
0\to \gr^0_{T^0,T^1}(M)\to \gr_{T^1}(M)\to \gr^1_{T^0,T^1}(M)\to 0
\]
\[
0\to \gr^1_{T^0,T^1}(M)(-1)\to \gr_{T^0}(M)\to \gr^0_{T^0,T^1}(M)\to 0
\]
By Lemma \ref{thetafacts}, these induce the following commutative diagrams (with exact rows) of homogeneous maps between graded $R$-modules:
\begin{center}{
  \begin{tikzcd}
    0\ar[r]& \displaystyle\mF_\A(\gr^0_{T^0,T^1}(M))\ar[r]\ar[d,"\Theta_{\gr^0_{T^0,T^1}(M)}"]& \mF_\A(\gr_{T^1}(M))\ar[d,"\Theta_{\gr_{T^1}(M)}"]\ar[r]& \mF_\A(\gr^1_{T^0,T^1}(M))\ar[d,"\Theta_{\gr^1_{T^0,T^1}(M)}"]\ar[r]& 0\\
    0\ar[r]& \gr^0_{T^0,T^1}(M)\ar[r]& \gr_{T^1}(M)\ar[r]& \gr^1_{T^0,T^1}(M)\ar[r]& 0
  \end{tikzcd}}
\end{center}
and
\begin{center}{
  \begin{tikzcd}
    0\ar[r]& \mF_\A(\gr^1_{T^0,T^1}(M))\ar[r]\ar[d,"\Theta_{\gr^1_{T^0,T^1}(M)}"]& \mF_\A(\gr_{T^0}(M))\ar[d,"\Theta_{\gr_{T^0}(M)}"]\ar[r]& \mF_\A(\gr^0_{T^0,T^1}(M))\ar[d,"\Theta_{\gr^0_{T^0,T^1}(M)}"]\ar[r]& 0\\
    0\ar[r]& \gr^1_{T^0,T^1}(M)\ar[r]& \gr_{T^0}(M)\ar[r]& \gr^0_{T^0,T^1}(M)\ar[r]& 0
  \end{tikzcd}}
\end{center}
Since $T^0$ and $T^1$ are great, Theorem \ref{thetaaction} implies that $C_{\gr_{T^0}(M)}$ and $C_{\gr_{T^1}(M)}$ are finitely generated $R$-modules, and thus, by the above diagrams, $C_{\gr^0_{T^0,T^1}(M)}$ and $C_{\gr^1_{T^0,T^1}(M)}$ are as well. By Theorem \ref{thetafg}\footnote{Since $\Sigma_i$ and $\Omega_j$ are always finite dimensional, and the $T^k$ filtrations are formed by intersections $\Sigma_{i+k}\cap\Omega_i$, with all $V^k_{T^a,T^b}$, $U^k_{T^a,T^b}$, and $\gr^k_{T^a,T^b}$ built from subquotients of these intersections, the hypotheses about all graded components being finite dimensional and vanishing in sufficiently large negative degrees are readily seen to hold for any of these modules.}, it follows that $\gr^0_{T^0,T^1}(M)$ and $\gr^1_{T^0,T^1}(M)$ are finitely generated over $\A$. The short exact sequences show them both embedding as graded submodules of finitely presented $\A$-modules, so Theorem \ref{subfpisfp} implies that they are themselves finitely presented. This completes the base case of the induction.

Fix $w\ge 2$, and suppose we already know that the claimed result of the theorem holds for any two great filtrations $A$ and $B$ on $M$ such that $A_i\subseteq B_i\subseteq A_{i+u}$ for all $i\in \N$, with $u$ fixed and strictly less than $w$. We are now given two great filtrations $\Sigma$ and $\Omega$ with $\Sigma_i\subseteq \Omega_i\subseteq \Sigma_{i+w}$, with various interpolating filtrations $T^k_i=\Sigma_{i+k}\cap \Omega_i$ for $0\leq k\leq w$. We claim that all finite presentation statements will follow if we can show that each $T^k$ for $0\leq k\leq w$ is a great filtration. Certainly, if each $T^k$ is assumed to be great, then for any $0\leq a<b\leq w$ with $b-a<w$, and any $0\leq k\leq b-a$, the induction hypothesis immediately applies to the modules $U^k_{T^a,T^b}$, $V^k_{T^a,T^b}$, and $\gr^k_{T^a,T^b}(M)$, since $T^a_i\subseteq T^b_i\subseteq T^a_{i+(b-a)}$ for all $i$ with $(b-a)<w$. The only modules to which our induction hypothesis does not obviously apply right away are $U^k_{T^0,T^w}$, $V^k_{T^0,T^w}$, and $\gr^k_{T^0,T^w}(M)$ for $0\leq k\leq w$. Under the transformations of Lemma \ref{filtransformations}, however, if $k\geq 1$, then $V^k_{T^0,T^w}=V^{k-1}_{T^{1},T^w}$ with $V^0_{T^0,T^w}=V^0_{T^0,T^{w-1}}$ while $U^k_{T^0,T^w}=U^{k-1}_{T^0,T^{w-1}}$ and $U^0_{T^0,T^w}=U^0_{T^1,T^w}(-1)$. All of these modules are finitely presented by induction. For $0\leq k<w$, we additionally have $\gr^k_{T^0,T^w}(M)=\gr^k_{T^0,T^{w-1}}(M)$ and $\gr^w_{T^0,T^w}(M)=U^0_{T^0,T^w}(+w)=U^0_{T^1,T^w}(+w-1)=\gr^{w-1}_{T^1,T^w}(M)$, and these too are finitely presented by induction.

So, it remains to show that each interpolating filtration $T^k$ is great. By induction, it is enough to show that $T^{w-1}$ is great, since we could then apply the inductive hypothesis to $T^0$ and $T^{w-1}$, which are separated by a strictly shorter linear gap ($T^0_i\subseteq T^{w-1}_i\subseteq T^0_{i+(w-1)}$ for all $i$) and whose interpolating filtrations are $T^0_{i+k}\cap T^{w-1}_i= T^0_{i+k}\cap T^w_i=T^{k}_i$ for $0\leq k\leq w-1$.

We have chains of homogeneous inclusions
\[
\gr^0_{T^0,T^w}=V^0_{T^0,T^w}\subseteq V^1_{T^0,T^w}\subseteq\cdots\subseteq V^w_{T^0,T^w}=\gr_{T^w}(M)
\]
and
\[
\gr^w_{T^0,T^w}(-w)=U^0_{T^0,T^w}\subseteq U^1_{T^0,T^w}\subseteq\cdots\subseteq U^w_{T^0,T^w}=\gr_{T^0}(M).
\]
The greatness of $T^0$ and $T^w$ imply that all the structural kernels $K_{U^k_{T^0,T^w}}$ and $K_{V^k_{T^0,T^w}}$ have finite length, including $K_{\gr^w_{T^0,T^w}(M)}$ and $K_{\gr^0_{T^0,T^w}(M)}$. The modules $\gr^0_{T^0,T^w}(M)$ and $\gr^w_{T^0,T^w}(M)$ are homomorphic images of $\gr_{T^0}(M)$ and $\gr_{T^w}(M)$, respectively, and hence, are finitely generated. But they are also submodules of the finitely presented modules $\gr_{T^w}(M)$ and $\gr_{T^0}(M)$, respectively, so finite generation implies their finite presentation, by Theorem \ref{subfpisfp}.

By Lemma \ref{filtransformations}, we have $V^{w-2}_{T^0,T^w}=V^{w-2}_{T^0,T^{w-1}}$, and by Lemma \ref{uvfilt}, there is a short exact sequence
\[
0\to V^{w-2}_{T^0,T^{w-1}}\to \gr_{T^{w-1}}(M) \to \gr^{w-1}_{T^0,T^{w-1}}(M)\to 0
\]
The module $\gr^{w-1}_{T^0,T^{w-1}}(M)=U^0_{T^0,T^{w-1}}(+w-1)$ is a graded submodule of $\gr_{T^0}(M)(+w-1)$, which is finitely presented. It follows that $K_{\gr^{w-1}_{T^0,T^{w-1}}(M)}$ has finite length, and since $V^{w-2}_{T^0,T^{w-1}}=V^{w-2}_{T^0,T^w}$ is a graded submodule of $\gr_{T^w}(M)$, we know that $K_{V^{w-2}_{T^0,T^{w-1}}}$ has finite length. The above exact sequnce implies that $K_{\gr_{T^{w-1}}(M)}$ has finite length as well. We now consider the short exact sequences
\[
0\to \gr^0_{T^{w-1},T^w}(M)\to \gr_{T^w}(M)\to \gr^1_{T^{w-1},T^w}(M)\to 0
\]
\[
0\to \gr^1_{T^{w-1},T^w}(M)(-1)\to \gr_{T^{w-1}}(M)\to \gr^0_{T^{w-1},T^w}(M)\to 0
\]
Since $K_{\gr_{T^{w-1}}(M)}$ and $K_{\gr_{T^{w}}(M)}$ have finite length, we see that $K_{\gr^0_{T^{w-1},T^w}(M)}$ and $K_{\gr^1_{T^{w-1},T^w}(M)}$ also have finite length. Using the finite generation of $C_{\gr_{T^w}(M)}$ along with the exact sequence
\[
K_{\gr^1_{T^{w-1},T^w}(M)}\to C_{\gr^0_{T^{w-1},T^w}(M)}\to C_{\gr_{T^w}(M)}\to C_{\gr^1_{T^{w-1},T^w}(M)}\to 0
\]
we conclude that $C_{\gr^0_{T^{w-1},T^w}(M)}$ and $C_{\gr^1_{T^w}(M)}$ are finitely generated over $R$. By Theorem \ref{thetafg}, it follows that $\gr^0_{T^{w-1},T^w}(M)$ and $\gr^1_{T^{w-1},T^w}(M)$ are finitely generated over $\A$. The module $\gr^1_{T^{w-1},T^w}(M)$ is a quotient of the finitely presented module $\gr_{T^w}(M)$ by the finitely generated submodule $\gr^0_{T^{w-1},T^w}(M)$, so it must itself be finitely presented. The module $\gr^0_{T^{w-1},T^w}(M)$ is a finitely generated submodule of the finitely presented module $\gr_{T^w}(M)$, so it is finitely presented by Theorem \ref{subfpisfp}. But then the second of our short exact sequences places $\gr_{T^{w-1}}(M)$ between two finitely presented $\A$-modules, $\gr^0_{T^{w-1},T^w}(M)$ and $\gr^1_{T^{w-1},T^w}(M)$, so $\gr_{T^{w-1}}(M)$ is itself finitely presented, i.e., $T^{w-1}$ is a great filtration, as desired.
\end{proof}

\begin{dfn}
    Let $\FF=\F_p[x_1,\ldots,x_n]\langle F\rangle$, let $\B_i$ be the Bernstein filtration on $\FF$, and let $M$ be a left $\FF$-module. Call $M$ \textit{great} if there exists a great $\B$-compatible filtration on $M$.
  \end{dfn}

By Lemma \ref{greatfp}, a great $\FF$-module is automatically finitely presented over $\FF$. We will study great $\FF$-modules in more detail in the next section. For now, we observe the following consequence of Theorem \ref{intergreat}.

\begin{theorem}\label{welldefined}
Let $M$ be a great $\FF$-module, with $\Sigma$ and $\Omega$ two great filtrations on $M$. Then $\delta(\gr_\Sigma(M))=\delta(\gr_\Omega(M))$ and $e(\gr_\Sigma(M))=e(\gr_\Omega(M))$.
\end{theorem}
\begin{proof}
  Reindex as necessary in order to assume that there is a $w\in \N$ such that $\Sigma_i\subseteq \Omega_i\subseteq \Sigma_{i+w}$ for all $i$. By Theorem \ref{intergreat}, each interpolating filtration $T^k_i=\Sigma_{i+k}\cap \Omega_i$ for $0\leq k\leq w$ is great, so it is enough to show that if $A$ and $B$ are great filtrations satisfying $A_i\subseteq B_i\subseteq A_{i+1}$ for all $i$, then $\delta(\gr_A(M))=\delta(\gr_B(M))$ and $e(\gr_A(M))=e(\gr_B(M))$. Given two such great filtrations $A$ and $B$, we have the following homogeneous short exact sequences of finitely presented graded $\A$-modules
  \[
0\to \gr^0_{A,B}(M)\to \gr_{B}(M)\to \gr^1_{A,B}(M)\to 0
\]
\[
0\to \gr^1_{A,B}(M)(-1)\to \gr_{A}(M)\to \gr^0_{A,B}(M)\to 0
\]
By Theorem \ref{edcalc}, 
\[
\delta(\gr_B(M))=\max\{\delta(\gr^0_{A,B}(M),\gr^1_{A,B}(M)\}=\delta(\gr_A(M))
\]
and
  \[
e(\gr_A(M))=e(\gr_B(M))=\begin{cases}
    e(\gr^0_{A,B}(M)) & \delta(\gr^0_{A,B}(M))>\delta(\gr^1_{A,B}(M))\\
    e(\gr^0_{A,B}(M))+e(\gr^1_{A,B}(M)) & \delta(\gr^0_{A,B}(M))=\delta(\gr^1_{A,B}(M))\\
    e(\gr^1_{A,B}(M)) & \delta(\gr^0_{A,B}(M))<\delta(\gr^1_{A,B}(M))
\end{cases}
\]
\end{proof}

\begin{dfn}
Let $M$ be a great $\FF$-module. The \textit{Bernstein dimension} and \textit{multiplicity} of $M$ are defined to be the values $\delta(\gr_\Omega(M))$ and $e(\gr_\Omega(M))$, respectively, obtained from some (equivalently, any) great filtration $\Omega$ on $M$
\end{dfn}

  %%%%%%%%%%%%%%%%%%%%%%%%%%%%%%%%%%%%%%%%%%%%%%%%%%%%%%%%%%%%%%%%%%%%%%%%%%%%%%%%%%%%%%%%%%%%%%%%%55
  %%%%%%%%%%%%%%%%%%%%%%%%%%%%%%%%%%%%%%%%%%%%%%%%%%%%%%%%%%%%%%%%%%%%%%%%%%%%%%%%%%%%%%%%%%%%%%%%%%%
  %%%%%%%%%%%%%%%%%%%%%%%%%%%%%%%%%%%%%%%%%%%%%%%%%%%%%%%%%%%%%%%%%%%%%%%%%%%%%%%%%%%%%%%%%%%%%%%%%%%

  %%% Finally, at long last we can deal with

               %%%   %%   %%%%%  %%     %%%%   %%%%%%  %%%%%  %%%%%%%
               %%%   %%    %%    %%     %%  %  %%      %%  %    %%
               %%%%%%%%    %%    %%     %%%%   %%%%%   %%%%     %%
               %%%   %%    %%    %%     %%  %  %%      %%  %    %%
               %%%   %%   %%%%%  %%%%%  %%%%   %%%%%   %%   %   %%

  %%%%%%%%%%%%%%%%%%%%%%%%%%%%%%%%%%%%%%%%%%%%%%%%%%%%%%%%%%%%%%%%%%%%%%%%%%%%%%%%%%%%%%%%%%%%%%%%%55
  %%%%%%%%%%%%%%%%%%%%%%%%%%%%%%%%%%%%%%%%%%%%%%%%%%%%%%%%%%%%%%%%%%%%%%%%%%%%%%%%%%%%%%%%%%%%%%%%%%%
  %%%%%%%%%%%%%%%%%%%%%%%%%%%%%%%%%%%%%%%%%%%%%%%%%%%%%%%%%%%%%%%%%%%%%%%%%%%%%%%%%%%%%%%%%%%%%%%%%%%

\section{The dimension and multiplicity of an $\FF$-module}

Throughout this section, $R=\F_p[x_1,\ldots,x_n]$, $\FF$ is the ring of Frobenius operators over $R$, $\B$ is the Bernstein $\N$-filtration on $\FF$, and all filtrations on $\FF$-modules are assumed to be $\B$-compatible and indexed by $\N$.

Let $M$ be an $R$-module, and let $\mathcal{F}_R(M)$ denote the base change of $M$ along the Frobenius homomorphism $F:R\to R$. As an abelian group $\mathcal{F}_R(M)$ can be described as a direct sum of copies of $M$, namely, $\mathcal{F}_R(M)=\bigoplus_{v\in\N_{<p}^n} M\varepsilon_v$ where $\varepsilon_v$ denotes a formal basis vector. The structure of $\mathcal{F}_R(M)$ as an $R$-module is given by
\[
x_i\cdot (m\varepsilon_v):=
\begin{cases}
  m\varepsilon_{v+e_i} & \text{ if } v_i<p-1\\
  x_im\varepsilon_{v-(p-1)e_i}& \text{ if } v_i=p-1
\end{cases}
\]
where $e_i$ denote the $i$th standard basis vector of $\N^n$. Since $R$ is a polynomial ring, we are, of course, in the setting where $\mF_R(-)$ is an exact functor.
\begin{dfn}
  Let $M$ be an $\FF$-module. The map $\theta_M:\mathcal{F}_R(M)\to M$ defined by $\theta_M(m\varepsilon_v)=x^vF(m)$ for all $v\in\N_{<p}^n$ is called the \textit{structure morphism} of $M$. If $\theta_M$ is an isomorphism, $M$ is called \textit{unit} (see \cite{lyufmod}, \cite{emertonkisin}, \cite{blicklethesis}).
\end{dfn}
It is not hard to see that $\theta_M$ is $R$-linear. In fact, if $\theta:\mathcal{F}_R(M)\to M$ is any $R$-linear map whatsoever, then the action $F(m):=\theta(m\varepsilon)$ makes $M$ into an $\FF$-module with structure map $\theta_M=\theta$.

The proof of the following straightforward lemma is essentially identical to the analogous statement in Theorem \ref{thetaaction}(i).
\begin{lem}
If $M$ is a finitely generated $\FF$-module, then $\Coker(\theta_M)$ is a finitely generated $R$-module. In fact, the images of any $\FF$-module generating set for $M$ form an $R$-module generating set for $\Coker(\theta_M)$.
\end{lem}
\begin{proof}
Let $m_1,\ldots,m_t$ be $\FF$-module generators for $R$. The cokernel $\Coker(\theta_M)$ is spanned by the images mod $\text{Im}(\theta_M)$ of elements of the form $cm_1,\ldots,cm_t$ where $c$ ranges over all monomials of $\FF$. However, if $c$ has positive $F$-order, then $cm_i$ is in the image of $\theta_M$, so the only elements $cm_i$ with nonzero image in $\Coker(\theta_M)$ are those in the span of $m_1,\ldots,m_t$.
\end{proof}

If $M$ is an $\FF$-module, we can make $\mathcal{F}_R(M)$ into an $\FF$-module as well, though $\theta_M:\mF_R(M)\to M$ is not $\FF$-linear. The action of $F$ on $\mathcal{F}_R(M)$ is given by
\[
F\cdot (m\varepsilon_v):=(x^vFm)\varepsilon_0
\]
If $\Omega$ is a $\B$-compatible filtration on $M$, then the Schreyer pullback filtration $\Omega^*_i:=\bigoplus_{v\in \N_{<p}^n}\Omega_{i-|v|-1}\varepsilon_v$ is a $\B$-compatible filtration with respect to the $\FF$-module structure we've given $\mathcal{F}_R(M)$, and the map $\theta_M$ clearly satisfies $\theta_M(\Omega^*_i)\subseteq \Omega_i$. 
\begin{prop}\label{grfrobcomp}
  Let $\Omega$ be a $\B$-compatible filtration on an $\FF$-module $M$, and let $\Omega^*$ be the Schreyer pullback filtration on the $\FF$-module $\mF_R(M)$.
  \begin{enumerate}
    \item There is an isomorphism of graded $R$-modules
\[
\gr_{\Omega^*}(\mF_R(M))\xrightarrow{\sim}\mF_\A(\gr_\Omega(M))
\]
where $\mF_\A(\gr_\Omega(M))$ is the module introduced in Definition \ref{mfaction}. 
\item Under the identification $\gr_{\Omega^*}(\mF_R(M))=\mF_\A(\gr_\Omega(M))$, it holds that $\gr(\theta_M)=\Theta_{\gr_\Omega(M)}$, where $\Theta_{\gr_\Omega(M)}$ is the structure map of Definition \ref{grstruct}. That is, the diagram below commutes.
  \begin{center}
    \begin{tikzcd}
      \gr_{\Omega^*}(\mF_R(M))\ar[d,"\text{\rotatebox{90}{$\sim$}}"']\ar[r,"\gr(\theta_M)"]&\gr_\Omega(M)\ar[d,equals]\\
\mF_\A(\gr_\Omega(M))\ar[r,"\Theta_{\gr_\Omega(M)}"]&\gr_\Omega(M)
\end{tikzcd}
\end{center}
\end{enumerate}
\end{prop}
\begin{proof}
  (i) For $m\in M$ of $\Omega$-degree $d$ and any $v\in\N_{<p}^n$, we have $\deg_{\Omega^*}(m\varepsilon_v)=d+|v|+1$. The map $\gr_{\Omega^*}(\mF_R(M))\to \mF_\A(\gr_\Omega(M))$ defined on each component by $[m\varepsilon_v]^{\Omega^*}_{d+|v|+1}\to [m]^\Omega_d\varepsilon_v$ is easily seen to be an isomorphism at the level of graded vector spaces. Concerning $R$-actions, if $v_i=p-1$, then
  \[
  \deg_{\Omega^*}(x_i\cdot m\varepsilon_v)=\deg_{\Omega^*}(x_im\varepsilon_{v-(p-1)e_i})=(d+1)-p
  \]
  so that $[x_i]^\B_1\cdot [m\varepsilon_v]^{\Omega^*}_d=0$, in exact agreement the action of $x_i$ on $\sigma^\Omega(m)\varepsilon_v$ in $\mF_\A(M)$. If $v_i<p-1$, on the other hand, then
  \[
  [x_i]^{\B}_1 [m\varepsilon_v]^{\Omega^*}_d=[m\varepsilon_{v+e_i}]^{\Omega^*}_{d+1}=\sigma^\Omega(m)\varepsilon_{v+e_i}
  \]
  which again agrees with the action of $x_i$ on $\sigma^\Omega(m)\varepsilon_v$ in $\mF_\A(M)$.

  (ii) The map induced by $\theta_M$ taking $\gr_{\Omega^*}(\mF_R(M))\to \gr_\Omega(M)$ is
\[
\gr(\theta_M)([m\varepsilon_v]^{\Omega^*}_i)=[\theta_M(m\varepsilon_v)]^\Omega_i = [x^vFm]^\Omega_i=x^vf[m]^\Omega_{i-|v|-1}
\]
Under the identification $\gr_{\Omega^*}(\mathcal{F}_R(M))=\mF_\A(\gr_\Omega(M))$, we recognize the above map at once as the structure morphism of $\gr_\Omega(M)$
  \end{proof}

It will be useful to understand how the $\Omega$ and $\Omega^*$ filtrations interact with the standard degree filtration $\mD$ on $R=\F_p[x_1,\ldots,x_n]$. It is certainly true that $\mD_d\cdot \Omega^*_i\subseteq \Omega^*_{i+d}$, but for $d\gg 0$, we can say something stronger.

\begin{lem}\label{domegastar}
Let $\mD$ be the standard $\N$-filtration on $R=\F_p[x_1,\ldots,x_n]$ by degree. Let $\Omega$ be a $\B$-compatible filtration on an $\FF$-module $M$, and let $\Omega^*$ be the Schreyer pullback filtration on $\mF_R(M)$. For all $d\in \N$, it holds that
  \[
    \mD_d\cdot \Omega^*_i\subseteq \Omega^*_{i+\lfloor d/p\rfloor + n(p+1)}
  \]
\end{lem}
\begin{proof}
 Given a vector $v\in \N^n$, let $r(v)\in \N_{<p}^n$ and $c(v)\in\N^n$ be defined by $v=r(v)+pc(v)$ with $0\leq r(v)_i<p$ for all $1\leq i\leq n$. Let $x^\alpha$ be a monomial of degree at most $d$. For $m\in M$ and $v\in\N_{<p}^n$, we have
  \[
  x^\alpha\cdot m\varepsilon_v = (x^{p c(\alpha)}x^{r(\alpha)} x^v)\cdot (m\varepsilon_0) = (x^{c(\alpha)+c(r(\alpha)+v)}m)\varepsilon_{r(\alpha+v)}
  \]
  so that
  \begin{align*}
    \deg_{\Omega^*}(x^\alpha\cdot m\varepsilon_v)&=|c(\alpha)|+|c(r(\alpha)+v)|+\deg_\Omega(m)+|r(\alpha+v)|+1\\
    &=\deg_{\Omega^*}(m) -|v|+|c(\alpha)|+|c(r(\alpha)+v)|+|r(\alpha+v)|\\
    &\leq \deg_{\Omega^*}(m) +\lfloor\,|\alpha|/p\,\rfloor+|c(r(\alpha)+v)|+|r(\alpha+v)|
  \end{align*}
  The result follows at once upon bounding  $|r(\alpha+v)|\leq n(p-1)$ and
  \[
  |c(r(\alpha)+v)|\leq \lfloor |r(\alpha)+v|/p\,\rfloor\leq \lfloor (2n(p-1))/p\,\rfloor\leq 2n
  \]
\end{proof}
If $\Omega^*$ is the Schreyer pullback filtration on $\mF_R(M)$, then the image of $\Omega$ defines a $\mD$-compatible filtration $\Omega_1:=\Omega/\Omega\cap\text{Im}(\theta_M)$ on $\Coker(\theta_M)$, while the restriction of $\Omega^*$ defines a $\mD$-compatible filtration $\Omega_0:=\Omega^*\cap \Ker(\theta_M)$ on $\Ker(\theta_M)$. We would like to compare the associated graded $R$-modules $\gr_{\Omega_1}(\Ker(\theta_M))$ and $\gr_{\Omega_1}(\Coker(\theta_M))$ to the graded $R$-modules $\Ker(\Theta_{\gr_\Omega(M)})$ and $\Coker(\Theta_{\gr_\Omega(M)})$. Let $\mf{m}=(x_1,\ldots,x_n)$ denote the homogeneous maximal ideal of $R$.
\begin{lem}\label{cokercompare}
Let $\Omega$ be a $\B$-compatible filtration on an $\FF$-module $M$, and let $\Omega^*$ be the Schreyer pullback filtration on the $\FF$-module $\mF_R(M)$. Let $\Omega_0=\Omega^*\cap\Ker(\theta_M)$ and $\Omega_1=\Omega/\Omega\cap\text{Im}(\theta_M)$ denote the $\mD$-compatible filtrations induced by $\Omega^*$ and $\Omega$ on $\Ker(\theta_M)$ and $\Coker(\theta_M)$, respectively. There are homogeneous short exact sequences of graded $R$-modules
\[
0\to N \to \text{Coker}(\Theta_{\gr_\Omega(M)})\to \gr_{\Omega_1}(\text{Coker}(\theta_M))\to 0
\]
and 
\[
0\to \gr_{\Omega_0}(\Ker(\theta_M)) \to \Ker(\Theta_{\gr_\Omega(M)})\to Q\to 0
\]
where $N$ and $Q$ are graded $\mf{m}$-torsion modules. In particular, if $\Omega$ is a good filtration, then $\text{Coker}(\theta_M)$ is a finitely generated $R$-module of the same Krull dimension as $\Coker(\Theta_{\gr_\Omega(M)})$. If $\Omega$ is a great filtration, then $\Ker(\theta_M)$ is a finite length $R$-module. 
\end{lem}
\begin{proof}
These short exact sequences arise at once from the following short exact sequences defined below on each graded component of the modules in question.
  \[
  0\to \frac{\Omega_i\cap\text{Im}(\theta_M)}{\theta_M(\Omega^*_i)+\Omega_{i-1}\cap\text{Im}(\theta_M)}\to\frac{\Omega_i}{\theta_M(\Omega^*_i)+\Omega_{i-1}}\to\frac{\Omega_i}{\Omega_i\cap\text{Im}(\theta_M)+\Omega_{i-1}}\to 0
  \]
  and
    \[
  0\to \frac{\Omega^*_i\cap\text{Ker}(\theta_M)}{\Omega^*_{i-1}\cap\Ker(\theta_M)}\to\frac{\Omega^*_i\cap\theta_M^{-1}(\Omega_{i-1})}{\Omega^*_{i-1}\cap\theta^{-1}_M(\Omega_{i-1})}\to\frac{\Omega^*_i\cap\theta_M^{-1}(\Omega_{i-1})}{\Omega^*_i\cap\text{Ker}(\theta_M)+\Omega^*_{i-1}\cap\theta^{-1}_M(\Omega_{i-1})}\to 0
  \]
  Since $\Ker(\Theta_{\gr_\Omega(M)})$ is a submodule of the $\mf{m}$-torsion module $\mF_\A(\gr_\Omega(M))$, the fact that $Q$ is $\mf{m}$-torsion is obvious. To show that $N$ is $\mf{m}$-torsion, we must show that if $m\in \Omega_i\cap\text{Im}(\theta_M)$, then there is some $d\gg 0$ such that for any monomial $x^\alpha$ of degree $d$, we have $x^\alpha\cdot m\in \theta_M(\Omega^*_{i+d}+\Omega_{i+d-1}\cap\text{Im}(\theta_M)$. Since $m\in \text{Im}(\theta_M)$, we can find some $\xi\in \mF_R(M)$, say of $\Omega^*$-degree $j$, with $m=\theta_M(\xi)$. Now, $x^\alpha\cdot m = \theta_M(x^\alpha\cdot \xi)$. By Lemma \ref{domegastar}, we have $\deg_{\Omega_*}(x^\alpha\cdot \xi)\leq j+\lfloor d/p\rfloor + n(p+1)$. We would be finished if $j+\lfloor d/p\rfloor + n(p+1) \leq i+d$, as this would give $x^{\alpha}\cdot m\in\theta_M(\Omega^*_{i+d})$. It is enough to choose $d$ large enough to satisfy $j-1+n(p+1)\leq i+d-d/p$, i.e., $d\geq \frac{p}{p-1}(j-i+n(p+1)-1)$.

  Regarding the claims about $\Coker(\theta_M)$ and $\Ker(\theta_M)$, by Lemma \ref{goodfg}, the finite generation of their associated graded modules or $\gr_\D(R)=R$ immediately implies their finite generation over $R$. That the Krull dimension of a finitely generated $R$-module is equal to the Krull dimension of its associated graded module with respect to a good $\mD$-compatible filtration is well known. So, we see that when $\Omega$ is good, $\Coker(\theta_M)$ has the same dimension as $\Coker(\Theta_{\gr_{\Omega}(M)})$, and that when $\Omega$ is great, $\Ker(\theta_M)$ has finite length.
\end{proof}

\begin{theorem}\label{greatconverse}
If $M$ is a finitely generated $\FF$-module such that $\Ker(\theta_M)$ is a finite length $R$-module, then $M$ admits a great filtration.
\end{theorem}
\begin{proof}
  Let $m_1,\ldots,m_t$ be a generating set for $M$ over $\FF$, and let us consider the (obviously good) filtration $\Omega_i = \B_i\cdot\{\xi_1,\ldots,\xi_t\}$. Let $N=\gr_\Omega(M)$. By Theorem \ref{thetaaction}, $\Coker(\Theta_N)$ is finitely generated over $\A$, and by Theorem \ref{thetafg}, to show that $N$ is finitely presented, it is enough to show that $\Ker(\Theta_N)$ is a finite length $R$-module.
  
  Let $g=\sum_{v\in\N_{<p}^n} g_v\varepsilon_v\in\mF_\A(M)$ be a homogeneous element of $\Ker(\Theta_N)$ with degree $d$. We must show that $g=0$ when $d\gg 0$. For each $v$, let $G_v$ be a lift of $g_v$ to $\B_{i-|v|-1}$, and let $G=\sum_{v\in\N_{<p}^n} G_v\varepsilon_v\in \mF_R(M)$. By Lemma \ref{grfrobcomp}(ii), to say that $\Theta_N(g)=0$ is precisely to say that $\theta_M(G)\in \Omega_{d-1}$, say $\theta_M(G)=\sum_{i=1}^t H_i m_i$ with $H_i\in\B_{d-1}$. For each $i$, we may write $H_i=A_i(x)+\sum_{v\in\N_{<p}^n} x^vFB_{i,v}$ with $A_i(x)\in R$ of degree $<d$, and $B_{i,v}\in\B_{d-|v|-2}$ for all $v$, breaking up $\theta_M(G)$ into two different sums as
  \[
  \theta_M(G)=\sum_{i=1}^tA_i(x)m_i +\textstyle\sum_{i=1}^t\sum_{v\in\N_{<p}^n} x^vFB_{i,v}
  \]
  The element $B=\sum_{i=1}^tB_{i,v}m_i\varepsilon_v$ in $\Omega^*_{d-1}$ maps under $\theta_M$ to the second of these, giving
\[
\theta_M(G-B)=\textstyle\sum_{i=1}^t A_i(x)m_i
\]
Let $M_0$ be the $R$-submodule of $M$ generated by $m_1,\ldots,m_t$, and consider the $R$-module $M_1=M_0\cap \text{Im}(\theta_M)$, to which our sum $\sum_{i=1}^t A_i(x)m_i$ belongs. Letting $\mD$ denote the standard degree filtration on $R$, we will compare two $\mD$-compatible good filtrations on $M_1$. First, pick a set of $R$-module generators $y_1,\ldots,y_s$ for $M_1$ and let $\mathcal{Y}_i=\mD_i\cdot \{y_1,\ldots,y_s\}$. Second, filter $M_0$ by $\Sigma_i=\mD_i\cdot\{m_1,\ldots,m_t\}$, and consider the filtration $\Sigma\cap M_1$. Since $R=\gr_{\mD}(R)$ is Noetherian and $\gr_\Sigma(M_0)$ is finitely generated, it is clear that $\Sigma\cap M_1$ is a good filtration, and hence, is linearly comparable to $\mathcal{Y}$. Let $w\in \N$ be such that $\Sigma_{i}\cap M_1\subseteq \mathcal{Y}_{i+w}$ for all $i$. Our sum $\sum_{i=1}^t A_i(x)m_i$ evidently belongs to $\Sigma_{d-1}\cap M_1$, and hence, is also contained in $\mathcal{Y}_{d+w-1}$, say $\sum_{i=1}^tA_i(x)m_i=\sum_{i=1}^s C_i(x)y_i$ with $\deg(C_i(x))\leq d+w-1$ for all $i$.

For each $i$, write $y_i=\theta_M(z_i)$ for some $z_i\in\mathcal{F}_R(M)$, and we obtain
\[
\theta_M(G-B -\textstyle\sum_{i=1}^s C_i(x) z_i)=0
\]
Let $\zeta_1,\ldots,\zeta_r$ be an $\F_p$-vector space basis for the finite-length $R$-module $\Ker(\theta_M)$, so that
\[
G-B-\textstyle\sum_{i=1}^s C_i(x) z_i = \sum_{i=1}^r \beta_i \zeta_i
\]
for some $\beta_i\in\F_p$. Since the $R$-span of $\zeta_1,\ldots,\zeta_r$ is finite-dimensional, we automatically have $\deg_{\Omega^*}(\sum_{i=1}^r \beta_i \zeta_i)<d$ when $d$ is sufficiently large. Letting $j$ be the maximum of the $\Omega^*$-degrees of the elements $z_1,\ldots,z_s$, we have by Lemma \ref{domegastar},
\[
\deg_{\Omega^*}(\textstyle\sum_{i=1}^s C_i(x) z_i)\leq j+\lfloor\, (d+w-1)/p\,\rfloor+n(p+1)
\]
If $p(j+n(p+1))<pd-d-w+1$, i.e., $d>\frac{p}{p-1}(j+n(p+1))+\frac{w-1}{p-1}$, then the $\Omega^*$-degree of $\sum_{i=1}^s C_i(x)z_i$ will be strictly less than $d$. Finally, as we remarked earlier, the $\Omega^*$-degree of $B$ is at most $d-1$. Taken together, if $d$ is sufficiently large, we must have
\[
\deg_{\Omega^*}(G)=\deg_{\Omega^*}(B+\textstyle\sum_{i=1}^sC_i(x)z_i+\sum_{i=1}^r\beta_i\zeta_i)<d
\]
It follows that $g=[G]^{\Omega^*}_d=0$ for $d\gg 0$, as desired.
\end{proof}

The following is now apparent.

\begin{theorem}\label{greatclass}
  Let $M$ be a finitely generated $\FF$-module with structure morphism $\theta_M$. Then $M$ is a great $\FF$-module if and only if $\Ker(\theta_M)$ is finite length over $R$. If this is the case, then for any great $\B$-filtration $\Omega$ on $M$, we have $\delta_1(\HS_{\gr_\Omega(M)})=\dim\Coker(\theta_M)$.
\end{theorem}
\begin{proof}
Corollary \ref{dimfromhs}, Theorem \ref{welldefined}, Lemma \ref{cokercompare}, and Theorem \ref{greatconverse}.
\end{proof}
\begin{cor}\label{hololist}
  Let $M$ be a finitely generated $\FF$-module with structure morphism $\theta_M$. The following are equivalent.
  \begin{enumerate}[(i)]
  \item $M$ is a great $\FF$-module of Bernstein dimension $0$.
  \item Both $\Ker(\theta_M)$ and $\Coker(\theta_M)$ are finite length $R$-modules.
  \item There exists a $\B$-filtration $\Omega$ on $M$ such that both $\Ker(\Theta_{\gr_{\Omega}(M)})$ and $\Coker(\Theta_{\gr_{\Omega}(M)})$ are finite length $R$-modules.
  \item There exists a $\B$-compatible filtration $\Omega$ on $M$ such that $M$ admits a finite $\Omega$-Gr\"obner basis $\G=(g_1,\ldots,g_t)$ such that $\inn(\Syz(\G))=I_1\oplus\cdots\oplus I_t$ is finitely generated, and each contraction $I_i\cap R$ contains a power of $\mf{m}$.
  \item There exists a great $\B$-compatible filtration $\Omega$ on $M$ such that the Hilbert series of $\gr_\Omega(M)$ has the form $a(t)/g_{p,n}(t)$, for some polynomial $a(t)$.
  \end{enumerate}
  In every instance above, ``there exists a [...] filtration'' can be replaced with ``for all [...] filtrations, of which there is at least one.''
\end{cor}

\begin{dfn}
Call a great $\FF$-module \textit{$F$-holonomic} if it satisfies any one of the conditions of the above corollary.
\end{dfn}
Using condition (ii), the following is apparent.
\begin{cor}
  Finitely generated unit $\FF$-modules are $F$-holonomic. In particular, for any ideal $I$ and any $i\geq 0$, the local cohomology module $H^i_I(R)$ is $F$-holonomic.
\end{cor}

\begin{theorem}\label{extcat}
Let $M$ and $N$ be $F$-holonomic $\FF$-modules. For any $\FF$-linear map $\varphi:M\to N$, the kernel, cokernel, and image of $\varphi$ are $F$-holonomic. An extension of two $F$-holonomic $\FF$-modules is $F$-holonomic. Multiplicity $M\mapsto e(M)$ is an additive function on the category of $F$-holonomic $\FF$-modules.
\end{theorem}
\begin{proof}
  If $0\to M\to N\to Q\to 0$ is an exact sequence of $\FF$-modules, the commutative diagram with exact rows
  \begin{center}{
    \begin{tikzcd}
      0\ar[r] & \mF_R(N)\ar[r]\ar[d,"\theta_N"]& \mF_R(M)\ar[r]\ar[d,"\theta_M"]& \mR_R(Q)\ar[r]\ar[d,"\theta_Q"]& 0\\
      0\ar[r] & N\ar[r]& M\ar[r]& Q\ar[r]& 0
      \end{tikzcd}}
  \end{center}
  makes it clear that if any two of the modules $M, N, Q$ are $F$-holonomic, then so is the third. If $\varphi:M\to N$ is any $\FF$-linear, let $I$ denote its image. The $F$-holonomicity of $M$ implies that $\Coker(\theta_I)$ has finite length, and the $F$-holonomicity of $N$ implies that $\Ker(\theta_I)$ has finite length. The result for $\Ker(\varphi)$ and $\Coker(\varphi)$ now follows from the two-out-of-three property on exact sequences. The additivity of multiplicity is apparent from Theorem \ref{edcalc}.
\end{proof}

\begin{ex}
Let $R=\F_2[x]$. As an $\FF$-module, $R$ has the presentation $R = \FF/\FF(F-1)$. Buchberger's algorithm shows that the associated graded module $\A/J$ has $\text{in}(J)=\A(f,fx,x^4)$, and the ideal $\inn(J)$ has a graded free resolution of the form
    \[
    0\to \A(-3)\oplus \A(-4)\to \A(-1)\oplus \A(-2)\oplus \A(-4)\to \A\to \A/\inn(J)\to 0
    \]
          so that
      \[
      \HS_{\gr(R)}(t) = \frac{1-t-t^2+t^3}{(1-t)(1-t-t^2)}=\frac{1-t^2}{(1-t-t^2)}=1+t+t^2+2t^3+3t^4+5t^6+8t^7+\cdots
      \]
      We see that $R$ has Bernstein dimension $0$ (so it is $F$-holonomic) and multiplicity $0$.
      \vspace{0.5em}
    \begin{center}
      
   % \vspace{-6.0em}
    \begin{tikzpicture}[scale=0.9, every node/.style={scale=1.0}]
\makeatletter
%      \clip (-7.2,-7.2) rectangle (6.8,6);
%\node at (1,0) {\begin{tabular}{|c||c|c|c|c|c|c|c|c|c|}\hline cx & 0 & 1 & 2 & 3 & 4 & 5 & 6 & 7 & 8\\\hline
%                        mon.ct. & 1 & 1 & 1 & 2 & 3 & 5 & 8 & 13 & 21\\\hline\end{tabular}};
      \draw[step=1.0,thin,white,opacity=0.4] (-7,-7) grid (9,-1);
      \node at (-7+0.5,-7+0.5) {\footnotesize$1$};
      %\node at (-5,3) {\large$H^2_{(x,y)}(K[x,y])$};
 %                 \node at (1,-8){Kill all monomials of the form $m\cdot [F]_1$, $m\cdot [x^2F]_2$, and $m\cdot [x^4]_4$.};
            \foreach \p in {      (0,5),(1,5),(2,5),(3,5),(4,5),(5,5),(6,5),(7,5),(8,5),(9,5),(10,5),(11,5),(12,5),(13,5),(14,5),(15,5),(16,5),(0,4),(1,4),(2,4),(3,4),(4,4),(5,4),(6,4),(7,4),(8,4),(9,4),(10,4),(11,4),(12,4),(13,4),(14,4),(15,4),(16,4),(0,3),(1,3),(2,3),(3,3),(4,3),(5,3),(6,3),(7,3),(0,3),(1,3),(2,3),(3,3),(4,3),(5,3),(6,3),(7,3),(8,3),(9,3),(10,3),(11,3),(12,3),(13,3),(14,3),(15,3),(0,2),(1,2),(2,2),(3,2),(4,2),(5,2),(6,2),(7,2),(0,1),(1,1),(2,1),(3,1),(4,0),(5,0),(6,0),(7,0),(8,0),(9,0),(10,0),(11,0),(12,0),(13,0),(14,0),(15,0),(16,0),(8,1),(9,1),(10,1),(11,1),(12,1),(13,1),(14,1),(15,1),(16,1),(16,2)}{
        \parsept{\x}{\y}{\p};
        \fill[black,opacity=1.0] (-7+\x,-7+\y)--(-6+\x,-7+\y)--(-6+\x,-6+\y)--(-7+\x,-6+\y)--(-7+\x,-7+\y);
      };

      \foreach \p in {(0,0)}{
        \parsept{\x}{\y}{\p};
        \fill[green,opacity=0.4] (-7+\x,-7+\y)--(-6+\x,-7+\y)--(-6+\x,-6+\y)--(-7+\x,-6+\y)--(-7+\x,-7+\y);
        };
      \foreach \p in {(0,1),(1,0)}{
        \parsept{\x}{\y}{\p};
        \fill[teal,opacity=0.4] (-7+\x,-7+\y)--(-6+\x,-7+\y)--(-6+\x,-6+\y)--(-7+\x,-6+\y)--(-7+\x,-7+\y);
        };
      \foreach \p in {(0,2),(1,1),(2,1),(2,0)}{
        \parsept{\x}{\y}{\p};
        \fill[red,opacity=0.4] (-7+\x,-7+\y)--(-6+\x,-7+\y)--(-6+\x,-6+\y)--(-7+\x,-6+\y)--(-7+\x,-7+\y);
        };
      \foreach \p in {(0,3),(1,2),(2,2),(4,2),(3,1),(4,1),(3,0)}{
        \parsept{\x}{\y}{\p};
        \fill[magenta,opacity=0.4] (-7+\x,-7+\y)--(-6+\x,-7+\y)--(-6+\x,-6+\y)--(-7+\x,-6+\y)--(-7+\x,-7+\y);
        };
      \foreach \p in {(0,4),(1,3),(2,3),(4,3),(8,3),(3,2),(5,2),(6,2),(8,2),(5,1),(6,1),(4,0)}{
        \parsept{\x}{\y}{\p};
        \fill[blue,opacity=0.4] (-7+\x,-7+\y)--(-6+\x,-7+\y)--(-6+\x,-6+\y)--(-7+\x,-6+\y)--(-7+\x,-7+\y);
        };
      \foreach \p in {(0,5),(1,4),(2,4),(4,4),(8,4),(16,4),(3,3),(5,3),(6,3),(9,3),(10,3),(12,3),(16,3),(7,2),(9,2),(10,2),(12,2),(7,1),(8,1),(5,0)}{
        \parsept{\x}{\y}{\p};
        \fill[Dandelion,opacity=0.4] (-7+\x,-7+\y)--(-6+\x,-7+\y)--(-6+\x,-6+\y)--(-7+\x,-6+\y)--(-7+\x,-7+\y);
      };

            \foreach \x in {1,2,...,16}{
        \foreach \y in {1,2,...,5}{
          \node at (-7+\x+0.5,-7+\y+0.5) {\footnotesize$x^{\x}F^{\y}$};
          %\node at (0-\x+\y,6-\x-\y) {\footnotesize$\left[\tfrac{1}{x^{\x}y^{\y}}\right]$};
        };
      };

      \foreach \x in {1,2,...,16}{
        \node at (-7+\x+0.5,-7+0.5) {\footnotesize$x^{\x}$};
          %\node at (1-\x,5-\x) {\footnotesize$\left[\tfrac{1}{x^{\x}y}\right]$};
      };
      \foreach \x in {1,2,...,5}{
        \node at (-7+0.5,-7+\x+0.5) {\footnotesize$F^{\x}$};
%          \node at (-1+\x,5-\x) {\footnotesize$\left[\tfrac{1}{xy^{\x}}\right]$};
      };     
      \foreach \x in {0,1,...,6}{
        %\draw[thick] (-7,-2-2*\x)--(0+\x,5-\x);
        %\draw[ thick] (7,-2-2*\x)--(0-\x,5-\x);
        \draw[thick] (-7,-7+\x)--(10,-7+\x);
      };
      \foreach \x in {0,1,...,17}{
        \draw[thick] (-7+\x,-7)--(-7+\x,-1);
      };
    \end{tikzpicture}    
  \end{center}
\end{ex}

\begin{ex}
  Let $R=\F_2[x]$. The local cohomology module $H^1_x(R)$ admits a presentation over $\FF$ as $\FF/\FF(x,xF-1)$. A quick application of Buchberger's algorithm confirms that the associated graded $\A/I$ has $\text{in}(I)=\A(x,xf)$. This ideal has a graded free resolution over $\A$ of the form
  \[
      0\to \A(-3)\to \A(-1)\oplus \A(-2)\to \A\to \A/\inn(I)\to 0
      \]
      so that
      \[
      \HS_{\gr\left(H^1_x(R)\right)}(t) = \frac{1-t-t^2+t^3}{(1-t)(1-t-t^2)}=\frac{1-t^2}{(1-t-t^2)}=1+t+t^2+2t^3+3t^4+5t^6+8t^7+\cdots
      \]
      We see that $H^1_x(R)$ has Bernstein dimension $0$ (so it is $F$-holonomic) and multiplicity $0$.
      \vspace{0.5em}
      
  \begin{center}
    \begin{tikzpicture}[scale=0.9, every node/.style={scale=1.0}]
\makeatletter
%      \clip (-7.2,-7.2) rectangle (6.8,6);
%\node at (1,0) {\begin{tabular}{|c||c|c|c|c|c|c|c|c|c|}\hline cx & 0 & 1 & 2 & 3 & 4 & 5 & 6 & 7 & 8\\\hline
%                        mon.ct. & 1 & 1 & 1 & 2 & 3 & 5 & 8 & 13 & 21\\\hline\end{tabular}};
      \draw[step=1.0,thin,white,opacity=0.4] (-7,-7) grid (9,-1);
      \node at (-7+0.5,-7+0.5) {\footnotesize$1$};
      %\node at (-5,3) {\large$H^2_{(x,y)}(K[x,y])$};
 %                 \node at (1,-8){Kill all monomials of the form $m\cdot [x]_1$ and $m\cdot [xF]_2$.};
     \foreach \p in {(1,0),(2,0),(3,0),(4,0),(5,0),(6,0),(7,0),(8,0),(9,0),(10,0),(11,0),(12,0),(13,0),(14,0),(15,0),(16,0),(2,1),(3,1),(4,1),(5,1),(6,1),(7,1),(8,1),(9,1),(10,1),(11,1),(12,1),(13,1),(14,1),(15,1),(16,1),(4,2),(5,2),(6,2),(7,2),(8,2),(9,2),(10,2),(11,2),(12,2),(13,2),(14,2),(15,2),(16,2),(8,3),(9,3),(10,3),(11,3),(12,3),(13,3),(14,3),(15,3),(16,3),(16,4)}{
        \parsept{\x}{\y}{\p};
        \fill[black,opacity=1.0] (-7+\x,-7+\y)--(-6+\x,-7+\y)--(-6+\x,-6+\y)--(-7+\x,-6+\y)--(-7+\x,-7+\y);
      };
      \foreach \p in {(1,1),(2,2),(3,2),(4,3),(5,3),(6,3),(7,3),(8,4),(9,4),(10,4),(11,4),(12,4),(13,4),(14,4),(15,4),(16,5)}{
        \parsept{\x}{\y}{\p};
        \fill[black,opacity=1.0] (-7+\x,-7+\y)--(-6+\x,-7+\y)--(-6+\x,-6+\y)--(-7+\x,-6+\y)--(-7+\x,-7+\y);
      };

      \foreach \p in {(0,0)}{
        \parsept{\x}{\y}{\p};
        \fill[green,opacity=0.4] (-7+\x,-7+\y)--(-6+\x,-7+\y)--(-6+\x,-6+\y)--(-7+\x,-6+\y)--(-7+\x,-7+\y);
        };
      \foreach \p in {(0,1),(1,0)}{
        \parsept{\x}{\y}{\p};
        \fill[teal,opacity=0.4] (-7+\x,-7+\y)--(-6+\x,-7+\y)--(-6+\x,-6+\y)--(-7+\x,-6+\y)--(-7+\x,-7+\y);
        };
      \foreach \p in {(0,2),(1,1),(2,1),(2,0)}{
        \parsept{\x}{\y}{\p};
        \fill[red,opacity=0.4] (-7+\x,-7+\y)--(-6+\x,-7+\y)--(-6+\x,-6+\y)--(-7+\x,-6+\y)--(-7+\x,-7+\y);
        };
      \foreach \p in {(0,3),(1,2),(2,2),(4,2),(3,1),(4,1),(3,0)}{
        \parsept{\x}{\y}{\p};
        \fill[magenta,opacity=0.4] (-7+\x,-7+\y)--(-6+\x,-7+\y)--(-6+\x,-6+\y)--(-7+\x,-6+\y)--(-7+\x,-7+\y);
        };
      \foreach \p in {(0,4),(1,3),(2,3),(4,3),(8,3),(3,2),(5,2),(6,2),(8,2),(5,1),(6,1),(4,0)}{
        \parsept{\x}{\y}{\p};
        \fill[blue,opacity=0.4] (-7+\x,-7+\y)--(-6+\x,-7+\y)--(-6+\x,-6+\y)--(-7+\x,-6+\y)--(-7+\x,-7+\y);
        };
      \foreach \p in {(0,5),(1,4),(2,4),(4,4),(8,4),(16,4),(3,3),(5,3),(6,3),(9,3),(10,3),(12,3),(16,3),(7,2),(9,2),(10,2),(12,2),(7,1),(8,1),(5,0)}{
        \parsept{\x}{\y}{\p};
        \fill[Dandelion,opacity=0.4] (-7+\x,-7+\y)--(-6+\x,-7+\y)--(-6+\x,-6+\y)--(-7+\x,-6+\y)--(-7+\x,-7+\y);
      };

            \foreach \x in {1,2,...,16}{
        \foreach \y in {1,2,...,5}{
          \node at (-7+\x+0.5,-7+\y+0.5) {\footnotesize$x^{\x}F^{\y}$};
          %\node at (0-\x+\y,6-\x-\y) {\footnotesize$\left[\tfrac{1}{x^{\x}y^{\y}}\right]$};
        };
      };

      \foreach \x in {1,2,...,16}{
        \node at (-7+\x+0.5,-7+0.5) {\footnotesize$x^{\x}$};
          %\node at (1-\x,5-\x) {\footnotesize$\left[\tfrac{1}{x^{\x}y}\right]$};
      };
      \foreach \x in {1,2,...,5}{
        \node at (-7+0.5,-7+\x+0.5) {\footnotesize$F^{\x}$};
%          \node at (-1+\x,5-\x) {\footnotesize$\left[\tfrac{1}{xy^{\x}}\right]$};
      };     
      \foreach \x in {0,1,...,6}{
        %\draw[thick] (-7,-2-2*\x)--(0+\x,5-\x);
        %\draw[ thick] (7,-2-2*\x)--(0-\x,5-\x);
        \draw[thick] (-7,-7+\x)--(10,-7+\x);
      };
      \foreach \x in {0,1,...,17}{
        \draw[thick] (-7+\x,-7)--(-7+\x,-1);
      };
    \end{tikzpicture}    
  \end{center}
\end{ex}
\vspace{0.5em}
\begin{ex}
Let $R=\F_2[x]$. As an $\FF$-module, the localization $R_x$ has the presentation $R_x = \FF/\FF(xF-1)$. Buchberger's algorithm shows that the associated graded $\A/K$ with $\text{in}(K)=\A(xf,fx,fx^2,x^5)$. The ideal $\inn(K)$ has a graded free resolution of the form
\[
      0\to \A(-3)\oplus\A(-4)\oplus \A(-5)\to \A(-2)\oplus \A(-2)\oplus\A(-3)\oplus \A(-5)\to \A\to \A/\inn(K)\to 0
      \]
      giving
      \[
      \HS_{\gr(R_x)}(t) = \frac{1-2t^2+t^4}{(1-t)(1-t-t^2)}=\frac{1+t-t^2-t^3}{(1-t-t^2)}=1+2t+2t^2+3t^3+5t^4+8t^5+13t^6+\cdots
      \]
      We see that $R_x$ has Bernstein dimension $0$ (so it is $F$-holonomic) and multiplicity $0$. We also see that, since $R\hookrightarrow R_x$ maps the generator of $R$ onto the element $x\cdot (1/x)$ of degree $1$ in the filtration $\Omega_i=\B_i\cdot \{1/x\}$ on $R_x$, we have
      \[
      \HS_{\gr(R_x)}(t)-t\cdot \HS_{\gr(R)}(t)=\HS_{\gr\left(H^1_x(R)\right)}(t)
      \]
      \vspace{0.5em}
\begin{center}
   % \vspace{-6.0em}
    \begin{tikzpicture}[scale=0.9, every node/.style={scale=1.0}]
\makeatletter
%      \clip (-7.2,-7.2) rectangle (6.8,6);

      \draw[step=1.0,thin,white,opacity=0.4] (-7,-7) grid (9,-1);
      \node at (-7+0.5,-7+0.5) {\footnotesize$1$};
      %\node at (-5,3) {\large$H^2_{(x,y)}(K[x,y])$};
      %                  \node at (1,-8){Kill all monomials of the form $m\cdot [F]_1$, $m\cdot [x^2F]_2$, and $m\cdot [x^4]_4$.};
%            \node at (1,-8){Kill all monomials of the form $m\cdot [xF]_2$, $m\cdot [x^2F]_2$, $m\cdot [x^4F]_3$, and $m\cdot [x^5]_5$.};
                  \foreach \p in {      (16,5),
                    (8,4),(9,4),(10,4),(11,4),(12,4),(13,4),(14,4),(15,4),(16,4),
                    (4,3),(5,3),(6,3),(7,3),(8,3),(9,3),(10,3),(11,3),(12,3),(13,3),(14,3),(15,3),(16,3),
                    (2,2),(3,2),(4,2),(5,2),(6,2),(7,2),(8,2),(9,2),(10,2),(11,2),
                    (1,1),(2,1),(3,1),(4,1),(5,1),(10,1),(11,1),(12,1),(13,1),(14,1),(15,1),(16,1),
                    (5,0),(6,0),(7,0),(8,0),(9,0),(10,0),(11,0),(12,0),(13,0),(14,0),(15,0),(16,0)}
 {
        \parsept{\x}{\y}{\p};
        \fill[black,opacity=1.0] (-7+\x,-7+\y)--(-6+\x,-7+\y)--(-6+\x,-6+\y)--(-7+\x,-6+\y)--(-7+\x,-7+\y);
      };

      \foreach \p in {(0,0)}{
        \parsept{\x}{\y}{\p};
        \fill[green,opacity=0.4] (-7+\x,-7+\y)--(-6+\x,-7+\y)--(-6+\x,-6+\y)--(-7+\x,-6+\y)--(-7+\x,-7+\y);
        };
      \foreach \p in {(0,1),(1,0)}{
        \parsept{\x}{\y}{\p};
        \fill[teal,opacity=0.4] (-7+\x,-7+\y)--(-6+\x,-7+\y)--(-6+\x,-6+\y)--(-7+\x,-6+\y)--(-7+\x,-7+\y);
        };
      \foreach \p in {(0,2),(1,1),(2,1),(2,0)}{
        \parsept{\x}{\y}{\p};
        \fill[red,opacity=0.4] (-7+\x,-7+\y)--(-6+\x,-7+\y)--(-6+\x,-6+\y)--(-7+\x,-6+\y)--(-7+\x,-7+\y);
        };
      \foreach \p in {(0,3),(1,2),(2,2),(4,2),(3,1),(4,1),(3,0)}{
        \parsept{\x}{\y}{\p};
        \fill[magenta,opacity=0.4] (-7+\x,-7+\y)--(-6+\x,-7+\y)--(-6+\x,-6+\y)--(-7+\x,-6+\y)--(-7+\x,-7+\y);
        };
      \foreach \p in {(0,4),(1,3),(2,3),(4,3),(8,3),(3,2),(5,2),(6,2),(8,2),(5,1),(6,1),(4,0)}{
        \parsept{\x}{\y}{\p};
        \fill[blue,opacity=0.4] (-7+\x,-7+\y)--(-6+\x,-7+\y)--(-6+\x,-6+\y)--(-7+\x,-6+\y)--(-7+\x,-7+\y);
        };
      \foreach \p in {(0,5),(1,4),(2,4),(4,4),(8,4),(16,4),(3,3),(5,3),(6,3),(9,3),(10,3),(12,3),(16,3),(7,2),(9,2),(10,2),(12,2),(7,1),(8,1),(5,0)}{
        \parsept{\x}{\y}{\p};
        \fill[Dandelion,opacity=0.4] (-7+\x,-7+\y)--(-6+\x,-7+\y)--(-6+\x,-6+\y)--(-7+\x,-6+\y)--(-7+\x,-7+\y);
      };

            \foreach \x in {1,2,...,16}{
        \foreach \y in {1,2,...,5}{
          \node at (-7+\x+0.5,-7+\y+0.5) {\footnotesize$x^{\x}F^{\y}$};
          %\node at (0-\x+\y,6-\x-\y) {\footnotesize$\left[\tfrac{1}{x^{\x}y^{\y}}\right]$};
        };
      };

      \foreach \x in {1,2,...,16}{
        \node at (-7+\x+0.5,-7+0.5) {\footnotesize$x^{\x}$};
          %\node at (1-\x,5-\x) {\footnotesize$\left[\tfrac{1}{x^{\x}y}\right]$};
      };
      \foreach \x in {1,2,...,5}{
        \node at (-7+0.5,-7+\x+0.5) {\footnotesize$F^{\x}$};
%          \node at (-1+\x,5-\x) {\footnotesize$\left[\tfrac{1}{xy^{\x}}\right]$};
      };     
      \foreach \x in {0,1,...,6}{
        %\draw[thick] (-7,-2-2*\x)--(0+\x,5-\x);
        %\draw[ thick] (7,-2-2*\x)--(0-\x,5-\x);
        \draw[thick] (-7,-7+\x)--(10,-7+\x);
      };
      \foreach \x in {0,1,...,17}{
        \draw[thick] (-7+\x,-7)--(-7+\x,-1);
      };
    \end{tikzpicture}
    \end{center}
\end{ex}

%\aphy
\printbibliography

\end{document}